\documentclass[10pt]{amsart}
\usepackage{graphicx} 
\usepackage{epic,eepic}
\usepackage[mathscr]{eucal} 
\numberwithin{equation}{section}
\usepackage{amsthm}
\theoremstyle{plain}
\newtheorem{theorem}{Theorem}[section]
\newtheorem{lemma}[theorem]{Lemma}
\newtheorem{proposition}[theorem]{Proposition}
\newtheorem{corollary}[theorem]{Corollary}

\theoremstyle{definition}
\newtheorem{definition}[theorem]{Definition}
\newtheorem{example}[theorem]{Example}
\newtheorem{remark}[theorem]{Remark}

\title[Periodicities of T and Y-systems I: Type $B_r$]
{Periodicities of T and Y-systems,\\
dilogarithm identities,
and cluster algebras I:\\ Type $B_r$
}

\author[R.\  Inoue]{Rei Inoue}
\address{ R.\  Inoue:
Faculty of Pharmaceutical Sciences, Suzuka University of Medical Science,
Suzuka, 513-8670, Japan}

\author[O.\ Iyama]{Osamu Iyama}
\address{ O.\ Iyama:
 Graduate School of Mathematics, Nagoya University,
Nagoya, 464-8604, Japan}

\author[B.\ Keller]{Bernhard Keller}
\address{B.\ Keller:
Universit\'e Paris Diderot -- Paris 7,
UFR de Math\'ematiques,
Institut de Math\'ematiques de Jussieu, UMR 7586 du CNRS,
Case 7012,
2, place Jussieu,
75251 Paris Cedex 05,
France}

\author[A.\ Kuniba]{\\Atsuo Kuniba}
\address{ A.\ Kuniba:
Institute of Physics,
University of Tokyo,
Tokyo, 153-8902, Japan}

\author[T.\ Nakanishi]{Tomoki Nakanishi}
\address{ T.\ Nakanishi:
 Graduate School of Mathematics, Nagoya University,
Nagoya, 464-8604, Japan}

\begin{document}


\begin{abstract}
We prove the
 periodicities
of the restricted T and Y-systems associated with
the quantum affine algebra of type $B_r$
at any level.
We also prove the dilogarithm identities
for the Y-systems of type $B_r$ at
any level.
Our proof is based on
the tropical Y-systems
and the categorification of
the cluster algebra associated with 
any skew-symmetric matrix by Plamondon.
Using this new method,
we also give an alternative and simplified proof
of the 
periodicities
of the  T and Y-systems associated with
pairs of simply laced Dynkin diagrams.
\end{abstract}

\maketitle

\tableofcontents

\section{Main results}

\subsection{Background}
The T and Y-systems are systems of algebraic 
relations originally associated with {\em quantum affine algebras}
\cite{KNS2,Nkj,Her1},
or more generally, with the quantum affinizations
of a wide class of quantum Kac-Moody algebras \cite{Her2,KNS3}.

On the other hand,
these T and Y-systems also appear naturally in
{\em cluster algebras} \cite{FZ2,FZ3}.
This identification has provided several fruitful results.
The periodicities of Y-systems were proved by \cite{FZ2}
for any simply laced type at level 2 (in our terminology).
Here we mean by `simply laced'
the Y-systems associated with the
{\em quantum affine algebras of simply laced type}.
The periodicities of Y-systems
were further proved by \cite{Ke1,Ke2}
for any simply laced type at any level,
by the combination  with the {\em cluster category} method.
Using the method of \cite{Ke1,Ke2},
the periodicities of T-systems were also proved by \cite{IIKNS}
for any simply laced type at any level.
Closely related to the Y-systems,
the dilogarithm identities were proved by \cite{C}
for any simply laced type at level 2 based on
the result of \cite{FZ2},
and  further proved by \cite{Nkn}
for any simply laced type and any level.
So far, however, all these systematic treatments were limited to
the simply laced case only,
since the above methods are not straightforwardly applicable
to the nonsimply laced case.

In this paper and the subsequent one \cite{IIKKN},
we prove
the periodicities of T and Y-systems,
and also the dilogarithm identities,
in the {\em nonsimply laced} case
using the cluster algebra/cluster category method with suitable
modifications to the method used in the simply laced case.
We remark that the nonsimply laced systems here
are different from {\em another} class of nonsimply laced systems
arising from cluster algebras and studied in \cite{FZ2,FZ3}.

As is often the case in the nonsimply laced setting,
each type requires some nonuniform, `customized'  treatment.
So, in this paper, we concentrate on 
type $B_r$ and highlight the
underlying common method.
Then, separately in \cite{IIKKN},
types $C_r$, $F_4$, and $G_2$ will be
treated  with emphasis on the special features of each case.

\subsection{Restricted T and Y-systems of type $B_r$}
\label{subsect:restrictedT}

Let $B_r$ be the Dynkin diagram of type $B$
with rank $r$,
and $I=\{1,\dots, r\}$ be the enumeration
of the vertices of $B_r$ as below.
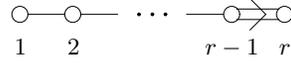
\begin{align*}
\begin{picture}(90,25)(0,-15)
%
%
\put(0,0){
\put(0,0){\circle{6}}
\put(20,0){\circle{6}}
\put(80,0){\circle{6}}
\put(100,0){\circle{6}}
\put(45,0){\circle*{1}}
\put(50,0){\circle*{1}}
\put(55,0){\circle*{1}}
\drawline(3,0)(17,0)
\drawline(23,0)(37,0)
\drawline(63,0)(77,0)
\drawline(82,-2)(98,-2)
\drawline(82,2)(98,2)
\drawline(87,6)(93,0)
\drawline(87,-6)(93,0)
\put(-2,-15){\small $1$}
\put(18,-15){\small $2$}
\put(70,-15){\small $ r-1$}
\put(98,-15){\small $r$}
}
\end{picture}
\end{align*}
Let $h=2r$ and $h^{\vee}=2r-1$ be 
the Coxeter number and the dual Coxeter number of $B_r$,
respectively.
We set  numbers $t_a$ ($a\in I$) by
\begin{align}
\label{eq:t1}
t_a=
\begin{cases}
1 & a=1,\dots,r-1,\\
2 & a=r.
\end{cases}
\end{align}

For a given integer $\ell \geq 2$,
we introduce a set of triplets $(a,m,u)$,
\begin{align}
\mathcal{I}_{\ell}=\mathcal{I}_{\ell}(B_r):=
\{(a,m,u)\mid
a\in I ; m=1,\dots,t_a\ell-1;
u\in \frac{1}{2}\mathbb{Z}
\}.
\end{align}

\begin{definition}[\cite{KNS2}]
\label{defn:RT}
Fix an integer $\ell \geq 2$.
The {\it level $\ell$ restricted T-system $\mathbb{T}_{\ell}(B_r)$
of type $B_r$
(with the unit boundary condition)}
is the following system of relations
\eqref{eq:TB1}
for
a family of variables $T_{\ell}=\{T^{(a)}_m(u)
\mid (a,m,u)\in \mathcal{I}_{\ell}
\}$,
where 
$T^{(0)}_m (u)=T^{(a)}_0 (u)=1$,
and furthermore,  $T^{(a)}_{t_a\ell}(u)=1$
(the {\em unit boundary condition\/}) if they occur
in the right hand sides in the relations:

(Here and throughout the paper,
$2m$ (resp.\ $2m+1$) in the left hand sides, for example,
represents elements  $ 2,4,\dots$
(resp.\ $1,3,\dots$).)

\begin{align}
\label{eq:TB1}
\begin{split}
T^{(a)}_m(u-1)T^{(a)}_m(u+1)
&=
T^{(a)}_{m-1}(u)T^{(a)}_{m+1}(u)
+T^{(a-1)}_{m}(u)T^{(a+1)}_{m}(u)\\
&\hskip120pt
 (1\leq a\leq r-2),\\
T^{(r-1)}_m(u-1)T^{(r-1)}_m(u+1)
&=
T^{(r-1)}_{m-1}(u)T^{(r-1)}_{m+1}(u)
+
T^{(r-2)}_{m}(u)T^{(r)}_{2m}(u),\\
T^{(r)}_{2m}\left(u-\textstyle\frac{1}{2}\right)
T^{(r)}_{2m}\left(u+\textstyle\frac{1}{2}\right)
&=
T^{(r)}_{2m-1}(u)T^{(r)}_{2m+1}(u)\\
&\qquad
+
T^{(r-1)}_{m}\left(u-\textstyle\frac{1}{2}\right)
T^{(r-1)}_{m}\left(u+\textstyle\frac{1}{2}\right),
\\
T^{(r)}_{2m+1}\left(u-\textstyle\frac{1}{2}\right)
T^{(r)}_{2m+1}\left(u+\textstyle\frac{1}{2}\right)
&=
T^{(r)}_{2m}(u)T^{(r)}_{2m+2}(u)
+
T^{(r-1)}_{m}(u)T^{(r-1)}_{m+1}(u).
\end{split}
\end{align}
\end{definition}

\begin{definition}[\cite{KN}]
\label{defn:RY}
Fix an integer $\ell \geq 2$.
The {\it level $\ell$ restricted Y-system $\mathbb{Y}_{\ell}(B_r)$
of type $B_r$}
is the following system of relations
\eqref{eq:YB1} for a family of variables $Y_{\ell}=\{Y^{(a)}_m(u)
\mid
(a,m,u)\in \mathcal{I}_{\ell}
\}$,
where 
$Y^{(0)}_m (u)=Y^{(a)}_0 (u)^{-1}
=Y^{(a)}_{t_a\ell}(u)^{-1}=0$
 if they occur
in the right hand sides in the relations:
\begin{align}
\label{eq:YB1}
\begin{split}
Y^{(a)}_m(u-1)Y^{(a)}_m(u+1)
&=
\frac{
(1+Y^{(a-1)}_{m}(u))(1+Y^{(a+1)}_{m}(u))
}
{
(1+Y^{(a)}_{m-1}(u)^{-1})(1+Y^{(a)}_{m+1}(u)^{-1})
}
\\
&\hskip30pt (1\leq a\leq r-2),\\
Y^{(r-1)}_m(u-1)Y^{(r-1)}_m(u+1)
&=
\frac{
\begin{array}{l}
\textstyle
(1+Y^{(r-2)}_{m}(u))
(1+Y^{(r)}_{2m-1}(u))
(1+Y^{(r)}_{2m+1}(u))
\\
\textstyle
\quad\times
(1+Y^{(r)}_{2m}\left(u-\frac{1}{2}\right))
(1+Y^{(r)}_{2m}\left(u+\frac{1}{2}\right))
\end{array}
}
{
(1+Y^{(r-1)}_{m-1}(u)^{-1})(1+Y^{(r-1)}_{m+1}(u)^{-1})
},
\\
Y^{(r)}_{2m}\left(u-\textstyle\frac{1}{2}\right)
Y^{(r)}_{2m}\left(u+\textstyle\frac{1}{2}\right)
&=
\frac{1+Y^{(r-1)}_{m}(u)}
{
(1+Y^{(r)}_{2m-1}(u)^{-1})(1+Y^{(r)}_{2m+1}(u)^{-1})
},\\
Y^{(r)}_{2m+1}\left(u-\textstyle\frac{1}{2}\right)
Y^{(r)}_{2m+1}\left(u+\textstyle\frac{1}{2}\right)
&=
\frac{1}{(1+Y^{(r)}_{2m}(u)^{-1})(1+Y^{(r)}_{2m+2}(u)^{-1})}.
\end{split}
\end{align}
\end{definition}

Let us write \eqref{eq:TB1} in a unified manner
\begin{align}
\label{eq:Tu}
T^{(a)}_{m}\left(u-\textstyle\frac{1}{t_a}\right)
T^{(a)}_{m}\left(u+\textstyle\frac{1}{t_a}\right)
&=
T^{(a)}_{m-1}(u)T^{(a)}_{m+1}(u)
+
\prod_{(b,k,v)\in \mathcal{I}_{\ell}}
T^{(b)}_{k}(v)^{G(b,k,v;a,m,u)}.
\end{align}
Define the transposition
$^{t}G(b,k,v;a,m,u)=G(a,m,u;b,k,v)$.
Then, we have
\begin{align}
\label{eq:Yu}
Y^{(a)}_{m}\left(u-\textstyle\frac{1}{t_a}\right)
Y^{(a)}_{m}\left(u+\textstyle\frac{1}{t_a}\right)
&=
\frac{
\displaystyle
\prod_{(b,k,v)\in \mathcal{I}_{\ell}}
(1+Y^{(b)}_{k}(v))^{{}^t\!G(b,k,v;a,m,u)}
}
{
(1+Y^{(a)}_{m-1}(u)^{-1})(1+Y^{(a)}_{m+1}(u)^{-1})
}.
\end{align}

See \cite{IIKNS,KNS3} and references therein on
the background of these systems.

\subsection{Periodicities}


\begin{definition}
\label{defn:RT2}
Let  $\EuScript{T}_{\ell}(B_r)$
be the commutative ring over $\mathbb{Z}$ with identity
element, with generators
$T^{(a)}_m(u)^{\pm 1}$
 ($(a,m,u)\in \mathcal{I}_{\ell}$)
and relations $\mathbb{T}_{\ell}(B_r)$
together with $T^{(a)}_m(u)T^{(a)}_m(u)^{-1}=1$.
Let $\EuScript{T}^{\circ}_{\ell}(B_r)$ be
 the subring of $\EuScript{T}_{\ell}(B_r)$
generated by 
$T^{(a)}_m(u)$ 
 ($(a,m,u)\in \mathcal{I}_{\ell}$).
\end{definition}

A {\em semifield\/} $(\mathbb{P},\oplus)$ is an
abelian multiplicative group $\mathbb{P}$ endowed with a binary
operation of addition $\oplus$ which is commutative,
associative, and distributive with respect to the
multiplication in $\mathbb{P}$ \cite{FZ3,HW}. 

\begin{definition}
\label{def:YB}
Let $\EuScript{Y}_{\ell}(B_r)$
be the semifield with generators
$Y^{(a)}_m(u)$
 $((a,m,u)\in \mathcal{I}_{\ell})$
and relations $\mathbb{Y}_{\ell}(B_r)$.
Let $\EuScript{Y}^{\circ}_{\ell}(B_r)$
be the multiplicative subgroup
of $\EuScript{Y}_{\ell}(B_r)$
generated by
$Y^{(a)}_m(u)$, $1+Y^{(a)}_m(u)$
 ($(a,m,u)\in \mathcal{I}_{\ell}$).
(Here we use the symbol $+$ instead of $\oplus$ 
for simplicity.)
\end{definition}

The first main result of the paper is the periodicities
of the T and Y-systems.

\begin{theorem}[Conjectured by \cite{IIKNS}]
\label{thm:Tperiod}
The following relations hold in
$\EuScript{T}^{\circ}_{\ell}(B_r)$.
\par
(i) Half periodicity: $T^{(a)}_m(u+h^{\vee}+\ell)
=T^{(a)}_{t_a \ell -m}(u)$.
\par
(ii) Full periodicity: $T^{(a)}_m(u+2(h^{\vee}+\ell))
=T^{(a)}_m(u)$.
\end{theorem}

\begin{theorem}[Conjectured by \cite{KNS2}]
\label{thm:Yperiod}
The following relations hold in
$\EuScript{Y}^{\circ}_{\ell}(B_r)$.
\par
(i) Half periodicity: $Y^{(a)}_m(u+h^{\vee}+\ell)
=Y^{(a)}_{t_a \ell -m}(u)$.
\par
(ii) Full periodicity: $Y^{(a)}_m(u+2(h^{\vee}+\ell))
=Y^{(a)}_m(u)$.
\end{theorem}

\subsection{Dilogarithm identities}
Let $L(x)$ be the {\em Rogers dilogarithm function\/}
\cite{L,Ki2,Zag}
\begin{align}
\label{eq:L0}
L(x)=-\frac{1}{2}\int_{0}^x 
\left\{ \frac{\log(1-y)}{y}+
\frac{\log y}{1-y}
\right\} dy
\quad (0\leq x\leq 1).
\end{align}
It is well known that
\begin{gather}
\label{eq:L1}
L(0)=0,
\quad L(1)=\frac{\pi^2}{6},\\
\label{eq:L2}
\quad L(x) + L(1-x)=\frac{\pi^2}{6}
\quad (0\leq x\leq 1).
\end{gather}

We introduce the {\em constant version\/} of the Y-system.

\begin{definition}
\label{defn:RYc}
Fix an integer $\ell \geq 2$.
The {\it level $\ell$ restricted constant Y-system
 $\mathbb{Y}^{\mathrm{c}}_{\ell}(B_r)$
of type $B_r$}
is the following system of relations
\eqref{eq:YBc} for a family of variables $Y^{\mathrm{c}}_{\ell}=\{Y^{(a)}_m
\mid
a\in I; m=1,\dots,t_a\ell-1 \}$,
where 
$Y^{(0)}_m =Y^{(a)}_0{}^{-1}=
Y^{(a)}_{t_a\ell}{}^{-1}=0$
 if they occur
in the right hand sides in the relations:
\begin{align}
\label{eq:YBc}
\begin{split}
(Y^{(a)}_m)^2
&=
\frac{
(1+Y^{(a-1)}_{m})(1+Y^{(a+1)}_{m})
}
{
(1+Y^{(a)}_{m-1}{}^{-1})(1+Y^{(a)}_{m+1}{}^{-1})
}\quad (1\leq a\leq r-2),\\
(Y^{(r-1)}_m)^2
&=
\frac{
\begin{array}{l}
\textstyle
(1+Y^{(r-2)}_{m})
(1+Y^{(r)}_{2m-1})(1+Y^{(r)}_{2m})^2
(1+Y^{(r)}_{2m+1})
\end{array}
}
{
(1+Y^{(r-1)}_{m-1}{}^{-1})(1+Y^{(r-1)}_{m+1}{}^{-1})
},\\
(Y^{(r)}_{2m})^2
&=
\frac{1+Y^{(r-1)}_{m}}
{
(1+Y^{(r)}_{2m-1}{}^{-1})(1+Y^{(r)}_{2m+1}{}^{-1})
},\\
(Y^{(r)}_{2m+1})^2
&=
\frac{1}{(1+Y^{(r)}_{2m}{}^{-1})(1+Y^{(r)}_{2m+2}{}^{-1})}.
\end{split}
\end{align}
\end{definition}

\begin{proposition}
There exists a unique positive real solution
of  $\mathbb{Y}^{\mathrm{c}}_{\ell}(B_r)$.
\end{proposition}

\begin{proof}
Set $f^{(a)}_m=Y^{(a)}_m/{(1+Y^{(a)}_m)}$.
Then, \eqref{eq:YBc} is equivalent to the system
of equations \cite[Eq.~(B.28)]{KNS2}
\begin{align}
f^{(a)}_m = \prod_{(b,k)}(1-f^{(b)}_k)^{K^{mk}_{ab}},
\quad
K^{mk}_{ab}=(\alpha_a|\alpha_b)
(\min(t_bm,t_ak)-\frac{mk}{\ell}),
\end{align}
where $(\alpha_a|\alpha_b)$ is the invariant bilinear form
for the simple Lie algebra of type $B_r$
with normalization $(\alpha_a|\alpha_a)=2$ for a long root
$\alpha_a$.
By elementary transformations,
one can show that every principal minor of the matrix $K$
is positive. Therefore, $K$ is positive definite.
Also, it is clear that $K$ is symmetric.
Then, the theorem by \cite[Section 1]{NK} is applicable.
\end{proof}

The second main result of the paper is the dilogarithm
identities conjectured
by Kirillov \cite[Eq.~(7)]{Ki1}, properly corrected by
Kuniba \cite[Eqs.~(A.1a), (A.1c)]{Ku},

\begin{theorem}[Dilogarithm identities]
\label{thm:DI}
Suppose that a family of positive real numbers
$\{Y^{(a)}_m \mid a\in I; m=1,\dots,t_a\ell-1\}$
satisfies \eqref{eq:YBc}.
Then, we have the identity
\begin{align}
\label{eq:DI}
\frac{6}{\pi^2}\sum_{a\in I}
\sum_{m=1}^{t_a\ell-1}
L\left(\frac{Y^{(a)}_m}{1+Y^{(a)}_m}\right)
=
\frac{\ell \dim \mathfrak{g}}{h^{\vee}+\ell} - r,
\end{align}
where $\mathfrak{g}$ is
 the simple Lie algebra of type $B_r$.
\end{theorem}

The rational number of the first term in the right hand side
of \eqref{eq:DI} is the central charge of the {\em Wess-Zumino-Witten
 conformal
field theory\/} of type $B_r$ with level $\ell$.
The rational number in the right hand side of \eqref{eq:DI}
 itself
is also the central charge of the {\em parafermion conformal
field theory\/} of type $B_r$ with level $\ell$.
See \cite{KNS1,Ki2,Nah,Zag} for more about background
of \eqref{eq:DI}.

Due to the well-known
formula $\dim\mathfrak{g}=r(h+1)$, the right hand side of 
\eqref{eq:DI} is
 equal to the number
\begin{align}
\label{eq:c1}
\frac{ r(\ell h - h^{\vee})}{h^{\vee}+\ell}.
\end{align}

In fact, we prove a functional generalization of
Theorem \ref{thm:DI}, following the ideas of \cite{GT,FS,C,Nkn}.

\begin{theorem}[Functional dilogarithm identities]
\label{thm:DI2}
Suppose that 
a family of positive real numbers
 $\{Y^{(a)}_m (u)\mid (a,m,u)\in \mathcal{I}_{\ell} \}$
satisfies $\mathbb{Y}_{\ell}(B_r)$.
Then, we have the identities
\begin{align}\label{eq:DI2}
\frac{6}{\pi^2}
\sum_{
\genfrac{}{}{0pt}{1}
{
(a,m,u)\in \mathcal{I}_{\ell}
}
{
0\leq u < 2(h^{\vee}+\ell)
}
}
L\left(
\frac{Y^{(a)}_m(u)}{1+Y^{(a)}_m(u)}
\right)
&=
4r(\ell h-h^{\vee})=4r(2r\ell-2r+1),\\
\label{eq:DI3}
\frac{6}{\pi^2}
\sum_{
\genfrac{}{}{0pt}{1}
{
(a,m,u)\in \mathcal{I}_{\ell}
}
{
0\leq u < 2(h^{\vee}+\ell)
}
}
L\left(
\frac{1}{1+Y^{(a)}_m(u)}
\right)
&
=4\ell(r\ell+\ell-1).
\end{align}
\end{theorem}

The two identities
\eqref{eq:DI2} and \eqref{eq:DI3}  are equivalent
to each other due to \eqref{eq:L2},
since the sum of the right hand sides
is equal to $4(h^{\vee}+\ell)(r\ell+\ell-r)$,
which is the total number of $(a,m,u)\in I_{\ell}$
with $0\leq u < 2(h^{\vee}+\ell)$.

It is clear that Theorem \ref{thm:DI}
follows form Theorem \ref{thm:DI2}
by considering a {\em constant solution} $
Y^{(a)}_m=Y^{(a)}_m(u)$
of $\mathbb{Y}_{\ell}(B_r)$ in
the variable $u$.

\subsection{Outline of method and contents}

Let us briefly explain the idea of our proof
of the main results, Theorems \ref{thm:Tperiod},
\ref{thm:Yperiod}, and \ref{thm:DI2}.

To start up, we identify the T and Y-systems,
$\mathbb{T}_{\ell}(B_r)$ and $\mathbb{Y}_{\ell}(B_r)$
in Definitions \ref{defn:RT} and \ref{defn:RY}
as systems of relations for cluster variables and coefficients of
a cluster algebra, respectively.
This procedure is mostly parallel to the simply
laced case \cite{FZ2,FZ3,Ke1,DiK,HL,IIKNS,KNS3},
but necessarily more complicated.
For example,
unlike the simply laced ones,
the composite mutation which generates
the translation of the spectral parameter $u$ is
neither of {\em bipartite type} nor obviously
related  to
the {\em Coxeter element} of a certain Weyl group.
This is not a serious problem, though.
A real problem is that the arising
quiver $Q_{\ell}(B_r)$ for the cluster algebra,
which is seen in Figure \ref{fig:quiverB} in Section 2.3,
is not a familiar one in the representation theory of quivers;
In particular, we have no known or obvious periodicity result.
This is the main obstacle to a straightforward application
of the method in  \cite{Ke1,IIKNS,Ke2},
where the periodicities of the T and Y-systems
in the simply laced case
were derived from the periodicity in the
corresponding cluster category.

The key to bypass this obstacle is to consider
the {\em tropical Y-system}.
The tropical Y-system is the tropicalization of the
Y-system,
or more generally,
the exchange relations
of {\em the coefficients in the tropical semifield}
(called {\em principal coefficients} in \cite{FZ3})
 of a given cluster algebra.
In fact, it was already used 
by Fomin-Zelevinsky \cite{FZ2}
as a main tool in the proof of the
periodicity of the Y-systems in the simply laced  case at level 2.
In addition,  we make two crucial observations.

{\bf Observation 1.} {\em The periodicities
of cluster variables and coefficients follow from
the periodicity of principal coefficients.}

See Theorem \ref{thm:XXX} for a precise statement.
This claim was essentially conjectured by Fomin-Zelevinsky
\cite[Conjecture 4.7]{FZ3}. 
We prove the claim for the cluster algebra associated
with any {\em skew-symmetric matrix},
or equivalently, with any {\em quiver}.
To prove it, we use the recent result
by Plamondon \cite{Pl1,Pl2} on the  categorification
of the cluster algebra associated with an {\em arbitrary} quiver.
It is so far the most general formulation of
the categorification by  {\em 2-Calabi-Yau categories}
recently developed by various authors in particular cases (see \cite{A}
and the references therein).
Since each principal coefficient tuple carries the complete information
of the corresponding object in the category
 through {\em index},
the periodicity of principal coefficients
implies the same periodicity of objects in the category.
Therefore,  it also implies the same
periodicities of cluster variables and coefficients
by categorification.

{\bf Observation 2.} {\em
The tropical Y-system for $\mathbb{Y}_{\ell}(B_r)$
has a remarkable `factorization
property'
 so that its periodicity can be directly verified.}

Such a factorization  property of the tropical Y-system
 was first noticed by \cite{Nkn} in the proof of the
dilogarithm identities in the simply laced case.
It roughly means that the tropical Y-system at a higher level
splits into the level 2 pieces and the type $A$ pieces.
Moreover, each piece can be described in terms
of the piecewise-linear analogue of the simple reflections
of a certain Weyl group introduced by \cite{FZ2}.
Therefore, the periodicity is tractable.

Combining these two observations, we  obtain
the desired periodicities in  Theorems \ref{thm:Tperiod} and
\ref{thm:Yperiod}.

The tropical Y-system plays a central role
not only in the periodicity but also
in the dilogarithm identity.
The following observation was made in \cite{Nkn} in the simply laced case.

{\bf Observation 3.} {\em
The dilogarithm identity reduces to
the positivity/negativity property of the
tropical Y-system for $\mathbb{Y}_{\ell}(B_r)$.}

Shortly speaking, {\em ``the tropical Y-system knows everything''}
is our slogan.

The organization of the paper is as follows.
In Section 2, we introduce a quiver
$Q_{\ell}(B_r)$ and identify the  T and Y-systems
with systems of relations for cluster variables
and coefficients of the cluster algebra
associated with $Q_{\ell}(B_r)$
(Theorems \ref{thm:Tiso} and \ref{thm:Yiso}).
In Section 3, we study the tropical Y-system at level 2
and derive the periodicity
and the positivity/negativity property
 (Proposition \ref{prop:lev2}).
In Section 4, we study the tropical Y-system at higher levels
and show the factorization property.
As a result, we obtain
the periodicity  (Theorems \ref{thm:tYperiod})
and the positivity/negativity property
(Theorem \ref{thm:levhd}).
In Section 5, based on the result by Plamondon,
we present a general theorem stating Observation 1 above
 (Theorem \ref{thm:XXX})
for the cluster algebra associated with any skew-symmetric matrix.
As corollaries, we obtain the periodicities
of T and Y-systems in Theorem \ref{thm:Tperiod} and \ref{thm:Yperiod}.
In Section 6, using Theorem \ref{thm:levhd},
we prove the dilogarithm identity in Theorem \ref{thm:DI2}.
In Section 7, as a feedback of the newly introduced method,
we give an alternative and simplified proof
of the periodicities of the
 T and Y-systems associated with pairs of simply laced
Dynkin diagrams,
which were formerly proved by \cite{Ke1,IIKNS,Ke2}.

{\em Acknowledgement.} We  thank Pierre-Guy Plamondon
for making his result in \cite{Pl1,Pl2} available to us
prior to the publication.

\section{Cluster algebraic formulation}

The systems $\mathbb{T}_{\ell}(B_r)$
and $\mathbb{Y}_{\ell}(B_r)$
are naturally identified with
systems of relations for
cluster variables 
and coefficients, respectively, of a certain
cluster algebra.
They are mostly parallel to the simply
laced case \cite{FZ2,FZ3,Ke1,DiK,HL,IIKNS,KNS3},
but necessarily more complicated.

\subsection{Groceries for cluster algebras}

\label{subsec:groc}

Here we collect basic definitions for cluster algebras
to fix the convention and notation,
mainly following  \cite{FZ3}.
For further necessary definitions and information
for cluster algebras, see \cite{FZ3}.

Let $I$ be a finite index set throughout this subsection.

{\bf (i) Semifield.}
A {\em semifield\/} $(\mathbb{P},\oplus, \cdot)$
is an abelian multiplicative group
endowed with a binary
operation of addition $\oplus$ which is commutative,
associative, and distributive with respect to the
multiplication in $\mathbb{P}$.
The following three examples are relevant in this paper.

(a) {\em  Trivial semifield.}
The {\em trivial semifield\/}
$\mathbf{1}=\{\mathrm{1}\}$
consists of the multiplicative identity element $1$
with $1\oplus 1 = 1$.

\par
(b) {\em Universal semifield.}
For an $I$-tuple of variables $y=(y_i)_{i\in I}$,
the {\em universal semifield\/}
$\mathbb{Q}_{\mathrm{sf}}(y)$
consists of  all the rational functions
of the form
$P(y)/Q(y)$  (subtraction-free  rational expressions),
where $P(y)$ and $Q(y)$ are the nonzero polynomials
 in $y_i$'s with {\em nonnegative\/} integer coefficients.
The multiplication and the addition
are given by the usual ones of rational functions.

\par
(c) {\em Tropical semifield.}
For an $I$-tuple of variables $y=(y_i)_{i\in I}$,
the {\em tropical semifield\/}
$\mathrm{Trop}(y)$
is the abelian multiplicative group freely generated by
the variables  $y_i$'s endowed with the addition $\oplus$
\begin{align}
\label{eq:trop}
\prod_i y_i^{a_i}\oplus
\prod_i y_i^{b_i}
=
\prod_i y_i^{\min(a_i,b_i)}.
\end{align}

{\bf (ii)  Mutations of matrix and quiver.}
An integer matrix
$B=(B_{ij})_{i,j\in I}$  is {\em skew-symmetrizable\/}
if there is a diagonal matrix $D=\mathrm{diag}
(d_i)_{i\in I}$ with $d_i\in \mathbb{N}$
such that $DB$ is skew-symmetric.
For a skew-symmetrizable matrix $B$ and
$k\in I$, another matrix $B'=\mu_k(B)$,
called the {\em mutation of $B$ at $k$\/}, is defined by
\begin{align}
\label{eq:Bmut}
B'_{ij}=
\begin{cases}
-B_{ij}& \mbox{$i=k$ or $j=k$},\\
B_{ij}+\frac{1}{2}
(|B_{ik}|B_{kj} + B_{ik}|B_{kj}|)
&\mbox{otherwise}.
\end{cases}
\end{align}
The matrix $\mu_k(B)$ is also skew-symmetrizable.

It is standard to represent
a {\em skew-symmetric} (integer) matrix $B=(B_{ij})_{i,j\in I}$
by a {\em quiver $Q$
without loops or 2-cycles}.
The set of the vertices of $Q$ is given by $I$,
and we put $B_{ij}$ arrows from $i$ to $j$ 
if $B_{ij}>0$.
The mutation $Q'=\mu_k(Q)$ of quiver $Q$ is given by the following
rule:
For each pair of an incoming arrow $i\rightarrow k$
and an outgoing arrow $k\rightarrow j$ in $Q$,
add a new arrow $i\rightarrow j$.
Then, remove a maximal set of pairwise disjoint 2-cycles.
Finally, reverse all arrows incident with $k$.

{\bf (iii)  Exchange relation of coefficient tuple.}
Let $\mathbb{P}$ be a given semifield.
For an $I$-tuple $y=(y_i)_{i\in I}$, $y_i\in \mathbb{P}$
and $k\in I$, another $I$-tuple $y'$ is defined 
by the {\em exchange relation}
\begin{align}
\label{eq:coef}
y'_i =
\begin{cases}
\displaystyle
{y_k}{}^{-1}&i=k,\\
\displaystyle
y_i \left(\frac{y_k}{1\oplus {y_k}}\right)^{B_{ki}}&
i\neq k,\ B_{ki}\geq 0,\\
y_i (1\oplus y_k)^{-B_{ki}}&
i\neq k,\ B_{ki}\leq 0.\\
\end{cases}
\end{align}

{\bf (iv)   Exchange relation of cluster.}
Let $\mathbb{QP}$ be the  quotient field of the group ring 
$\mathbb{Z}\mathbb{P}$ of $\mathbb{P}$,
 and let $\mathbb{QP}(z)$ be the rational function field of
algebraically independent variables, say, $z=(z_i)_{i\in I}$
over $\mathbb{QP}$.
For an $I$-tuple $x=(x_i)_{i\in I}$ which
is a free generating set of $\mathbb{QP}(z)$
and $k\in I$, another $I$-tuple $x'$ is defined 
by the {\em exchange relation}
\begin{align}
\label{eq:clust}
x'_i =
\begin{cases}
{x_i}&i\neq k,\\
\displaystyle
\frac{y_k
\prod_{j: B_{jk}>0} x_j^{B_{jk}}
+
\prod_{j: B_{jk}<0} x_j^{-B_{jk}}
}{(1\oplus y_k)x_k}
&
i= k.\\
\end{cases}
\end{align}

{\bf (v)  Seed mutation.} For the above triplet $(B,x,y)$
in (ii)--(iv), which is called a {\em seed},  the mutation
$\mu_k(B,x,y)=(B',x',y')$  at $k$ is defined 
 by combining \eqref{eq:Bmut}, \eqref{eq:coef},
and \eqref{eq:clust}.

{\bf (vi)  Cluster algebra}. Fix a semifield $\mathbb{P}$
and a seed ({\em initial seed\/}) $(B,x,y)$, where
$x=(x_i)_{i\in I}$  are algebraically independent variables
over $\mathbb{Q}\mathbb{P}$.
Starting from $(B,x,y)$, iterate mutations and collect all the
seeds $(B',x',y')$.
We call  $y'$  and $y'_i$ a {\em coefficient tuple} and
a {\em coefficient}, respectively.
We call  $x'$  and $x'_i\in \mathbb{Q}\mathbb{P}(x)$, a {\em cluster} and
a {\em cluster variable}, respectively.
The {\em cluster algebra $\mathcal{A}(B,x,y)$ with
coefficients in $\mathbb{P}$} is the
$\mathbb{Z}\mathbb{P}$-subalgebra of the
rational function field $\mathbb{Q}\mathbb{P}(x)$
generated by all the cluster variables.

{\bf (vii)  Cluster pattern.}
Let $I=\{1,\dots,n\}$, and let $\mathbb{T}_n$ be the
$n$-regular tree whose edges are labeled by the numbers
$1,\dots,n$. A {\em cluster pattern} is the assignment
of a seed $(B(t),x(t),y(t))$ for each vertex $t\in \mathbb{T}_n$
such that the seeds assigned to the endpoints of any edge
$t \frac{k}{\phantom{xxx}} t'$ are obtained from each other by the seed
mutation at $k$. Take $t_0\in \mathbb{T}_n$ arbitrarily, and
consider the cluster algebra $\mathcal{A}(B(t_0),x(t_0),y(t_0))$.
Then, $x(t)$ and $y(t)$ ($t\in \mathbb{T}_n$) are 
a cluster and a coefficient tuple of $\mathcal{A}(B(t_0),x(t_0),y(t_0))$.

{\bf (viii) $F$-polynomial.} 
The cluster algebra $\mathcal{A}(B,x,y)$ with coefficients in
the tropical semifield $\mathrm{Trop}(y)$ is called the
cluster algebra with {\em principal coefficients}.
There, each cluster variable $x'_i$ is an element
in $\mathbb{Z}[x^{\pm 1},y]$.
The $F$-polynomial $F'_i(y)\in \mathbb{Z}[y]$ (for $x'_i$)
is defined as the specialization of $x'_i$ with $x_i=1$ ($i\in I$).

\subsection{Parity decompositions of T and Y-systems}

For a triplet $(a,m,u)\in \mathcal{I}_{\ell}$,
we set the `parity conditions' $\mathbf{P}_{+}$ and
$\mathbf{P}_{-}$ by
\begin{align}
\label{eq:Pcond}
\mathbf{P}_{+}:& \ \mbox{
$2u$ is even
if
$a\neq r$; $m+2u$ is even if $a=r$},\\
\mathbf{P}_{-}:& \ \mbox{
$2u$ is odd if
$a\neq r$; $m+2u$ is odd if $a=r$}.
\end{align}
We write, for example, $(a,m,u):\mathbf{P}_{+}$ if $(a,m,u)$ satisfies
$\mathbf{P}_{+}$.
We have $\mathcal{I}_{\ell}=
\mathcal{I}_{\ell+}\sqcup \mathcal{I}_{\ell-}$,
where $\mathcal{I}_{\ell\varepsilon}$ 
is the set of all $(a,m,u):\mathbf{P}_{\varepsilon}$.

Define $\EuScript{T}^{\circ}_{\ell}(B_r)_{\varepsilon}$
($\varepsilon=\pm$)
to be the subring of $\EuScript{T}^{\circ}_{\ell}(B_r)$
generated by
 $T^{(a)}_m(u)$
$((a,m,u)\in \mathcal{I}_{\ell\varepsilon})$.
Then, we have
$\EuScript{T}^{\circ}_{\ell}(B_r)_+
\simeq
\EuScript{T}^{\circ}_{\ell}(B_r)_-
$
by $T^{(a)}_m(u)\mapsto T^{(a)}_m(u+\frac{1}{2})$ and
\begin{align}
\EuScript{T}^{\circ}_{\ell}(B_r)
\simeq
\EuScript{T}^{\circ}_{\ell}(B_r)_+
\otimes_{\mathbb{Z}}
\EuScript{T}^{\circ}_{\ell}(B_r)_-.
\end{align}

For a triplet $(a,m,u)\in \mathcal{I}_{\ell}$ ,
we set another `parity conditions' $\mathbf{P}'_{+}$ and
$\mathbf{P}'_{-}$ by
\begin{align}
\label{eq:Pcond2}
\mathbf{P}'_{+}:&\ \mbox{
$2u$ is even if
$a\neq r$;
 $m+2u$ is odd if $a=r$},\\
\mathbf{P}'_{-}:&\ \mbox{
$2u$ is odd if
$a\neq r$;
 $m+2u$ is even if $a=r$}.
\end{align}
We have $\mathcal{I}_{\ell}=
\mathcal{I}'_{\ell+}\sqcup \mathcal{I}'_{\ell-}$,
where $\mathcal{I}'_{\ell\varepsilon}$ 
is the set of all $(a,m,u):\mathbf{P}'_{\varepsilon}$.
We also have
\begin{align}
(a,m,u):\mathbf{P}'_+
\ 
\Longleftrightarrow
\
\textstyle
(a,m,u\pm \frac{1}{t_a}):\mathbf{P}_+.
\end{align}

Define $\EuScript{Y}^{\circ}_{\ell}(B_r)_{\varepsilon}$
($\varepsilon=\pm$)
to be the subgroup of $\EuScript{Y}^{\circ}_{\ell}(B_r)$
generated by
$Y^{(a)}_m(u)$, $1+Y^{(a)}_m(u)$
$((a,m,u)\in \mathcal{I}'_{\ell\varepsilon})$.
Then, we have
$\EuScript{Y}^{\circ}_{\ell}(B_r)_+
\simeq
\EuScript{Y}^{\circ}_{\ell}(B_r)_-
$
by $Y^{(a)}_m(u)\mapsto Y^{(a)}_m(u+\frac{1}{2})$,
$1+Y^{(a)}_m(u)\mapsto 1+Y^{(a)}_m(u+\frac{1}{2})$,
 and
\begin{align}
\EuScript{Y}^{\circ}_{\ell}(B_r)
\simeq
\EuScript{Y}^{\circ}_{\ell}(B_r)_+
\times
\EuScript{Y}^{\circ}_{\ell}(B_r)_-.
\end{align}

\subsection{Quiver $Q_{\ell}(B_r)$}

With type $B_r$ and $\ell\geq 2$ we associate
 the quiver $Q_{\ell}(B_r)$
by Figure
\ref{fig:quiverB},
where, in addition,
we assign the empty or filled circle $\circ$/$\bullet$ and
the sign +/$-$ to each vertex.
Let $B_{\ell}(B_r)$ be the skew-symmetric matrix
corresponding to $Q_{\ell}(B_r)$ as defined
in Section \ref{subsec:groc} (ii).
(Unfortunately, there is an obvious confliction
between two standard notations using symbol $B$.
We hope it does not cause serious confusion
to the reader.)

Let us choose  the index set $\mathbf{I}$
of the vertices of $Q_{\ell}(B_r)$
so that $\mathbf{i}=(i,i')\in \mathbf{I}$ represents
the vertex 
at the $i'$th row (from the bottom)
of the $i$th column (from the left).
Thus, $i=1,\dots,2r-1$, and $i'=1,\dots,\ell-1$ if $i\neq r$
and $i'=1,\dots,2\ell-1$ if $i=r$.
We use a natural notation $\mathbf{I}^{\circ}$
(resp.\ $\mathbf{I}^{\circ}_+$)
for the set of the vertices $\mathbf{i}$ with
property $\circ$ (resp.\ $\circ$ and +), and so on.
We have $\mathbf{I}=\mathbf{I}^{\circ}\sqcup \mathbf{I}^{\bullet}
=\mathbf{I}^{\circ}_+\sqcup
\mathbf{I}^{\circ}_-\sqcup
\mathbf{I}^{\bullet}_+\sqcup
\mathbf{I}^{\bullet}_-$.

We define composite mutations,
\begin{align}
\label{eq:mupm2}
\mu^{\circ}_+=\prod_{\mathbf{i}\in\mathbf{I}^{\circ}_+}
\mu_{\mathbf{i}},
\quad
\mu^{\circ}_-=\prod_{\mathbf{i}\in\mathbf{I}^{\circ}_-}
\mu_{\mathbf{i}},
\quad
\mu^{\bullet}_+=\prod_{\mathbf{i}\in\mathbf{I}^{\bullet}_+}
\mu_{\mathbf{i}},
\quad
\mu^{\bullet}_-=\prod_{\mathbf{i}\in\mathbf{I}^{\bullet}_-}
\mu_{\mathbf{i}}.
\end{align}
Note that they do not depend on the order of the product.

\begin{figure}[t]
\begin{picture}(318,110)(-80,-30)
\put(-30,0){\circle{5}}
\put(0,0){\circle{5}}
\put(30,0){\circle{5}}
\put(60,0){\circle{5}}
\put(90,0){\circle*{5}}
\put(120,0){\circle{5}}
\put(150,0){\circle{5}}
\put(180,0){\circle{5}}
\put(210,0){\circle{5}}
\put(90,15){\circle*{5}}
\put(-3,0){\vector(-1,0){24}}
\put(3,0){\vector(1,0){24}}
\put(57,0){\vector(-1,0){24}}
\put(87,0){\vector(-1,0){24}}
\put(93,0){\vector(1,0){24}}
\put(147,0){\vector(-1,0){24}}
\put(153,0){\vector(1,0){24}}
\put(207,0){\vector(-1,0){24}}
\put(0,30)
{
\put(-30,0){\circle{5}}
\put(0,0){\circle{5}}
\put(30,0){\circle{5}}
\put(60,0){\circle{5}}
\put(90,0){\circle*{5}}
\put(120,0){\circle{5}}
\put(150,0){\circle{5}}
\put(180,0){\circle{5}}
\put(210,0){\circle{5}}
\put(90,15){\circle*{5}}
\put(-27,0){\vector(1,0){24}}
\put(27,0){\vector(-1,0){24}}
\put(33,0){\vector(1,0){24}}
\put(87,0){\vector(-1,0){24}}
\put(93,0){\vector(1,0){24}}
\put(123,0){\vector(1,0){24}}
\put(177,0){\vector(-1,0){24}}
\put(183,0){\vector(1,0){24}}
}
\put(0,60)
{
\put(-30,0){\circle{5}}
\put(0,0){\circle{5}}
\put(30,0){\circle{5}}
\put(60,0){\circle{5}}
\put(90,0){\circle*{5}}
\put(120,0){\circle{5}}
\put(150,0){\circle{5}}
\put(180,0){\circle{5}}
\put(210,0){\circle{5}}
\put(90,15){\circle*{5}}
\put(-3,0){\vector(-1,0){24}}
\put(3,0){\vector(1,0){24}}
\put(57,0){\vector(-1,0){24}}
\put(87,0){\vector(-1,0){24}}
\put(93,0){\vector(1,0){24}}
\put(147,0){\vector(-1,0){24}}
\put(153,0){\vector(1,0){24}}
\put(207,0){\vector(-1,0){24}}
}
\put(90,-15){\circle*{5}}
%
\put(90,-12){\vector(0,1){9}}
\put(90,12){\vector(0,-1){9}}
\put(90,18){\vector(0,1){9}}
\put(90,42){\vector(0,-1){9}}
\put(90,48){\vector(0,1){9}}
\put(90,72){\vector(0,-1){9}}
\put(63,-2){\vector(2,-1){24}}
\put(63,2){\vector(2,1){24}}
\put(63,58){\vector(2,-1){24}}
\put(63,62){\vector(2,1){24}}
\put(117,28){\vector(-2,-1){24}}
\put(117,32){\vector(-2,1){24}}
\put(-30,3){\vector(0,1){24}}
\put(0,27){\vector(0,-1){24}}
\put(30,3){\vector(0,1){24}}
\put(60,27){\vector(0,-1){24}}
\put(120,3){\vector(0,1){24}}
\put(150,27){\vector(0,-1){24}}
\put(180,3){\vector(0,1){24}}
\put(210,27){\vector(0,-1){24}}
\put(-30,57){\vector(0,-1){24}}
\put(0,33){\vector(0,1){24}}
\put(30,57){\vector(0,-1){24}}
\put(60,33){\vector(0,1){24}}
\put(120,57){\vector(0,-1){24}}
\put(150,33){\vector(0,1){24}}
\put(180,57){\vector(0,-1){24}}
\put(210,33){\vector(0,1){24}}
\put(-50,-2){$\cdots$}
\put(-50,28){$\cdots$}
\put(-50,58){$\cdots$}
\put(220,-2){$\cdots$}
\put(220,28){$\cdots$}
\put(220,58){$\cdots$}
\put(0,1)
{
\put(-28,2){\small $ +$}
\put(-28,32){\small $-$}
\put(-28,62){\small $+$}
\put(2,2){\small $-$}
\put(2,32){\small $+$}
\put(2,62){\small $-$}
\put(32,2){\small $+$}
\put(32,32){\small $-$}
\put(32,62){\small $+$}
\put(62,5){\small $-$}
\put(62,32){\small $+$}
\put(62,65){\small $-$}
\put(92,-13){\small $+$}
\put(92,2){\small $-$}
\put(92,20){\small $+$}
\put(92,32){\small $-$}
\put(92,47){\small $+$}
\put(92,62){\small $-$}
\put(92,77){\small $+$}
\put(122,2){\small $+$}
\put(122,32){\small $-$}
\put(122,62){\small $+$}
\put(152,2){\small $-$}
\put(152,32){\small $+$}
\put(152,62){\small $-$}
\put(182,2){\small $+$}
\put(182,32){\small $-$}
\put(182,62){\small $+$}
\put(212,2){\small $-$}
\put(212,32){\small $+$}
\put(212,62){\small $-$}
}
\put(-60,-15){$\underbrace{\hskip120pt}_{r-1}$}
\put(120,-15){$\underbrace{\hskip120pt}_{r-1}$}
\put(-85,28){${\scriptstyle\ell -1}\left\{ \makebox(0,37){}\right.$}
\end{picture}
%
%
\begin{picture}(318,150)(-80,-30)
\put(-30,0){\circle{5}}
\put(0,0){\circle{5}}
\put(30,0){\circle{5}}
\put(60,0){\circle{5}}
\put(90,0){\circle*{5}}
\put(120,0){\circle{5}}
\put(150,0){\circle{5}}
\put(180,0){\circle{5}}
\put(210,0){\circle{5}}
\put(90,15){\circle*{5}}
\put(-3,0){\vector(-1,0){24}}
\put(3,0){\vector(1,0){24}}
\put(57,0){\vector(-1,0){24}}
\put(87,0){\vector(-1,0){24}}
\put(93,0){\vector(1,0){24}}
\put(147,0){\vector(-1,0){24}}
\put(153,0){\vector(1,0){24}}
\put(207,0){\vector(-1,0){24}}
\put(0,30)
{
\put(-30,0){\circle{5}}
\put(0,0){\circle{5}}
\put(30,0){\circle{5}}
\put(60,0){\circle{5}}
\put(90,0){\circle*{5}}
\put(120,0){\circle{5}}
\put(150,0){\circle{5}}
\put(180,0){\circle{5}}
\put(210,0){\circle{5}}
\put(90,15){\circle*{5}}
\put(-27,0){\vector(1,0){24}}
\put(27,0){\vector(-1,0){24}}
\put(33,0){\vector(1,0){24}}
\put(87,0){\vector(-1,0){24}}
\put(93,0){\vector(1,0){24}}
\put(123,0){\vector(1,0){24}}
\put(177,0){\vector(-1,0){24}}
\put(183,0){\vector(1,0){24}}
}
\put(0,60)
{
\put(-30,0){\circle{5}}
\put(0,0){\circle{5}}
\put(30,0){\circle{5}}
\put(60,0){\circle{5}}
\put(90,0){\circle*{5}}
\put(120,0){\circle{5}}
\put(150,0){\circle{5}}
\put(180,0){\circle{5}}
\put(210,0){\circle{5}}
\put(90,15){\circle*{5}}
\put(-3,0){\vector(-1,0){24}}
\put(3,0){\vector(1,0){24}}
\put(57,0){\vector(-1,0){24}}
\put(87,0){\vector(-1,0){24}}
\put(93,0){\vector(1,0){24}}
\put(147,0){\vector(-1,0){24}}
\put(153,0){\vector(1,0){24}}
\put(207,0){\vector(-1,0){24}}
}
\put(90,-15){\circle*{5}}
\put(0,90)
{
\put(-30,0){\circle{5}}
\put(0,0){\circle{5}}
\put(30,0){\circle{5}}
\put(60,0){\circle{5}}
\put(90,0){\circle*{5}}
\put(120,0){\circle{5}}
\put(150,0){\circle{5}}
\put(180,0){\circle{5}}
\put(210,0){\circle{5}}
\put(90,15){\circle*{5}}
\put(-27,0){\vector(1,0){24}}
\put(27,0){\vector(-1,0){24}}
\put(33,0){\vector(1,0){24}}
\put(87,0){\vector(-1,0){24}}
\put(93,0){\vector(1,0){24}}
\put(123,0){\vector(1,0){24}}
\put(177,0){\vector(-1,0){24}}
\put(183,0){\vector(1,0){24}}
}
\put(90,-12){\vector(0,1){9}}
\put(90,12){\vector(0,-1){9}}
\put(90,18){\vector(0,1){9}}
\put(90,42){\vector(0,-1){9}}
\put(90,48){\vector(0,1){9}}
\put(90,72){\vector(0,-1){9}}
\put(90,78){\vector(0,1){9}}
\put(90,102){\vector(0,-1){9}}
\put(-30,3){\vector(0,1){24}}
\put(0,27){\vector(0,-1){24}}
\put(30,3){\vector(0,1){24}}
\put(60,27){\vector(0,-1){24}}
\put(120,3){\vector(0,1){24}}
\put(150,27){\vector(0,-1){24}}
\put(180,3){\vector(0,1){24}}
\put(210,27){\vector(0,-1){24}}
\put(-30,57){\vector(0,-1){24}}
\put(0,33){\vector(0,1){24}}
\put(30,57){\vector(0,-1){24}}
\put(60,33){\vector(0,1){24}}
\put(120,57){\vector(0,-1){24}}
\put(150,33){\vector(0,1){24}}
\put(180,57){\vector(0,-1){24}}
\put(210,33){\vector(0,1){24}}
\put(-30,63){\vector(0,1){24}}
\put(0,87){\vector(0,-1){24}}
\put(30,63){\vector(0,1){24}}
\put(60,87){\vector(0,-1){24}}
\put(120,63){\vector(0,1){24}}
\put(150,87){\vector(0,-1){24}}
\put(180,63){\vector(0,1){24}}
\put(210,87){\vector(0,-1){24}}
\put(63,-2){\vector(2,-1){24}}
\put(63,2){\vector(2,1){24}}
\put(63,58){\vector(2,-1){24}}
\put(63,62){\vector(2,1){24}}
\put(117,28){\vector(-2,-1){24}}
\put(117,32){\vector(-2,1){24}}
\put(117,88){\vector(-2,-1){24}}
\put(117,92){\vector(-2,1){24}}
\put(-50,-2){$\cdots$}
\put(-50,28){$\cdots$}
\put(-50,58){$\cdots$}
\put(-50,88){$\cdots$}
\put(220,-2){$\cdots$}
\put(220,28){$\cdots$}
\put(220,58){$\cdots$}
\put(220,88){$\cdots$}
\put(0,1)
{
\put(-28,2){\small $ +$}
\put(-28,32){\small $-$}
\put(-28,62){\small $+$}
\put(-28,92){\small $-$}
\put(2,2){\small $-$}
\put(2,32){\small $+$}
\put(2,62){\small $-$}
\put(2,92){\small $+$}
\put(32,2){\small $+$}
\put(32,32){\small $-$}
\put(32,62){\small $+$}
\put(32,92){\small $-$}
\put(62,5){\small $-$}
\put(62,32){\small $+$}
\put(62,65){\small $-$}
\put(62,92){\small $+$}
\put(92,-13){\small $+$}
\put(92,2){\small $-$}
\put(92,20){\small $+$}
\put(92,32){\small $-$}
\put(92,47){\small $+$}
\put(92,62){\small $-$}
\put(92,80){\small $+$}
\put(92,92){\small $-$}
\put(92,107){\small $+$}
\put(122,2){\small $+$}
\put(122,32){\small $-$}
\put(122,62){\small $+$}
\put(122,92){\small $-$}
\put(152,2){\small $-$}
\put(152,32){\small $+$}
\put(152,62){\small $-$}
\put(152,92){\small $+$}
\put(182,2){\small $+$}
\put(182,32){\small $-$}
\put(182,62){\small $+$}
\put(182,92){\small $-$}
\put(212,2){\small $-$}
\put(212,32){\small $+$}
\put(212,62){\small $-$}
\put(212,92){\small $+$}
}
\put(-60,-15){$\underbrace{\hskip120pt}_{r-1}$}
\put(120,-15){$\underbrace{\hskip120pt}_{r-1}$}
\put(-85,43){${\scriptstyle\ell -1}\left\{ \makebox(0,50){}\right.$}
\end{picture}
\caption{The quiver $Q_{\ell}(B_r)$  for even $\ell$
(upper) and for odd $\ell$ (lower).}
\label{fig:quiverB}
\end{figure}
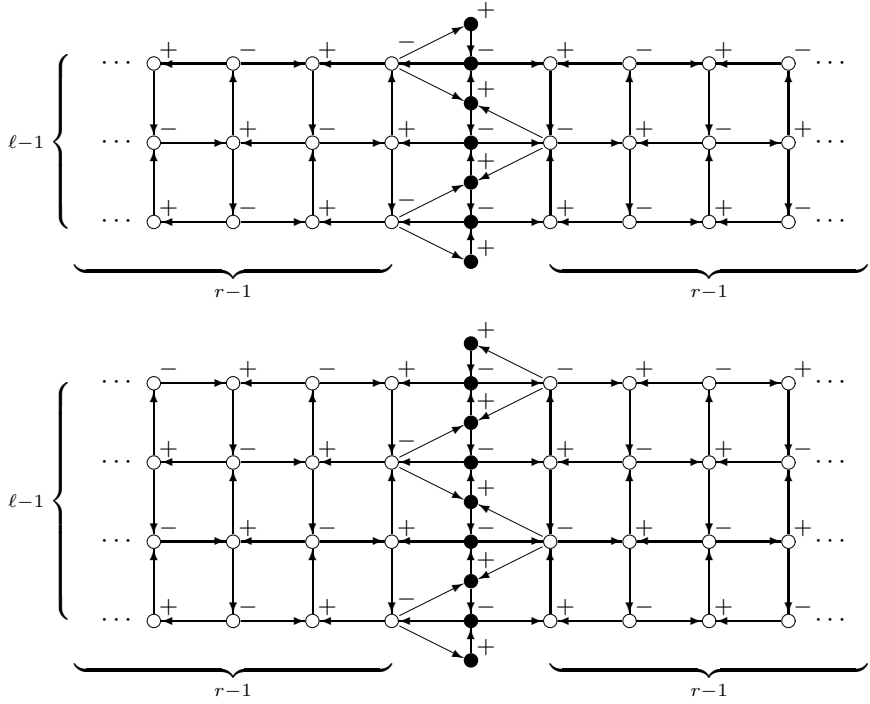

Let $\boldsymbol{r}$ be the involution acting on $\mathbf{I}$
by the left-right reflection.
Let $\boldsymbol{\omega}$ be the involution acting on $\mathbf{I}$
by the $180^{\circ}$ rotation.
Let $\boldsymbol{r}(Q_{\ell}(B_r))$
and $\boldsymbol{\omega}(Q_{\ell}(B_r))$ denote the quivers
 induced from $Q_{\ell}(B_r)$
 by
$\boldsymbol{r}$ and $\boldsymbol{\omega}$,
respectively.
For example, 
if there is an arrow
 $\mathbf{i}\rightarrow
\mathbf{j}$ in $Q_{\ell}(B_r)$,
then, there is an arrow
$\boldsymbol{r}(\mathbf{i})
\rightarrow
\boldsymbol{r}(\mathbf{j})
$
in  $\boldsymbol{r}(Q_{\ell}(B_r))$.
For a quiver $Q$, $Q^{\mathrm{op}}$ denotes the opposite quiver.

\begin{lemma}
\label{lem:Qmut}
Let $Q=Q_{\ell}(B_r)$.
\par
(i)
We have a periodic sequence of mutations of quivers
\begin{align}
\label{eq:B2}
Q\ 
\mathop{\longleftrightarrow}^{\mu^{\bullet}_+
\mu^{\circ}_+}
\
Q^{\mathrm{op}}
\
\mathop{\longleftrightarrow}^{\mu^{\bullet}_-}
\
\boldsymbol{r}(Q)
\
\mathop{\longleftrightarrow}^{\mu^{\bullet}_+
\mu^{\circ}_-}
\
\boldsymbol{r}(Q){}^{\mathrm{op}}
\
\mathop{\longleftrightarrow}^{\mu^{\bullet}_-}
\
Q.
\end{align}
\par
(ii) $\boldsymbol{\omega}(Q)=Q$ if $h^{\vee}+\ell$ is even,
and
$\boldsymbol{\omega}(Q)=\boldsymbol{r}(Q)$ if $h^{\vee}+\ell$
 is odd.
\end{lemma}

\begin{example} The sequence \eqref{eq:B2} for
$Q=Q_{2}(B_3)$ is given below.
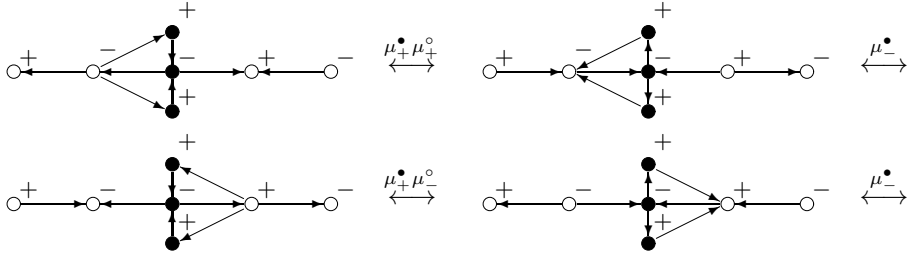
\begin{align*}
\begin{picture}(310,95)(-70,0)
\put(-110,70)
{
\put(30,0){\circle{5}}
\put(60,0){\circle{5}}
\put(90,0){\circle*{5}}
\put(120,0){\circle{5}}
\put(150,0){\circle{5}}
\put(90,15){\circle*{5}}
\put(57,0){\vector(-1,0){24}}
\put(87,0){\vector(-1,0){24}}
\put(93,0){\vector(1,0){24}}
\put(147,0){\vector(-1,0){24}}
\put(90,-15){\circle*{5}}
\put(90,12){\vector(0,-1){9}}
\put(90,-12){\vector(0,1){9}}
\put(63,-2){\vector(2,-1){24}}
\put(63,2){\vector(2,1){24}}
\put(0,1)
{
\put(32,2){\small $+$}
\put(62,5){\small $-$}
\put(92,-13){\small $+$}
\put(92,2){\small $-$}
\put(92,20){\small $+$}
\put(122,2){\small $+$}
\put(152,2){\small $-$}
}
\put(170,0){$\displaystyle
\mathop{\longleftrightarrow}^{\mu^{\bullet}_+
\mu^{\circ}_+}$
}
}
\put(70,70)
{
\put(30,0){\circle{5}}
\put(60,0){\circle{5}}
\put(90,0){\circle*{5}}
\put(120,0){\circle{5}}
\put(150,0){\circle{5}}
\put(90,15){\circle*{5}}
\put(33,0){\vector(1,0){24}}
\put(63,0){\vector(1,0){24}}
\put(117,0){\vector(-1,0){24}}
\put(123,0){\vector(1,0){24}}
\put(90,-15){\circle*{5}}
\put(90,-3){\vector(0,-1){9}}
\put(90,3){\vector(0,1){9}}
\put(87,-13){\vector(-2,1){24}}
\put(87,13){\vector(-2,-1){24}}
\put(0,1)
{
\put(32,2){\small $+$}
\put(62,5){\small $-$}
\put(92,-13){\small $+$}
\put(92,2){\small $-$}
\put(92,20){\small $+$}
\put(122,2){\small $+$}
\put(152,2){\small $-$}
}
\put(170,0){$\displaystyle
\mathop{\longleftrightarrow}^{\mu^{\bullet}_-}$}
}
\put(-110,20)
{
\put(30,0){\circle{5}}
\put(60,0){\circle{5}}
\put(90,0){\circle*{5}}
\put(120,0){\circle{5}}
\put(150,0){\circle{5}}
\put(90,15){\circle*{5}}
\put(33,0){\vector(1,0){24}}
\put(87,0){\vector(-1,0){24}}
\put(93,0){\vector(1,0){24}}
\put(123,0){\vector(1,0){24}}
\put(90,-15){\circle*{5}}
\put(90,12){\vector(0,-1){9}}
\put(90,-12){\vector(0,1){9}}
\put(117,-2){\vector(-2,-1){24}}
\put(117,2){\vector(-2,1){24}}
\put(0,1)
{
\put(32,2){\small $+$}
\put(62,2){\small $-$}
\put(92,-10){\small $+$}
\put(92,2){\small $-$}
\put(92,20){\small $+$}
\put(122,2){\small $+$}
\put(152,2){\small $-$}
}
\put(170,0){$\displaystyle
\mathop{\longleftrightarrow}^{\mu^{\bullet}_+
\mu^{\circ}_-}$}
}
\put(70,20)
{
\put(30,0){\circle{5}}
\put(60,0){\circle{5}}
\put(90,0){\circle*{5}}
\put(120,0){\circle{5}}
\put(150,0){\circle{5}}
\put(90,15){\circle*{5}}
\put(57,0){\vector(-1,0){24}}
\put(63,0){\vector(1,0){24}}
\put(117,0){\vector(-1,0){24}}
\put(147,0){\vector(-1,0){24}}
\put(90,-15){\circle*{5}}
\put(90,3){\vector(0,1){9}}
\put(90,-3){\vector(0,-1){9}}
\put(93,13){\vector(2,-1){24}}
\put(93,-13){\vector(2,1){24}}
\put(0,1)
{
\put(32,2){\small $+$}
\put(62,2){\small $-$}
\put(92,-10){\small $+$}
\put(92,2){\small $-$}
\put(92,20){\small $+$}
\put(122,2){\small $+$}
\put(152,2){\small $-$}
}
\put(170,0){$\displaystyle
\mathop{\longleftrightarrow}^{\mu^{\bullet}_-}$}
}
\end{picture}
\end{align*}
\end{example}

\subsection{Cluster algebra $\mathcal{A}(B,x,y)$ 
and coefficient group $\mathcal{G}(B,y)$}

For the matrix $B=(B_{\mathbf{i}\mathbf{j}})_{\mathbf{i},
\mathbf{j}\in \mathbf{I}}=B_{\ell}(B_r)$,
let $\mathcal{A}(B,x,y)$ 
be the {\em cluster algebra
 with coefficients
in the universal
semifield
$\mathbb{Q}_{\mathrm{sf}}(y)$},
where $(B,x,y)$ is the initial seed.
(Here we use the symbol $+$ instead of $\oplus$ 
in $\mathbb{Q}_{\mathrm{sf}}(y)$,
since it is the ordinary addition of subtraction-free
expressions of rational functions of $y$.)

To our purpose, it is natural to introduce not only
the `ring of cluster variables' but also
the `group of coefficients'.

\begin{definition}
The {\em coefficient group $\mathcal{G}(B,y)$
associated with $\mathcal{A}(B,x,y)$}
is the multiplicative subgroup of
the semifield $\mathbb{Q}_{\mathrm{sf}}(y)$ generated by all
the coefficients $y_{\mathbf{i}}'$ of $\mathcal{A}(B,x,y)$
together with $1+y_{\mathbf{i}}'$.
\end{definition}

In view of Lemma \ref{lem:Qmut}
we set $x(0)=x$, $y(0)=y$ and define 
clusters $x(u)=(x_{\mathbf{i}}(u))_{\mathbf{i}\in \mathbf{I}}$
 ($u\in \frac{1}{2}\mathbb{Z}$)
 and coefficient tuples $y(u)=(y_\mathbf{i}(u))_{\mathbf{i}\in \mathbf{I}}$
 ($u\in \frac{1}{2}\mathbb{Z}$)
by the sequence of mutations
\begin{align}
\label{eq:mutseq}
\begin{split}
\cdots
& 
\mathop{\longleftrightarrow}^{\mu^{\bullet}_-}
\
(B,x(0),y(0))\
\mathop{\longleftrightarrow}^{\mu^{\bullet}_+
\mu^{\circ}_+}
\
(-B,x({\textstyle \frac{1}{2}),y(\frac{1}{2})})
\\
&
\mathop{\longleftrightarrow}^{\mu^{\bullet}_-}
\
(\boldsymbol{r}(B),x(1),y(1))
\
\mathop{\longleftrightarrow}^{\mu^{\bullet}_+
\mu^{\circ}_-}
\
(-\boldsymbol{r}(B),x({\textstyle \frac{3}{2}),y(\frac{3}{2})})
\
\mathop{\longleftrightarrow}^{\mu^{\bullet}_-}
\
\cdots,
\end{split}
\end{align}
where $\boldsymbol{r}(B)=B'$ is defined
by $B'_{\boldsymbol{r}(\mathbf{i})
\boldsymbol{r}(\mathbf{j})}=B_{\mathbf{i}\mathbf{j}}$.

For a pair $(\mathbf{i},u)\in
 \mathbf{I}\times \frac{1}{2}\mathbb{Z}$,
we set the parity condition $\mathbf{p}_+$
and $\mathbf{p}_-$ by
\begin{align}
\mathbf{p}_+:
\begin{cases}
 \mathbf{i}\in \mathbf{I}^{\circ}_+\sqcup \mathbf{I}^{\bullet}_+
& u\equiv 0\\
 \mathbf{i}\in \mathbf{I}^{\bullet}_-
& u\equiv \frac{1}{2},\frac{3}{2}\\
 \mathbf{i}\in \mathbf{I}^{\circ}_-\sqcup \mathbf{I}^{\bullet}_+
& u\equiv 1,
\end{cases}
\qquad
\mathbf{p}_-:
\begin{cases}
 \mathbf{i}\in \mathbf{I}^{\circ}_+\sqcup \mathbf{I}^{\bullet}_+
& u\equiv \frac{1}{2}\\
 \mathbf{i}\in \mathbf{I}^{\bullet}_-
& u\equiv 0,1\\
 \mathbf{i}\in \mathbf{I}^{\circ}_-\sqcup \mathbf{I}^{\bullet}_+
& u\equiv \frac{3}{2},
\end{cases}
\end{align}
where $\equiv$ is modulo $2\mathbb{Z}$.
We have
\begin{align}
(\mathbf{i},u):\mathbf{p}_+
\quad
\Longleftrightarrow
\quad
(\mathbf{i},u+\frac{1}{2}):\mathbf{p}_-.
\end{align}
Each $(\mathbf{i},u):\mathbf{p}_+$
is a mutation point of \eqref{eq:mutseq} in the forward
direction of $u$,
and each  $(\mathbf{i},u):\mathbf{p}_-$
is so in the backward direction of $u$.
Notice that there are also some $(\mathbf{i},u)$ which
do not satisfy $\mathbf{p}_+$ nor $\mathbf{p}_-$,
and are not mutation points of \eqref{eq:mutseq};
explicitly, they are $(\mathbf{i},u)$
with $\mathbf{i}\in \mathbf{I}^{\circ}_+$,
$u\equiv 1,\frac{3}{2}$ mod $2\mathbb{Z}$,
or 
with $\mathbf{i}\in \mathbf{I}^{\circ}_-$,
$u\equiv 0,\frac{1}{2}$ mod $2\mathbb{Z}$.
Consequently, we have the following relations
for $(\mathbf{i},u):\mathbf{p}_{\pm}$.
\begin{align}
\label{eq:xmu1}
x_{\mathbf{i}}(u)&
\textstyle
=x_{\mathbf{i}}(u\mp\frac{1}{2})
\quad (\mathbf{i}\in \mathbf{I}^{\bullet}),\\
\label{eq:xmu2}
x_{\mathbf{i}}(u)&
\textstyle
=x_{\mathbf{i}}(u\mp\frac{1}{2})
=x_{\mathbf{i}}(u\mp1)=x_{\mathbf{i}}(u\mp\frac{3}{2}),
\quad
(\mathbf{i}\in \mathbf{I}^{\circ}),\\
\label{eq:ymu}
y_{\mathbf{i}}(u)&
\textstyle
=y_{\mathbf{i}}(u\pm\frac{1}{2})^{-1}.
\end{align}

There is a correspondence between the parity condition
$\mathbf{p}_{+}$ here and $\mathbf{P}_{+}$,
$\mathbf{P}'_{+}$
in \eqref{eq:Pcond} and \eqref{eq:Pcond2}.
\begin{lemma}
Below $\equiv$ means the equivalence modulo $2\mathbb{Z}$.
\par
(i)
The map
\begin{align}
\begin{matrix}
g: &\mathcal{I}_{\ell+}&\rightarrow & \{ (\mathbf{i},u): \mathbf{p}_+
\}\hfill \\
&(a,m,u-\frac{1}{t_a})&\mapsto &
\begin{cases}
((a,m),u)& \mbox{\rm $a\neq r$;
$r+a+m+u\equiv 1$}\\
((2r-a,m),u)& \mbox{\rm $a\neq r$;
$r+a+m+u\equiv 0$}\\
((r,m),u)& \mbox{\rm $a=r$}\\
\end{cases}
\end{matrix}
\end{align}
is a bijection.
\par
(ii)
The map
\begin{align}
\begin{matrix}
g': &\mathcal{I}'_{\ell+}&\rightarrow & \{ (\mathbf{i},u): \mathbf{p}_+
\}\hfill \\
&(a,m,u)&\mapsto &
\begin{cases}
((a,m),u)& \mbox{\rm $a\neq r$;
$r+a+m+u\equiv 1$}\\
((2r-a,m),u)& \mbox{\rm $a\neq r$;
$r+a+m+u\equiv 0$}\\
((r,m),u)& a=r\\
\end{cases}
\end{matrix}
\end{align}
is a bijection.
\end{lemma}
\begin{proof}
They can be easily confirmed by looking at the example in Figures
\ref{fig:labelxB1}--\ref{fig:labelyB2}.
\end{proof}

We introduce alternative labels
$x_{\mathbf{i}}(u)=x^{(a)}_m(u-1/t_a)$
($(a,m,u-1/t_a)\in \mathcal{I}_{\ell+}$)
for $(\mathbf{i},u)=g((a,m,u-1/t_a))$
and
$y_{\mathbf{i}}(u)=y^{(a)}_m(u)$
($(a,m,u)\in \mathcal{I}'_{\ell+}$)
for $(\mathbf{i},u)=g'((a,m,u))$,
respectively, as in Figures \ref{fig:labelxB1}--\ref{fig:labelyB2}.
They will be used below (and also in Section \ref{sec:dilog})
to relate the T and Y-systems with
$\mathcal{A}(B,x,y)$ and $\mathcal{G}(B,y)$.

\begin{figure}
\begin{picture}(280,370)(-50,-230)
\put(-60,46.5){$x(0)$}
\put(-19,-7.5)
{
\put(0,0){\framebox(38,15){\tiny $x^{(r-2)}_1(-1)$}}
\put(50,0){\makebox(38,15){\small $*$}}
\put(150,0){\framebox(38,15){\tiny $x^{(r-1)}_1(-1)$}}
\put(200,0){\makebox(38,15){\small $*$}}
\put(0,50){\makebox(38,15){\small $*$}}
\put(50,50){\framebox(38,15){\tiny $x^{(r-1)}_2(-1)$}}
\put(150,50){\makebox(38,15){\small $*$}}
\put(200,50){\framebox(38,15){\tiny $x^{(r-2)}_2(-1)$}}
\put(0,100){\framebox(38,15){\tiny $x^{(r-2)}_3(-1)$}}
\put(50,100){\makebox(38,15){\small $*$}}
\put(150,100){\framebox(38,15){\tiny $x^{(r-1)}_3(-1)$}}
\put(200,100){\makebox(38,15){\small $*$}}
\put(100,-25){\framebox(38,15){\small $x^{(r)}_1(-\frac{1}{2})$}}
\put(100,0){\dashbox{3}(38,15){\small $*$}}
\put(100,25){\framebox(38,15){\small $x^{(r)}_3(-\frac{1}{2})$}}
\put(100,50){\dashbox{3}(38,15){\small $*$}}
\put(100,75){\framebox(38,15){\small $x^{(r)}_5(-\frac{1}{2})$}}
\put(100,100){\dashbox{3}(38,15){\small $*$}}
\put(100,125){\framebox(38,15){\small $x^{(r)}_7(-\frac{1}{2})$}}
}
%
\put(30,0){$\vector(-1,0){10}$}
\put(80,0){$\vector(-1,0){10}$}
\put(120,0){$\vector(1,0){10}$}
\put(180,0){$\vector(-1,0){10}$}
\put(20,50){$\vector(1,0){10}$}
\put(80,50){$\vector(-1,0){10}$}
\put(120,50){$\vector(1,0){10}$}
\put(170,50){$\vector(1,0){10}$}
\put(30,100){$\vector(-1,0){10}$}
\put(80,100){$\vector(-1,0){10}$}
\put(120,100){$\vector(1,0){10}$}
\put(180,100){$\vector(-1,0){10}$}
\put(0,18){\vector(0,1){14}}
\put(0,82){\vector(0,-1){14}}
\put(50,32){\vector(0,-1){14}}
\put(50,68){\vector(0,1){14}}
\put(100,-17){\vector(0,1){9}}
\put(100,17){\vector(0,-1){9}}
\put(100,33){\vector(0,1){9}}
\put(100,67){\vector(0,-1){9}}
\put(100,83){\vector(0,1){9}}
\put(100,117){\vector(0,-1){9}}
\put(150,18){\vector(0,1){14}}
\put(150,82){\vector(0,-1){14}}
\put(200,32){\vector(0,-1){14}}
\put(200,68){\vector(0,1){14}}

\put(70,-10){\vector(2,-1){10}}
\put(70,10){\vector(2,1){10}}
\put(70,90){\vector(2,-1){10}}
\put(70,110){\vector(2,1){10}}
\put(130,40){\vector(-2,-1){10}}
\put(130,60){\vector(-2,1){10}}
\put(-35,-1){$\dots$}
\put(-35,49){$\dots$}
\put(-35,99){$\dots$}
\put(225,-1){$\dots$}
\put(225,49){$\dots$}
\put(225,99){$\dots$}
\dottedline(-20,-40)(220,-40)
\put(-45,-42){$\updownarrow$}
\put(-70,-42){$\mu^{\bullet}_+\mu^{\circ}_+$}
\put(0,-180)
{
\put(-60,46.5){$x(\frac{1}{2})$}
\put(-19,-7.5)
{
\put(0,0){\dashbox{3}(38,15){\small $*$}}
\put(50,0){\makebox(38,15){\small $*$}}
\put(150,0){\dashbox{3}(38,15){\small $*$}}
\put(200,0){\makebox(38,15){\small $*$}}
\put(0,50){\makebox(38,15){\small $*$}}
\put(50,50){\dashbox{3}(38,15){\small $*$}}
\put(150,50){\makebox(38,15){\small $*$}}
\put(200,50){\dashbox{3}(38,15){\small $*$}}
\put(0,100){\dashbox{3}(38,15){\small $*$}}
\put(50,100){\makebox(38,15){\small $*$}}
\put(150,100){\dashbox{3}(38,15){\small $*$}}
\put(200,100){\makebox(38,15){\small $*$}}
\put(100,-25){\dashbox{3}(38,15){\small $*$}}
\put(100,0){\framebox(38,15){\small $x^{(r)}_2(0)$}}
\put(100,25){\dashbox{3}(38,15){\small $*$}}
\put(100,50){\framebox(38,15){\small $x^{(r)}_4(0)$}}
\put(100,75){\dashbox{3}(38,15){\small $*$}}
\put(100,100){\framebox(38,15){\small $x^{(r)}_6(0)$}}
\put(100,125){\dashbox{3}(38,15){\small $*$}}
}
%
\put(20,0){$\vector(1,0){10}$}
\put(70,0){$\vector(1,0){10}$}
\put(130,0){$\vector(-1,0){10}$}
\put(170,0){$\vector(1,0){10}$}
\put(30,50){$\vector(-1,0){10}$}
\put(70,50){$\vector(1,0){10}$}
\put(130,50){$\vector(-1,0){10}$}
\put(180,50){$\vector(-1,0){10}$}
\put(20,100){$\vector(1,0){10}$}
\put(70,100){$\vector(1,0){10}$}
\put(130,100){$\vector(-1,0){10}$}
\put(170,100){$\vector(1,0){10}$}
\put(0,32){\vector(0,-1){14}}
\put(0,68){\vector(0,1){14}}
\put(50,18){\vector(0,1){14}}
\put(50,82){\vector(0,-1){14}}
\put(100,-8){\vector(0,-1){9}}
\put(100,8){\vector(0,1){9}}
\put(100,42){\vector(0,-1){9}}
\put(100,58){\vector(0,1){9}}
\put(100,92){\vector(0,-1){9}}
\put(100,108){\vector(0,1){9}}
\put(150,32){\vector(0,-1){14}}
\put(150,68){\vector(0,1){14}}
\put(200,18){\vector(0,1){14}}
\put(200,82){\vector(0,-1){14}}

\put(80,-15){\vector(-2,1){10}}
\put(80,15){\vector(-2,-1){10}}
\put(80,85){\vector(-2,1){10}}
\put(80,115){\vector(-2,-1){10}}
\put(120,35){\vector(2,1){10}}
\put(120,65){\vector(2,-1){10}}
\put(-35,-1){$\dots$}
\put(-35,49){$\dots$}
\put(-35,99){$\dots$}
\put(225,-1){$\dots$}
\put(225,49){$\dots$}
\put(225,99){$\dots$}
\dottedline(-20,-40)(220,-40)
\put(-45,-42){$\updownarrow$}
\put(-60,-42){$\mu^{\bullet}_-$}
}
\end{picture}
\caption{(Continues to Figure \ref{fig:labelxB2})
Label of cluster variables $x_{\mathbf{i}}(u)$
by $\mathcal{I}_{\ell+}$  for
$B_r$, $\ell=4$.
The variables framed by solid/dashed lines
satisfy the condition $\mathbf{p}_+$/$\mathbf{p}_-$,
respectively.}
\label{fig:labelxB1}
\end{figure}
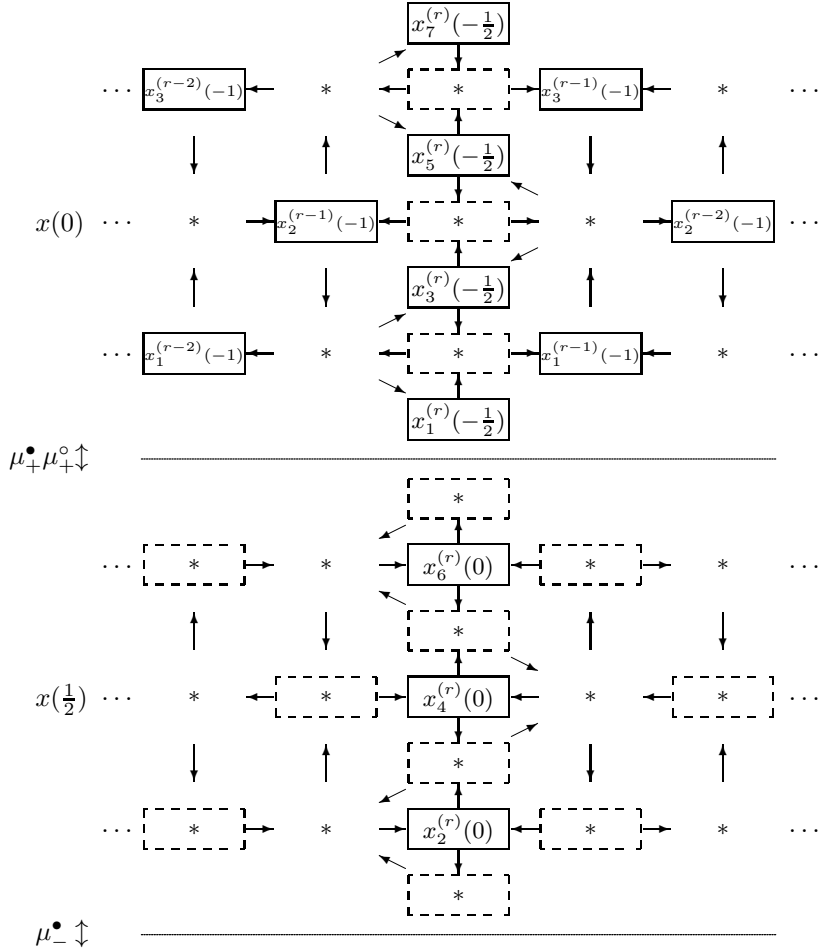

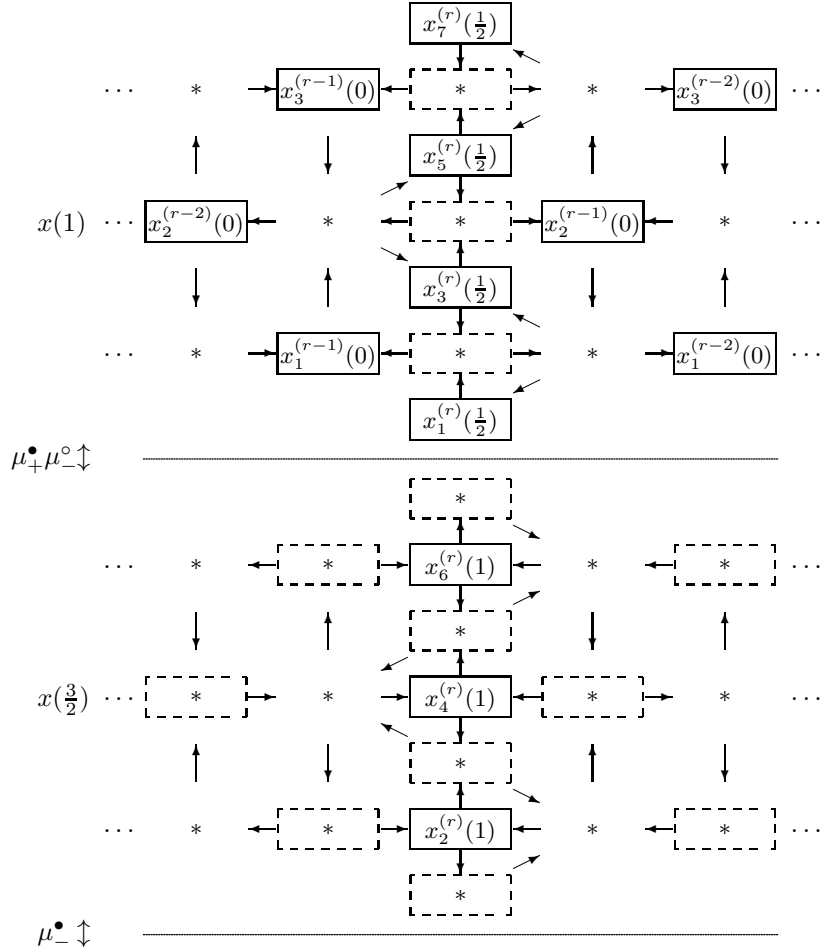
\begin{figure}
\begin{picture}(280,370)(-50,-230)
\put(0,0)
{
\put(-60,46.5){$x(1)$}
\put(-19,-7.5)
{
\put(0,0){\makebox(38,15){\small $*$}}
\put(50,0){\framebox(38,15){\small $x^{(r-1)}_1(0)$}}
\put(150,0){\makebox(38,15){\small $*$}}
\put(200,0){\framebox(38,15){\small $x^{(r-2)}_1(0)$}}
\put(0,50){\framebox(38,15){\small $x^{(r-2)}_2(0)$}}
\put(50,50){\makebox(38,15){\small $*$}}
\put(150,50){\framebox(38,15){\small $x^{(r-1)}_2(0)$}}
\put(200,50){\makebox(38,15){\small $*$}}
\put(0,100){\makebox(38,15){\small $*$}}
\put(50,100){\framebox(38,15){\small $x^{(r-1)}_3(0)$}}
\put(150,100){\makebox(38,15){\small $*$}}
\put(200,100){\framebox(38,15){\small $x^{(r-2)}_3(0)$}}
\put(100,-25){\framebox(38,15){\small $x^{(r)}_1(\frac{1}{2})$}}
\put(100,0){\dashbox{3}(38,15){\small $*$}}
\put(100,25){\framebox(38,15){\small $x^{(r)}_3(\frac{1}{2})$}}
\put(100,50){\dashbox{3}(38,15){\small $*$}}
\put(100,75){\framebox(38,15){\small $x^{(r)}_5(\frac{1}{2})$}}
\put(100,100){\dashbox{3}(38,15){\small $*$}}
\put(100,125){\framebox(38,15){\small $x^{(r)}_7(\frac{1}{2})$}}
}
%
\put(20,0){$\vector(1,0){10}$}
\put(80,0){$\vector(-1,0){10}$}
\put(120,0){$\vector(1,0){10}$}
\put(170,0){$\vector(1,0){10}$}
\put(30,50){$\vector(-1,0){10}$}
\put(80,50){$\vector(-1,0){10}$}
\put(120,50){$\vector(1,0){10}$}
\put(180,50){$\vector(-1,0){10}$}
\put(20,100){$\vector(1,0){10}$}
\put(80,100){$\vector(-1,0){10}$}
\put(120,100){$\vector(1,0){10}$}
\put(170,100){$\vector(1,0){10}$}
\put(0,32){\vector(0,-1){14}}
\put(0,68){\vector(0,1){14}}
\put(50,18){\vector(0,1){14}}
\put(50,82){\vector(0,-1){14}}
\put(100,-17){\vector(0,1){9}}
\put(100,17){\vector(0,-1){9}}
\put(100,33){\vector(0,1){9}}
\put(100,67){\vector(0,-1){9}}
\put(100,83){\vector(0,1){9}}
\put(100,117){\vector(0,-1){9}}
\put(150,32){\vector(0,-1){14}}
\put(150,68){\vector(0,1){14}}
\put(200,18){\vector(0,1){14}}
\put(200,82){\vector(0,-1){14}}

\put(130,-10){\vector(-2,-1){10}}
\put(130,10){\vector(-2,1){10}}
\put(130,90){\vector(-2,-1){10}}
\put(130,110){\vector(-2,1){10}}
\put(70,40){\vector(2,-1){10}}
\put(70,60){\vector(2,1){10}}
\put(-35,-1){$\dots$}
\put(-35,49){$\dots$}
\put(-35,99){$\dots$}
\put(225,-1){$\dots$}
\put(225,49){$\dots$}
\put(225,99){$\dots$}
\dottedline(-20,-40)(220,-40)
\put(-45,-42){$\updownarrow$}
\put(-70,-42){$\mu^{\bullet}_+\mu^{\circ}_-$}
}
\put(0,-180)
{
\put(-60,46.5){$x(\frac{3}{2})$}
\put(-19,-7.5)
{
\put(0,0){\makebox(38,15){\small $*$}}
\put(50,0){\dashbox{3}(38,15){\small $*$}}
\put(150,0){\makebox(38,15){\small $*$}}
\put(200,0){\dashbox{3}(38,15){\small $*$}}
\put(0,50){\dashbox{3}(38,15){\small $*$}}
\put(50,50){\makebox(38,15){\small $*$}}
\put(150,50){\dashbox{3}(38,15){\small $*$}}
\put(200,50){\makebox(38,15){\small $*$}}
\put(0,100){\makebox(38,15){\small $*$}}
\put(50,100){\dashbox{3}(38,15){\small $*$}}
\put(150,100){\makebox(38,15){\small $*$}}
\put(200,100){\dashbox{3}(38,15){\small $*$}}
\put(100,-25){\dashbox{3}(38,15){\small $*$}}
\put(100,0){\framebox(38,15){\small $x^{(r)}_2(1)$}}
\put(100,25){\dashbox{3}(38,15){\small $*$}}
\put(100,50){\framebox(38,15){\small $x^{(r)}_4(1)$}}
\put(100,75){\dashbox{3}(38,15){\small $*$}}
\put(100,100){\framebox(38,15){\small $x^{(r)}_6(1)$}}
\put(100,125){\dashbox{3}(38,15){\small $*$}}
}
%
\put(30,0){$\vector(-1,0){10}$}
\put(70,0){$\vector(1,0){10}$}
\put(130,0){$\vector(-1,0){10}$}
\put(180,0){$\vector(-1,0){10}$}
\put(20,50){$\vector(1,0){10}$}
\put(70,50){$\vector(1,0){10}$}
\put(130,50){$\vector(-1,0){10}$}
\put(170,50){$\vector(1,0){10}$}
\put(30,100){$\vector(-1,0){10}$}
\put(70,100){$\vector(1,0){10}$}
\put(130,100){$\vector(-1,0){10}$}
\put(180,100){$\vector(-1,0){10}$}
\put(0,18){\vector(0,1){14}}
\put(0,82){\vector(0,-1){14}}
\put(50,32){\vector(0,-1){14}}
\put(50,68){\vector(0,1){14}}
\put(100,-8){\vector(0,-1){9}}
\put(100,8){\vector(0,1){9}}
\put(100,42){\vector(0,-1){9}}
\put(100,58){\vector(0,1){9}}
\put(100,92){\vector(0,-1){9}}
\put(100,108){\vector(0,1){9}}
\put(150,18){\vector(0,1){14}}
\put(150,82){\vector(0,-1){14}}
\put(200,32){\vector(0,-1){14}}
\put(200,68){\vector(0,1){14}}

\put(120,-15){\vector(2,1){10}}
\put(120,15){\vector(2,-1){10}}
\put(120,85){\vector(2,1){10}}
\put(120,115){\vector(2,-1){10}}
\put(80,35){\vector(-2,1){10}}
\put(80,65){\vector(-2,-1){10}}
\put(-35,-1){$\dots$}
\put(-35,49){$\dots$}
\put(-35,99){$\dots$}
\put(225,-1){$\dots$}
\put(225,49){$\dots$}
\put(225,99){$\dots$}
\dottedline(-20,-40)(220,-40)
\put(-45,-42){$\updownarrow$}
\put(-60,-42){$\mu^{\bullet}_-$}
}

\end{picture}
\caption{
(Continues from Figure \ref{fig:labelxB1}).
}
\label{fig:labelxB2}
\end{figure}

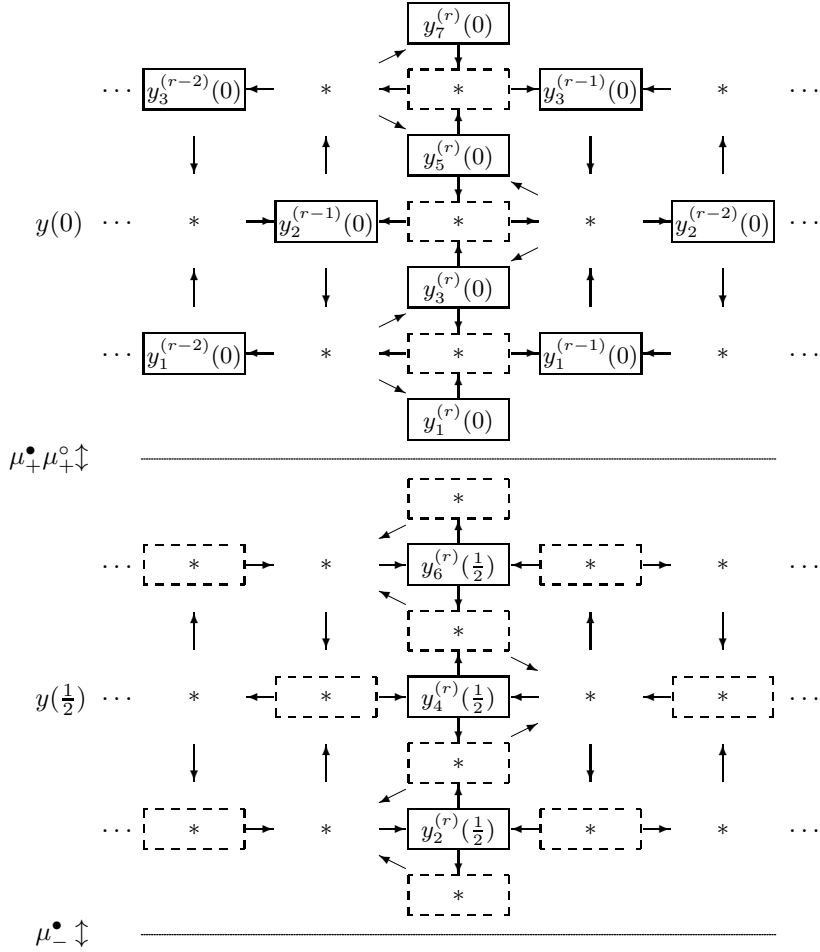
\begin{figure}
\begin{picture}(280,370)(-50,-230)
\put(-60,46.5){$y(0)$}
\put(-19,-7.5)
{
\put(0,0){\framebox(38,15){\small $y^{(r-2)}_1(0)$}}
\put(50,0){\makebox(38,15){\small $*$}}
\put(150,0){\framebox(38,15){\small $y^{(r-1)}_1(0)$}}
\put(200,0){\makebox(38,15){\small $*$}}
\put(0,50){\makebox(38,15){\small $*$}}
\put(50,50){\framebox(38,15){\small $y^{(r-1)}_2(0)$}}
\put(150,50){\makebox(38,15){\small $*$}}
\put(200,50){\framebox(38,15){\small $y^{(r-2)}_2(0)$}}
\put(0,100){\framebox(38,15){\small $y^{(r-2)}_3(0)$}}
\put(50,100){\makebox(38,15){\small $*$}}
\put(150,100){\framebox(38,15){\small $y^{(r-1)}_3(0)$}}
\put(200,100){\makebox(38,15){\small $*$}}
\put(100,-25){\framebox(38,15){\small $y^{(r)}_1(0)$}}
\put(100,0){\dashbox{3}(38,15){\small $*$}}
\put(100,25){\framebox(38,15){\small $y^{(r)}_3(0)$}}
\put(100,50){\dashbox{3}(38,15){\small $*$}}
\put(100,75){\framebox(38,15){\small $y^{(r)}_5(0)$}}
\put(100,100){\dashbox{3}(38,15){\small $*$}}
\put(100,125){\framebox(38,15){\small $y^{(r)}_7(0)$}}
}
%
\put(30,0){$\vector(-1,0){10}$}
\put(80,0){$\vector(-1,0){10}$}
\put(120,0){$\vector(1,0){10}$}
\put(180,0){$\vector(-1,0){10}$}
\put(20,50){$\vector(1,0){10}$}
\put(80,50){$\vector(-1,0){10}$}
\put(120,50){$\vector(1,0){10}$}
\put(170,50){$\vector(1,0){10}$}
\put(30,100){$\vector(-1,0){10}$}
\put(80,100){$\vector(-1,0){10}$}
\put(120,100){$\vector(1,0){10}$}
\put(180,100){$\vector(-1,0){10}$}
\put(0,18){\vector(0,1){14}}
\put(0,82){\vector(0,-1){14}}
\put(50,32){\vector(0,-1){14}}
\put(50,68){\vector(0,1){14}}
\put(100,-17){\vector(0,1){9}}
\put(100,17){\vector(0,-1){9}}
\put(100,33){\vector(0,1){9}}
\put(100,67){\vector(0,-1){9}}
\put(100,83){\vector(0,1){9}}
\put(100,117){\vector(0,-1){9}}
\put(150,18){\vector(0,1){14}}
\put(150,82){\vector(0,-1){14}}
\put(200,32){\vector(0,-1){14}}
\put(200,68){\vector(0,1){14}}

\put(70,-10){\vector(2,-1){10}}
\put(70,10){\vector(2,1){10}}
\put(70,90){\vector(2,-1){10}}
\put(70,110){\vector(2,1){10}}
\put(130,40){\vector(-2,-1){10}}
\put(130,60){\vector(-2,1){10}}
\put(-35,-1){$\dots$}
\put(-35,49){$\dots$}
\put(-35,99){$\dots$}
\put(225,-1){$\dots$}
\put(225,49){$\dots$}
\put(225,99){$\dots$}
\dottedline(-20,-40)(220,-40)
\put(-45,-42){$\updownarrow$}
\put(-70,-42){$\mu^{\bullet}_+\mu^{\circ}_+$}
\put(0,-180)
{
\put(-60,46.5){$y(\frac{1}{2})$}
\put(-19,-7.5)
{
\put(0,0){\dashbox{3}(38,15){\small $*$}}
\put(50,0){\makebox(38,15){\small $*$}}
\put(150,0){\dashbox{3}(38,15){\small $*$}}
\put(200,0){\makebox(38,15){\small $*$}}
\put(0,50){\makebox(38,15){\small $*$}}
\put(50,50){\dashbox{3}(38,15){\small $*$}}
\put(150,50){\makebox(38,15){\small $*$}}
\put(200,50){\dashbox{3}(38,15){\small $*$}}
\put(0,100){\dashbox{3}(38,15){\small $*$}}
\put(50,100){\makebox(38,15){\small $*$}}
\put(150,100){\dashbox{3}(38,15){\small $*$}}
\put(200,100){\makebox(38,15){\small $*$}}
\put(100,-25){\dashbox{3}(38,15){\small $*$}}
\put(100,0){\framebox(38,15){\small $y^{(r)}_2(\frac{1}{2})$}}
\put(100,25){\dashbox{3}(38,15){\small $*$}}
\put(100,50){\framebox(38,15){\small $y^{(r)}_4(\frac{1}{2})$}}
\put(100,75){\dashbox{3}(38,15){\small $*$}}
\put(100,100){\framebox(38,15){\small $y^{(r)}_6(\frac{1}{2})$}}
\put(100,125){\dashbox{3}(38,15){\small $*$}}
}
%
\put(20,0){$\vector(1,0){10}$}
\put(70,0){$\vector(1,0){10}$}
\put(130,0){$\vector(-1,0){10}$}
\put(170,0){$\vector(1,0){10}$}
\put(30,50){$\vector(-1,0){10}$}
\put(70,50){$\vector(1,0){10}$}
\put(130,50){$\vector(-1,0){10}$}
\put(180,50){$\vector(-1,0){10}$}
\put(20,100){$\vector(1,0){10}$}
\put(70,100){$\vector(1,0){10}$}
\put(130,100){$\vector(-1,0){10}$}
\put(170,100){$\vector(1,0){10}$}
\put(0,32){\vector(0,-1){14}}
\put(0,68){\vector(0,1){14}}
\put(50,18){\vector(0,1){14}}
\put(50,82){\vector(0,-1){14}}
\put(100,-8){\vector(0,-1){9}}
\put(100,8){\vector(0,1){9}}
\put(100,42){\vector(0,-1){9}}
\put(100,58){\vector(0,1){9}}
\put(100,92){\vector(0,-1){9}}
\put(100,108){\vector(0,1){9}}
\put(150,32){\vector(0,-1){14}}
\put(150,68){\vector(0,1){14}}
\put(200,18){\vector(0,1){14}}
\put(200,82){\vector(0,-1){14}}

\put(80,-15){\vector(-2,1){10}}
\put(80,15){\vector(-2,-1){10}}
\put(80,85){\vector(-2,1){10}}
\put(80,115){\vector(-2,-1){10}}
\put(120,35){\vector(2,1){10}}
\put(120,65){\vector(2,-1){10}}
\put(-35,-1){$\dots$}
\put(-35,49){$\dots$}
\put(-35,99){$\dots$}
\put(225,-1){$\dots$}
\put(225,49){$\dots$}
\put(225,99){$\dots$}
\dottedline(-20,-40)(220,-40)
\put(-45,-42){$\updownarrow$}
\put(-60,-42){$\mu^{\bullet}_-$}
}
\end{picture}
\caption{(Continues to Figure \ref{fig:labelyB2})
Label of coefficients $y_{\mathbf{i}}(u)$
by $\mathcal{I}'_{\ell+}$  for
$B_r$, $\ell=4$.
The variables framed by solid/dashed lines
satisfy the condition $\mathbf{p}_+$/$\mathbf{p}_-$,
respectively.}
\label{fig:labelyB1}
\end{figure}

\begin{figure}
\begin{picture}(280,370)(-50,-230)
\put(0,0)
{
\put(-60,46.5){$y(1)$}
\put(-19,-7.5)
{
\put(0,0){\makebox(38,15){\small $*$}}
\put(50,0){\framebox(38,15){\small $y^{(r-1)}_1(1)$}}
\put(150,0){\makebox(38,15){\small $*$}}
\put(200,0){\framebox(38,15){\small $y^{(r-2)}_1(1)$}}
\put(0,50){\framebox(38,15){\small $y^{(r-2)}_2(1)$}}
\put(50,50){\makebox(38,15){\small $*$}}
\put(150,50){\framebox(38,15){\small $y^{(r-1)}_2(1)$}}
\put(200,50){\makebox(38,15){\small $*$}}
\put(0,100){\makebox(38,15){\small $*$}}
\put(50,100){\framebox(38,15){\small $y^{(r-1)}_3(1)$}}
\put(150,100){\makebox(38,15){\small $*$}}
\put(200,100){\framebox(38,15){\small $y^{(r-2)}_3(1)$}}
\put(100,-25){\framebox(38,15){\small $y^{(r)}_1(1)$}}
\put(100,0){\dashbox{3}(38,15){\small $*$}}
\put(100,25){\framebox(38,15){\small $y^{(r)}_3(1)$}}
\put(100,50){\dashbox{3}(38,15){\small $*$}}
\put(100,75){\framebox(38,15){\small $y^{(r)}_5(1)$}}
\put(100,100){\dashbox{3}(38,15){\small $*$}}
\put(100,125){\framebox(38,15){\small $y^{(r)}_7(1)$}}
}
%
\put(20,0){$\vector(1,0){10}$}
\put(80,0){$\vector(-1,0){10}$}
\put(120,0){$\vector(1,0){10}$}
\put(170,0){$\vector(1,0){10}$}
\put(30,50){$\vector(-1,0){10}$}
\put(80,50){$\vector(-1,0){10}$}
\put(120,50){$\vector(1,0){10}$}
\put(180,50){$\vector(-1,0){10}$}
\put(20,100){$\vector(1,0){10}$}
\put(80,100){$\vector(-1,0){10}$}
\put(120,100){$\vector(1,0){10}$}
\put(170,100){$\vector(1,0){10}$}
\put(0,32){\vector(0,-1){14}}
\put(0,68){\vector(0,1){14}}
\put(50,18){\vector(0,1){14}}
\put(50,82){\vector(0,-1){14}}
\put(100,-17){\vector(0,1){9}}
\put(100,17){\vector(0,-1){9}}
\put(100,33){\vector(0,1){9}}
\put(100,67){\vector(0,-1){9}}
\put(100,83){\vector(0,1){9}}
\put(100,117){\vector(0,-1){9}}
\put(150,32){\vector(0,-1){14}}
\put(150,68){\vector(0,1){14}}
\put(200,18){\vector(0,1){14}}
\put(200,82){\vector(0,-1){14}}

\put(130,-10){\vector(-2,-1){10}}
\put(130,10){\vector(-2,1){10}}
\put(130,90){\vector(-2,-1){10}}
\put(130,110){\vector(-2,1){10}}
\put(70,40){\vector(2,-1){10}}
\put(70,60){\vector(2,1){10}}
\put(-35,-1){$\dots$}
\put(-35,49){$\dots$}
\put(-35,99){$\dots$}
\put(225,-1){$\dots$}
\put(225,49){$\dots$}
\put(225,99){$\dots$}
\dottedline(-20,-40)(220,-40)
\put(-45,-42){$\updownarrow$}
\put(-70,-42){$\mu^{\bullet}_+\mu^{\circ}_-$}
}
\put(0,-180)
{
\put(-60,46.5){$y(\frac{3}{2})$}
\put(-19,-7.5)
{
\put(0,0){\makebox(38,15){\small $*$}}
\put(50,0){\dashbox{3}(38,15){\small $*$}}
\put(150,0){\makebox(38,15){\small $*$}}
\put(200,0){\dashbox{3}(38,15){\small $*$}}
\put(0,50){\dashbox{3}(38,15){\small $*$}}
\put(50,50){\makebox(38,15){\small $*$}}
\put(150,50){\dashbox{3}(38,15){\small $*$}}
\put(200,50){\makebox(38,15){\small $*$}}
\put(0,100){\makebox(38,15){\small $*$}}
\put(50,100){\dashbox{3}(38,15){\small $*$}}
\put(150,100){\makebox(38,15){\small $*$}}
\put(200,100){\dashbox{3}(38,15){\small $*$}}
\put(100,-25){\dashbox{3}(38,15){\small $*$}}
\put(100,0){\framebox(38,15){\small $y^{(r)}_2(\frac{3}{2})$}}
\put(100,25){\dashbox{3}(38,15){\small $*$}}
\put(100,50){\framebox(38,15){\small $y^{(r)}_4(\frac{3}{2})$}}
\put(100,75){\dashbox{3}(38,15){\small $*$}}
\put(100,100){\framebox(38,15){\small $y^{(r)}_6(\frac{3}{2})$}}
\put(100,125){\dashbox{3}(38,15){\small $*$}}
}
%
\put(30,0){$\vector(-1,0){10}$}
\put(70,0){$\vector(1,0){10}$}
\put(130,0){$\vector(-1,0){10}$}
\put(180,0){$\vector(-1,0){10}$}
\put(20,50){$\vector(1,0){10}$}
\put(70,50){$\vector(1,0){10}$}
\put(130,50){$\vector(-1,0){10}$}
\put(170,50){$\vector(1,0){10}$}
\put(30,100){$\vector(-1,0){10}$}
\put(70,100){$\vector(1,0){10}$}
\put(130,100){$\vector(-1,0){10}$}
\put(180,100){$\vector(-1,0){10}$}
\put(0,18){\vector(0,1){14}}
\put(0,82){\vector(0,-1){14}}
\put(50,32){\vector(0,-1){14}}
\put(50,68){\vector(0,1){14}}
\put(100,-8){\vector(0,-1){9}}
\put(100,8){\vector(0,1){9}}
\put(100,42){\vector(0,-1){9}}
\put(100,58){\vector(0,1){9}}
\put(100,92){\vector(0,-1){9}}
\put(100,108){\vector(0,1){9}}
\put(150,18){\vector(0,1){14}}
\put(150,82){\vector(0,-1){14}}
\put(200,32){\vector(0,-1){14}}
\put(200,68){\vector(0,1){14}}

\put(120,-15){\vector(2,1){10}}
\put(120,15){\vector(2,-1){10}}
\put(120,85){\vector(2,1){10}}
\put(120,115){\vector(2,-1){10}}
\put(80,35){\vector(-2,1){10}}
\put(80,65){\vector(-2,-1){10}}
\put(-35,-1){$\dots$}
\put(-35,49){$\dots$}
\put(-35,99){$\dots$}
\put(225,-1){$\dots$}
\put(225,49){$\dots$}
\put(225,99){$\dots$}
\dottedline(-20,-40)(220,-40)
\put(-45,-42){$\updownarrow$}
\put(-60,-42){$\mu^{\bullet}_-$}
}

\end{picture}
\caption{
(Continues from Figure \ref{fig:labelyB1}).
}
\label{fig:labelyB2}
\end{figure}

\subsection{T-system and cluster algebra}
The T-system $\mathbb{T}_{\ell}(B_r)$
naturally appears as a system of relations
among the cluster variables
$x_{\mathbf{i}}(u)$ 
in the trivial evaluation of coefficients.
(The quiver $Q_{\ell}(B_r)$ is designed to do so.)
Let $\mathcal{A}(B,x)$ be the cluster algebra
with trivial coefficients, where $(B,x)$ is
the initial seed.
Let $\mathbf{1}=\{1\}$ 
be the {\em trivial semifield}
and $\pi_{\mathbf{1}}:
\mathbb{Q}_{\mathrm{sf}}(y)\rightarrow 
\mathbf{1}$, $y_{\mathbf{i}}\mapsto 1$ be the projection.
Let $[x_{\mathbf{i}}(u)]_{\mathbf{1}}$
denote the image of $x_{\mathbf{i}}(u)$
 by the algebra homomorphism
$\mathcal{A}(B,x,y)\rightarrow \mathcal{A}(B,x)$
 induced from $\pi_{\mathbf{1}}$.
It is called the {\em trivial evaluation}.

Recall that $G(b,k,v;a,m,u)$ is defined in \eqref{eq:Tu}.

\begin{lemma}
\label{lem:x2}
The family $\{x^{(a)}_m(u)
\mid (a,m,u)\in \mathcal{I}_{\ell+}\}$
satisfies a system of relations
\begin{align}
\begin{split}
x^{(a)}_{m}\left(u-\textstyle\frac{1}{t_a}\right)
x^{(a)}_{m}\left(u+\textstyle\frac{1}{t_a}\right)
&=
\frac{y^{(a)}_m(u)}{1+y^{(a)}_m(u)}
\prod_{(b,k,v)\in \mathcal{I}_{\ell+}}
x^{(b)}_{k}(v)^{G(b,k,v;\, a,m,u)}\\
&\qquad +
\frac{1}{1+y^{(a)}_m(u)}
x^{(a)}_{m-1}(u)x^{(a)}_{m+1}(u),
\end{split}
\end{align}
where $(a,m,u)\in \mathcal{I}'_{+\ell}$.
In particular,
the family $\{ [x^{(a)}_m(u)]_{\mathbf{1}}
\mid (a,m,u)\in \mathcal{I}_{\ell+}\}$
satisfies the T-system $\mathbb{T}_{\ell}(B_r)$
in $\mathcal{A}(B,x)$
by replacing $T^{(a)}_m(u)$ with $[x^{(a)}_m(u)]_{\mathbf{1}}$.
\end{lemma}
\begin{proof}
An easy way to prove it is
to represent all the relevant cluster variables in
 Figures
\ref{fig:labelxB1} and \ref{fig:labelxB2}
by $x^{(a)}_m(u)$
$(a,m,u)\in \mathcal{I}_{\ell+}$,
then to apply  mutations
at $(\mathbf{i},u):\mathbf{p}_+$
 in the figures.
For example,
consider the mutation at
$((r,3),0):\mathbf{p}_+$.
Then, by the exchange relation,
 $x^{(r)}_3(-\frac{1}{2})$ is
mutated to
\begin{align}
\begin{split}
\frac{1}{
x^{(r)}_{3}\left(\textstyle -\frac{1}{2}\right)
}
\left\{
\frac{y^{(r)}_3(0)}{1+y^{(r)}_3(0)}
x^{(r-1)}_{1}(0)x^{(r-1)}_{2}(0)
+
\frac{1}{1+y^{(r)}_3(0)}
x^{(r)}_{2}(0)x^{(r)}_{4}(0)
\right\}
,
\end{split}
\end{align}
which should be equal to $x^{(r)}_3(\frac{1}{2})$.
We note that  Figures
\ref{fig:labelxB1} and \ref{fig:labelxB2}
are general enough for that purpose.
\end{proof}

\begin{definition}
The {\em T-subalgebra
$\mathcal{A}_T(B,x)$
of ${\mathcal{A}}(B,x)$
associated with the sequence \eqref{eq:mutseq}}
is the subalgebra of
${\mathcal{A}}(B,x)$
generated by
$[x_{\mathbf{i}}(u)]_{\mathbf{1}}$
($(\mathbf{i},u)\in \mathbf{I}\times \frac{1}{2}\mathbb{Z}$).
\end{definition}

\begin{theorem}
\label{thm:Tiso}
The ring $\EuScript{T}^{\circ}_{\ell}(B_r)_+$ is isomorphic to
$\mathcal{A}_T(B,x)$ by the correspondence
$T^{(a)}_m(u)\mapsto [x^{(a)}_m(u)]_{\mathbf{1}}$.
\end{theorem}
\begin{proof}
First we note that $\mathcal{A}_T(B,x)$ is generated
by
$[x_{\mathbf{i}}(u)]_{\mathbf{1}}$
($(\mathbf{i},u): \mathbf{p}_+$) by
\eqref{eq:xmu1} and \eqref{eq:xmu2}.
Then, the claim follows from Lemma \ref{lem:x2} in the same way
as \cite[Proposition 4.24]{IIKNS}.
\end{proof}

\subsection{Y-system and cluster algebra}

The Y-system $\mathbb{Y}_{\ell}(B_r)$
also naturally appears as a system of relations
among the coefficients
$y_{\mathbf{i}}(u)$.

\begin{lemma}
\label{lem:y2}
The family $\{ y^{(a)}_m(u)
\mid (a,m,u)\in \mathcal{I}'_{\ell+}\}$
satisfies the Y-system $\mathbb{Y}_{\ell}(B_r)$
by replacing $Y^{(a)}_m(u)$ with $y^{(a)}_m(u)$.
\end{lemma}
\begin{proof}
This can be easily shown
 using Figures \ref{fig:labelyB1} and \ref{fig:labelyB2}.
For example,
consider the mutation at
$((r-1,2),0):\mathbf{p}_+$.
The coefficient $y^{(r-1)}_2(0)$
is mutated to $y^{(r-1)}_2(0)^{-1}$.
Then, at $u=\frac{1}{2}, 1,\frac{3}{2}$,
the factors 
\begin{align}
\begin{split}
{\textstyle
1+y^{(r)}_4(\frac{1}{2})},\quad
\frac{y^{(r-1)}_3(1)}{1+y^{(r-1)}_1(1)},\quad
\frac{y^{(r-1)}_3(1)}{1+y^{(r-1)}_1(1)},\\
\textstyle
1+y^{(r)}_3(1),\quad
1+y^{(r)}_5(1),\quad
1+y^{(r)}_4(\frac{3}{2})
\end{split}
\end{align}
are multiplied to $y^{(r-1)}_2(0)^{-1}$.
The result should be equal to $y^{(r-1)}_2(2)$.
\end{proof}

\begin{definition}
The {\em Y-subgroup
$\mathcal{G}_Y(B,y)$
of ${\mathcal{G}}(B,y)$
associated with the sequence \eqref{eq:mutseq}}
is the subgroup of
${\mathcal{G}}(B,y)$ 
generated by
$y_{\mathbf{i}}(u)$
($(\mathbf{i},u)\in \mathbf{I}\times \frac{1}{2}\mathbb{Z}$)
and $1+y_{\mathbf{i}}(u)$
($(\mathbf{i},u):\mathbf{p}_+$ or 
$\mathbf{p}_-$).
\end{definition}
Notice that we excluded $1+y_{\mathbf{i}}(u)$
for $(\mathbf{i},u)$
 not satisfying $\mathbf{p}_+$
nor $\mathbf{p}_-$.
This is because such $(\mathbf{i},u)$
is not a mutation point so that
the factor $1+y_{\mathbf{i}}(u)$
does not appear anywhere for the mutation
sequence  \eqref{eq:mutseq}.

\begin{theorem}
\label{thm:Yiso}
The group $\EuScript{Y}^{\circ}_{\ell}(B_r)_+$ is isomorphic to
$\mathcal{G}_Y(B,y)$ by the correspondence
$Y^{(a)}_m(u)\mapsto y^{(a)}_m(u)$
and $1+Y^{(a)}_m(u)\mapsto 1+y^{(a)}_m(u)$.
\end{theorem}
\begin{proof}
We note that $\mathcal{G}_Y(B,y)$ is generated
by
$y_{\mathbf{i}}(u)$, $1+y_{\mathbf{i}}(u)$
($(\mathbf{i},u): \mathbf{p}_+$) by
\eqref{eq:ymu}.
Then, the claim follows from Lemma \ref{lem:y2} in the same way
as \cite[Theorem 6.19]{KNS3}.
\end{proof}

\section{Tropical Y-system at level 2}

In this section we study the tropical version of
the Y-system at level 2.

\subsection{Tropical Y-system}

Let $y=y(0)$ be the initial coefficient tuple
of the cluster algebra $\mathcal{A}(B,x,y)$
with $B=B_{\ell}(B_r)$ in the previous section.
Let 
$\mathrm{Trop}(y)$ be
the {\em tropical semifield} for $y$.
Let $\pi_{\mathbf{T}}:
\mathbb{Q}_{\mathrm{sf}}(y)\rightarrow 
\mathrm{Trop}(y)$, $y_{\mathbf{i}}\mapsto 
y_{\mathbf{i}}$ be the projection.
Let $[y_{\mathbf{i}}(u)]_{\mathbf{T}}$
and $[\mathcal{G}_Y(B,y)]_{\mathbf{T}}$
denote the images of $y_{\mathbf{i}}(u)$
and $\mathcal{G}_Y(B,y)$
by the multiplicative group
 homomorphism induced from $\pi_{\mathbf{T}}$, respectively.
They are called the {\em tropical evaluations},
and the resulting relations in
the group $[\mathcal{G}_Y(B,y)]_{\mathbf{T}}$
is called the {\em tropical Y-system}.
They are first studied in \cite{FZ2}
for simply laced type at level 2 in our
terminology.

We say a (Laurent) monomial $m=\prod_{\mathbf{i}\in \mathbf{I}}
y_{\mathbf{i}}^{k_{\mathbf{i}}}$
is {\em positive} (resp. {\em negative})
if $m\neq 1$ and $k_{\mathbf{i}}\geq 0$
(resp. $k_{\mathbf{i}}\leq 0$)
for any $\mathbf{i}$.

The next `tropical mutation rule' for $[y_{\mathbf{i}}(u)]_{\mathbf{T}}$
is general and  useful.
\begin{lemma}
\label{lem:mutation}
Suppose that $y''$ is the coefficient tuple obtained from
the mutation of another coefficient tuple $y'$ at $\mathbf{k}$
with mutation matrix $B'$. Then, for any $\mathbf{i}\neq \mathbf{k}$,
we have the rule:
\par
(i) $[y''_{\mathbf{i}}]_{\mathbf{T}} =
 [y'_{\mathbf{i}}]_{\mathbf{T}}[y'_{\mathbf{k}}]_{\mathbf{T}}$
 if one of the following conditions holds.
\par\quad
(a) $B'_{\mathbf{k}\mathbf{i}}> 0$,
 and $[y'_{\mathbf{k}}]_{\mathbf{T}}$ is 
positive.
\par\quad
(b)  $B'_{\mathbf{k}\mathbf{i}}< 0$,
 and $[y'_{\mathbf{k}}]_{\mathbf{T}}$ is 
negative.
\par
(ii)
$[y''_{\mathbf{i}}]_{\mathbf{T}} =
 [y'_{\mathbf{i}}]_{\mathbf{T}}$
 if one of the following conditions holds.
\par\quad
(a) $B'_{\mathbf{k}\mathbf{i}}=0$.
\par\quad
(b)
 $B'_{\mathbf{k}\mathbf{i}}> 0$,
 and $[y'_{\mathbf{k}}]_{\mathbf{T}}$ is 
 negative.
\par\quad
(c) $B'_{\mathbf{k}\mathbf{i}}< 0$,
 and $[y'_{\mathbf{k}}]_{\mathbf{T}}$ is
 positive.
\end{lemma}
\begin{proof}
This is an immediate consequence of the exchange relation
\eqref{eq:coef} and \eqref{eq:trop}.
\end{proof}

The following properties
of the tropical Y-system at level 2
will be the key in the entire method.

\begin{proposition}
\label{prop:lev2}
 For 
$[\mathcal{G}_Y(B,y)]_{\mathbf{T}}$
with $B=B_{2}(B_r)$, the following facts hold.
\par
(i) Let $u$ be in the region $0\le u < 2$.
For any $(\mathbf{i},u):\mathbf{p}_+$,
the  monomial $[y_{\mathbf{i}}(u)]_{\mathbf{T}}$
is positive.

\par
(ii) Let $u$ be in the region $-h^{\vee}\le u < 0$.
\begin{itemize}
\item[\em (a)] Let $\mathbf{i}\in \mathbf{I}^{\circ}
\sqcup \mathbf{I}^{\bullet}_-
$.
For any $(\mathbf{i},u):\mathbf{p}_+$,
the  monomial $[y_{\mathbf{i}}(u)]_{\mathbf{T}}$
is negative.
\item[\em (b)]
 Let $\mathbf{i}\in \mathbf{I}^{\bullet}_+$.
For any $(\mathbf{i},u):\mathbf{p}_+$,
the  monomial $[y_{\mathbf{i}}(u)]_{\mathbf{T}}$
is negative for $u=-1,-3,\dots$
and positive for $u=-2,-4,\dots$.
\end{itemize}
\par
(iii)
$y_{ii'}(2)=y_{ii'}^{-1}$ if $i\neq r$ and
$y_{r,4-i'}^{-1}$ if $i=r$.
\par
(iv) $y_{ii'}(-h^{\vee})=
y_{2r-i,i'}^{-1}$.
\end{proposition}

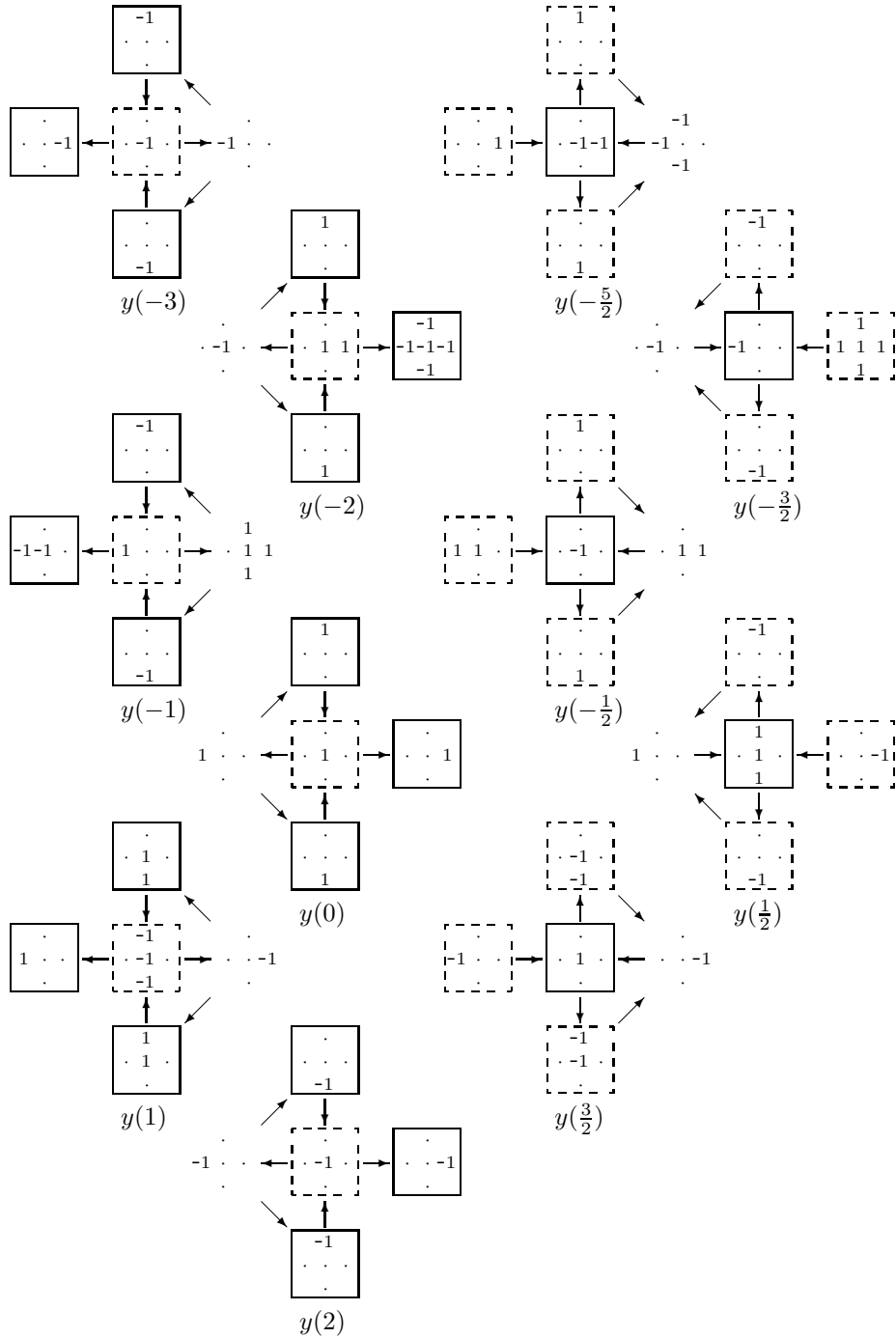
\begin{figure}
\begin{picture}(355,530)(-15,-75)
\put(0,340)
{
\put(37,77.5){\framebox(26,26)[c]}
\put(46,97){-$\scriptstyle 1$}
\put(41,88){$\scriptstyle \cdot$}
\put(49,88){$\scriptstyle \cdot$}
\put(57,88){$\scriptstyle \cdot$}
\put(49,79){$\scriptstyle \cdot$}
\put(-3,37.5){\framebox(26,26)[c]}
\put(9,57){$\scriptstyle \cdot$}
\put(1,48){$\scriptstyle \cdot$}
\put(9,48){$\scriptstyle \cdot$}
\put(14,48){-$\scriptstyle  1$}
\put(9,39){$\scriptstyle \cdot$}
%
\put(37,37.5){\dashbox{3}(26,26)[c]}
\put(49,57){$\scriptstyle \cdot$}
\put(41,48){$\scriptstyle \cdot$}
\put(46,48){-$\scriptstyle 1$}
\put(57,48){$\scriptstyle \cdot$}
\put(49,39){$\scriptstyle \cdot$}
%
%
\put(89,57){$\scriptstyle \cdot$}
\put(78,48){-$\scriptstyle 1$}
\put(89,48){$\scriptstyle \cdot$}
\put(97,48){$\scriptstyle \cdot$}
\put(89,39){$\scriptstyle \cdot$}
%
%
\put(37,-2.5){\framebox(26,26)[c]}
\put(49,17){$\scriptstyle \cdot$}
\put(41,8){$\scriptstyle \cdot$}
\put(49,8){$\scriptstyle \cdot$}
\put(57,8){$\scriptstyle \cdot$}
\put(46,-1){-$\scriptstyle 1$}
%
\put(50,75){\vector(0,-1){10}}
\put(35,50){\vector(-1,0){10}}
\put(65,50){\vector(1,0){10}}
\put(50,25){\vector(0,1){10}}
\put(75,65){\vector(-1,1){10}}
\put(75,35){\vector(-1,-1){10}}
\put(40,-15){$y(-3)$}
}
\put(170,340)
{
\put(37,77.5){\dashbox{3}(26,26)[c]}
\put(48,97){$\scriptstyle 1$}
\put(41,88){$\scriptstyle \cdot$}
\put(49,88){$\scriptstyle \cdot$}
\put(57,88){$\scriptstyle \cdot$}
\put(49,79){$\scriptstyle \cdot$}
\put(-3,37.5){\dashbox{3}(26,26)[c]}
\put(9,57){$\scriptstyle \cdot$}
\put(1,48){$\scriptstyle \cdot$}
\put(9,48){$\scriptstyle \cdot$}
\put(16,48){$\scriptstyle  1$}
\put(9,39){$\scriptstyle \cdot$}
%
\put(37,37.5){\framebox(26,26)[c]}
\put(49,57){$\scriptstyle \cdot$}
\put(41,48){$\scriptstyle \cdot$}
\put(46,48){-$\scriptstyle 1$}
\put(54,48){-$\scriptstyle 1$}
\put(49,39){$\scriptstyle \cdot$}
%
%
\put(86,57){-$\scriptstyle 1$}
\put(78,48){-$\scriptstyle 1$}
\put(89,48){$\scriptstyle \cdot$}
\put(97,48){$\scriptstyle \cdot$}
\put(86,39){-$\scriptstyle 1$}
%
%
\put(37,-2.5){\dashbox{3}(26,26)[c]}
\put(49,17){$\scriptstyle \cdot$}
\put(41,8){$\scriptstyle \cdot$}
\put(49,8){$\scriptstyle \cdot$}
\put(57,8){$\scriptstyle \cdot$}
\put(48,-1){$\scriptstyle 1$}
%
\put(50,65){\vector(0,1){10}}
\put(50,35){\vector(0,-1){10}}
\put(25,50){\vector(1,0){10}}
\put(75,50){\vector(-1,0){10}}
\put(65,75){\vector(1,-1){10}}
\put(65,25){\vector(1,1){10}}
\put(40,-15){$y(-\frac{5}{2})$}
}
\put(70,260)
{
\put(37,77.5){\framebox(26,26)[c]}
\put(48,97){$\scriptstyle 1$}
\put(41,88){$\scriptstyle \cdot$}
\put(49,88){$\scriptstyle \cdot$}
\put(57,88){$\scriptstyle \cdot$}
\put(49,79){$\scriptstyle \cdot$}
%
%
\put(9,57){$\scriptstyle \cdot$}
\put(1,48){$\scriptstyle \cdot$}
\put(6,48){-$\scriptstyle 1$}
\put(17,48){$\scriptstyle \cdot$}
\put(9,39){$\scriptstyle \cdot$}
%
\put(37,37.5){\dashbox{3}(26,26)[c]}
\put(49,57){$\scriptstyle \cdot$}
\put(41,48){$\scriptstyle \cdot$}
\put(48,48){$\scriptstyle 1$}
\put(56,48){$\scriptstyle 1$}
\put(49,39){$\scriptstyle \cdot$}
%
%
\put(77,37.5){\framebox(26,26)[c]}
\put(86,57){-$\scriptstyle 1$}
\put(78,48){-$\scriptstyle 1$}
\put(86,48){-$\scriptstyle 1$}
\put(94,48){-$\scriptstyle 1$}
\put(86,39){-$\scriptstyle 1$}
%
%
\put(37,-2.5){\framebox(26,26)[c]}
\put(49,17){$\scriptstyle \cdot$}
\put(41,8){$\scriptstyle \cdot$}
\put(49,8){$\scriptstyle \cdot$}
\put(57,8){$\scriptstyle \cdot$}
\put(48,-1){$\scriptstyle 1$}
%
\put(50,75){\vector(0,-1){10}}
\put(35,50){\vector(-1,0){10}}
\put(65,50){\vector(1,0){10}}
\put(50,25){\vector(0,1){10}}
\put(25,65){\vector(1,1){10}}
\put(25,35){\vector(1,-1){10}}
\put(40,-15){$y(-2)$}
}
\put(240,260)
{
\put(37,77.5){\dashbox{3}(26,26)[c]}
\put(46,97){-$\scriptstyle 1$}
\put(41,88){$\scriptstyle \cdot$}
\put(49,88){$\scriptstyle \cdot$}
\put(57,88){$\scriptstyle \cdot$}
\put(49,79){$\scriptstyle \cdot$}
%
%
\put(9,57){$\scriptstyle \cdot$}
\put(1,48){$\scriptstyle \cdot$}
\put(6,48){-$\scriptstyle 1$}
\put(17,48){$\scriptstyle \cdot$}
\put(9,39){$\scriptstyle \cdot$}
%
\put(37,37.5){\framebox(26,26)[c]}
\put(49,57){$\scriptstyle \cdot$}
\put(38,48){-$\scriptstyle 1$}
\put(49,48){$\scriptstyle \cdot$}
\put(57,48){$\scriptstyle \cdot$}
\put(49,39){$\scriptstyle \cdot$}
%
%
\put(77,37.5){\dashbox{3}(26,26)[c]}
\put(88,57){$\scriptstyle 1$}
\put(80,48){$\scriptstyle 1$}
\put(88,48){$\scriptstyle 1$}
\put(96,48){$\scriptstyle 1$}
\put(88,39){$\scriptstyle 1$}
%
%
\put(37,-2.5){\dashbox{3}(26,26)[c]}
\put(49,17){$\scriptstyle \cdot$}
\put(41,8){$\scriptstyle \cdot$}
\put(49,8){$\scriptstyle \cdot$}
\put(57,8){$\scriptstyle \cdot$}
\put(46,-1){-$\scriptstyle 1$}
%
\put(25,50){\vector(1,0){10}}
\put(75,50){\vector(-1,0){10}}
\put(50,65){\vector(0,1){10}}
\put(50,35){\vector(0,-1){10}}
\put(35,75){\vector(-1,-1){10}}
\put(35,25){\vector(-1,1){10}}
\put(40,-15){$y(-\frac{3}{2})$}
}
\put(0,180)
{
\put(37,77.5){\framebox(26,26)[c]}
\put(46,97){-$\scriptstyle 1$}
\put(41,88){$\scriptstyle \cdot$}
\put(49,88){$\scriptstyle \cdot$}
\put(57,88){$\scriptstyle \cdot$}
\put(49,79){$\scriptstyle \cdot$}
\put(-3,37.5){\framebox(26,26)[c]}
\put(9,57){$\scriptstyle \cdot$}
\put(-2,48){-$\scriptstyle 1$}
\put(6,48){-$\scriptstyle 1$}
\put(17,48){$\scriptstyle \cdot$}
\put(9,39){$\scriptstyle \cdot$}
%
\put(37,37.5){\dashbox{3}(26,26)[c]}
\put(49,57){$\scriptstyle \cdot$}
\put(40,48){$\scriptstyle 1$}
\put(49,48){$\scriptstyle \cdot$}
\put(57,48){$\scriptstyle \cdot$}
\put(49,39){$\scriptstyle \cdot$}
%
%
\put(88,57){$\scriptstyle 1$}
\put(81,48){$\scriptstyle \cdot$}
\put(88,48){$\scriptstyle 1$}
\put(96,48){$\scriptstyle 1$}
\put(88,39){$\scriptstyle 1$}
%
%
\put(37,-2.5){\framebox(26,26)[c]}
\put(49,17){$\scriptstyle \cdot$}
\put(41,8){$\scriptstyle \cdot$}
\put(49,8){$\scriptstyle \cdot$}
\put(57,8){$\scriptstyle \cdot$}
\put(46,-1){-$\scriptstyle 1$}
%
\put(50,75){\vector(0,-1){10}}
\put(35,50){\vector(-1,0){10}}
\put(65,50){\vector(1,0){10}}
\put(50,25){\vector(0,1){10}}
\put(75,65){\vector(-1,1){10}}
\put(75,35){\vector(-1,-1){10}}
\put(40,-15){$y(-1)$}
}
\put(170,180)
{
\put(37,77.5){\dashbox{3}(26,26)[c]}
\put(48,97){$\scriptstyle 1$}
\put(41,88){$\scriptstyle \cdot$}
\put(49,88){$\scriptstyle \cdot$}
\put(57,88){$\scriptstyle \cdot$}
\put(49,79){$\scriptstyle \cdot$}
\put(-3,37.5){\dashbox{3}(26,26)[c]}
\put(9,57){$\scriptstyle \cdot$}
\put(0,48){$\scriptstyle 1$}
\put(8,48){$\scriptstyle 1$}
\put(17,48){$\scriptstyle \cdot$}
\put(9,39){$\scriptstyle \cdot$}
%
\put(37,37.5){\framebox(26,26)[c]}
\put(49,57){$\scriptstyle \cdot$}
\put(41,48){$\scriptstyle \cdot$}
\put(46,48){-$\scriptstyle 1$}
\put(57,48){$\scriptstyle \cdot$}
\put(49,39){$\scriptstyle \cdot$}
%
%
\put(89,57){$\scriptstyle \cdot$}
\put(81,48){$\scriptstyle \cdot$}
\put(88,48){$\scriptstyle 1$}
\put(96,48){$\scriptstyle 1$}
\put(89,39){$\scriptstyle \cdot$}
%
%
\put(37,-2.5){\dashbox{3}(26,26)[c]}
\put(49,17){$\scriptstyle \cdot$}
\put(41,8){$\scriptstyle \cdot$}
\put(49,8){$\scriptstyle \cdot$}
\put(57,8){$\scriptstyle \cdot$}
\put(48,-1){$\scriptstyle 1$}
%
\put(50,65){\vector(0,1){10}}
\put(50,35){\vector(0,-1){10}}
\put(25,50){\vector(1,0){10}}
\put(75,50){\vector(-1,0){10}}
\put(65,75){\vector(1,-1){10}}
\put(65,25){\vector(1,1){10}}
\put(40,-15){$y(-\frac{1}{2})$}
}
\put(70,100)
{
\put(37,77.5){\framebox(26,26)[c]}
\put(48,97){$\scriptstyle 1$}
\put(41,88){$\scriptstyle \cdot$}
\put(49,88){$\scriptstyle \cdot$}
\put(57,88){$\scriptstyle \cdot$}
\put(49,79){$\scriptstyle \cdot$}
%
%
\put(9,57){$\scriptstyle \cdot$}
\put(0,48){$\scriptstyle 1$}
\put(9,48){$\scriptstyle \cdot$}
\put(17,48){$\scriptstyle \cdot$}
\put(9,39){$\scriptstyle \cdot$}
%
\put(37,37.5){\dashbox{3}(26,26)[c]}
\put(49,57){$\scriptstyle \cdot$}
\put(41,48){$\scriptstyle \cdot$}
\put(48,48){$\scriptstyle 1$}
\put(57,48){$\scriptstyle \cdot$}
\put(49,39){$\scriptstyle \cdot$}
%
%
\put(77,37.5){\framebox(26,26)[c]}
\put(89,57){$\scriptstyle \cdot$}
\put(81,48){$\scriptstyle \cdot$}
\put(89,48){$\scriptstyle \cdot$}
\put(96,48){$\scriptstyle 1$}
\put(89,39){$\scriptstyle \cdot$}
%
%
\put(37,-2.5){\framebox(26,26)[c]}
\put(49,17){$\scriptstyle \cdot$}
\put(41,8){$\scriptstyle \cdot$}
\put(49,8){$\scriptstyle \cdot$}
\put(57,8){$\scriptstyle \cdot$}
\put(48,-1){$\scriptstyle 1$}
%
\put(50,75){\vector(0,-1){10}}
\put(35,50){\vector(-1,0){10}}
\put(65,50){\vector(1,0){10}}
\put(50,25){\vector(0,1){10}}
\put(25,65){\vector(1,1){10}}
\put(25,35){\vector(1,-1){10}}
\put(40,-15){$y(0)$}
}
\put(240,100)
{
\put(37,77.5){\dashbox{3}(26,26)[c]}
\put(46,97){-$\scriptstyle 1$}
\put(41,88){$\scriptstyle \cdot$}
\put(49,88){$\scriptstyle \cdot$}
\put(57,88){$\scriptstyle \cdot$}
\put(49,79){$\scriptstyle \cdot$}
%
%
\put(9,57){$\scriptstyle \cdot$}
\put(0,48){$\scriptstyle 1$}
\put(9,48){$\scriptstyle \cdot$}
\put(17,48){$\scriptstyle  \cdot$}
\put(9,39){$\scriptstyle \cdot$}
%
\put(37,37.5){\framebox(26,26)[c]}
\put(48,57){$\scriptstyle 1$}
\put(41,48){$\scriptstyle \cdot$}
\put(48,48){$\scriptstyle 1$}
\put(57,48){$\scriptstyle \cdot$}
\put(48,39){$\scriptstyle 1$}
%
%
\put(77,37.5){\dashbox{3}(26,26)[c]}
\put(89,57){$\scriptstyle \cdot$}
\put(81,48){$\scriptstyle \cdot$}
\put(89,48){$\scriptstyle \cdot$}
\put(94,48){-$\scriptstyle 1$}
\put(89,39){$\scriptstyle \cdot$}
%
%
\put(37,-2.5){\dashbox{3}(26,26)[c]}
\put(49,17){$\scriptstyle \cdot$}
\put(41,8){$\scriptstyle \cdot$}
\put(49,8){$\scriptstyle \cdot$}
\put(57,8){$\scriptstyle \cdot$}
\put(46,-1){-$\scriptstyle 1$}
%
\put(25,50){\vector(1,0){10}}
\put(75,50){\vector(-1,0){10}}
\put(50,65){\vector(0,1){10}}
\put(50,35){\vector(0,-1){10}}
\put(35,75){\vector(-1,-1){10}}
\put(35,25){\vector(-1,1){10}}
\put(40,-15){$y(\frac{1}{2})$}
}
\put(0,20)
{
\put(37,77.5){\framebox(26,26)[c]}
\put(49,97){$\scriptstyle \cdot$}
\put(41,88){$\scriptstyle \cdot$}
\put(48,88){$\scriptstyle 1$}
\put(57,88){$\scriptstyle \cdot$}
\put(48,79){$\scriptstyle 1$}
\put(-3,37.5){\framebox(26,26)[c]}
\put(9,57){$\scriptstyle \cdot$}
\put(0,48){$\scriptstyle 1$}
\put(9,48){$\scriptstyle \cdot$}
\put(16,48){$\scriptstyle \cdot$}
\put(9,39){$\scriptstyle \cdot$}
%
\put(37,37.5){\dashbox{3}(26,26)[c]}
\put(46,57){-$\scriptstyle 1$}
\put(41,48){$\scriptstyle \cdot$}
\put(46,48){-$\scriptstyle 1$}
\put(57,48){$\scriptstyle \cdot$}
\put(46,39){-$\scriptstyle 1$}
%
%
\put(89,57){$\scriptstyle \cdot$}
\put(81,48){$\scriptstyle \cdot$}
\put(89,48){$\scriptstyle \cdot$}
\put(94,48){-$\scriptstyle 1$}
\put(89,39){$\scriptstyle \cdot$}
%
%
\put(37,-2.5){\framebox(26,26)[c]}
\put(48,17){$\scriptstyle 1$}
\put(41,8){$\scriptstyle \cdot$}
\put(48,8){$\scriptstyle 1$}
\put(57,8){$\scriptstyle \cdot$}
\put(49,-1){$\scriptstyle \cdot$}
%
\put(50,75){\vector(0,-1){10}}
\put(35,50){\vector(-1,0){10}}
\put(65,50){\vector(1,0){10}}
\put(50,25){\vector(0,1){10}}
\put(75,65){\vector(-1,1){10}}
\put(75,35){\vector(-1,-1){10}}
\put(40,-15){$y(1)$}
}
\put(170,20)
{
\put(37,77.5){\dashbox{3}(26,26)[c]}
\put(49,97){$\scriptstyle \cdot$}
\put(41,88){$\scriptstyle \cdot$}
\put(46,88){-$\scriptstyle 1$}
\put(57,88){$\scriptstyle \cdot$}
\put(46,79){-$\scriptstyle 1$}
\put(-3,37.5){\dashbox{3}(26,26)[c]}
\put(9,57){$\scriptstyle \cdot$}
\put(-2,48){-$\scriptstyle 1$}
\put(9,48){$\scriptstyle \cdot$}
\put(17,48){$\scriptstyle \cdot$}
\put(9,39){$\scriptstyle \cdot$}
%
\put(37,37.5){\framebox(26,26)[c]}
\put(49,57){$\scriptstyle \cdot$}
\put(41,48){$\scriptstyle \cdot$}
\put(48,48){$\scriptstyle 1$}
\put(57,48){$\scriptstyle \cdot$}
\put(49,39){$\scriptstyle \cdot$}
%
%
\put(89,57){$\scriptstyle \cdot$}
\put(81,48){$\scriptstyle \cdot$}
\put(89,48){$\scriptstyle \cdot$}
\put(94,48){-$\scriptstyle 1$}
\put(89,39){$\scriptstyle \cdot$}
%
%
\put(37,-2.5){\dashbox{3}(26,26)[c]}
\put(46,17){-$\scriptstyle 1$}
\put(41,8){$\scriptstyle \cdot$}
\put(46,8){-$\scriptstyle 1$}
\put(57,8){$\scriptstyle \cdot$}
\put(49,-1){$\scriptstyle \cdot$}
%
\put(50,65){\vector(0,1){10}}
\put(50,35){\vector(0,-1){10}}
\put(25,50){\vector(1,0){10}}
\put(75,50){\vector(-1,0){10}}
\put(65,75){\vector(1,-1){10}}
\put(65,25){\vector(1,1){10}}
\put(40,-15){$y(\frac{3}{2})$}
}
\put(70,-60)
{
\put(37,77.5){\framebox(26,26)[c]}
\put(49,97){$\scriptstyle \cdot$}
\put(41,88){$\scriptstyle \cdot$}
\put(49,88){$\scriptstyle \cdot$}
\put(57,88){$\scriptstyle \cdot$}
\put(46,79){-$\scriptstyle 1$}
%
%
\put(9,57){$\scriptstyle \cdot$}
\put(-2,48){-$\scriptstyle 1$}
\put(9,48){$\scriptstyle \cdot$}
\put(17,48){$\scriptstyle  \cdot$}
\put(9,39){$\scriptstyle \cdot$}
%
\put(37,37.5){\dashbox{3}(26,26)[c]}
\put(49,57){$\scriptstyle \cdot$}
\put(41,48){$\scriptstyle \cdot$}
\put(46,48){-$\scriptstyle 1$}
\put(57,48){$\scriptstyle \cdot$}
\put(49,39){$\scriptstyle \cdot$}
%
%
\put(77,37.5){\framebox(26,26)[c]}
\put(89,57){$\scriptstyle \cdot$}
\put(81,48){$\scriptstyle \cdot$}
\put(89,48){$\scriptstyle \cdot$}
\put(94,48){-$\scriptstyle 1$}
\put(89,39){$\scriptstyle \cdot$}
%
%
\put(37,-2.5){\framebox(26,26)[c]}
\put(46,17){-$\scriptstyle 1$}
\put(41,8){$\scriptstyle \cdot$}
\put(49,8){$\scriptstyle \cdot$}
\put(57,8){$\scriptstyle \cdot$}
\put(49,-1){$\scriptstyle \cdot$}
%
\put(50,75){\vector(0,-1){10}}
\put(35,50){\vector(-1,0){10}}
\put(65,50){\vector(1,0){10}}
\put(50,25){\vector(0,1){10}}
\put(25,65){\vector(1,1){10}}
\put(25,35){\vector(1,-1){10}}
\put(40,-15){$y(2)$}
}

\end{picture}
\caption{Tropical Y-system of type $B_2$ at level 2.
We continue to use the convention in
 Figures \ref{fig:labelxB1}--\ref{fig:labelyB2}
such that the variables framed by solid/dashed lines
satisfy the condition $\mathbf{p}_+$/$\mathbf{p}_-$.
}
\label{fig:tropB22}
\end{figure}

\begin{example}
Consider the simplest case $B_2$.
All the coefficients $[y_{\mathbf{i}}(u)]_{\mathbf{T}}$
in the region $-3 \leq u \leq 2$  are
calculated with Lemma \ref{lem:mutation} and
explicitly given in Figure \ref{fig:tropB22}.
We continue to use the convention in
 Figures \ref{fig:labelxB1}--\ref{fig:labelyB2}
such that the variables framed by solid/dashed lines
satisfy the condition $\mathbf{p}_+$/$\mathbf{p}_-$.
We recall that they are the mutation points for
the forward/backward direction of $u$.
In Figure \ref{fig:tropB22},
the configuration
\begin{align*}
\begin{picture}(50,23)(20,0)
\put(49,18){$\cdot$}
\put(36,8){ $\scriptstyle 1$}
\put(48,8){$\scriptstyle 1$}
\put(59,8){$\cdot$}
\put(49,-2){$\cdot$}
\end{picture}
\end{align*}
for example represents the monomial $y_{11}y_{22}$,
where $\cdot$ stands for 0.
One can observe all the properties in
Proposition \ref{prop:lev2}
in Figure \ref{fig:tropB22}.
Let us further observe that,
in the region $-3\leq u <0$,
we have six negative monomials
for  $\mathbf{i}\in \mathbf{I}^{\circ}
\sqcup \mathbf{I}^{\bullet}_-
$, $(\mathbf{i},u):\mathbf{p}_+$,
\begin{align*}
\begin{picture}(50,23)(20,0)
\put(49,18){$ \cdot$}
\put(39,8){$\cdot$}
\put(49,8){$\cdot$}
\put(56,8){-$\scriptstyle 1$}
\put(49,-2){$\cdot$}
\end{picture}
\begin{picture}(50,20)(20,0)
\put(49,18){$\cdot$}
\put(39,8){$\cdot$}
\put(46,8){-$\scriptstyle 1$}
\put(56,8){-$\scriptstyle 1$}
\put(49,-2){$\cdot$}
\end{picture}
\begin{picture}(50,20)(20,0)
\put(46,18){-$\scriptstyle 1$}
\put(36,8){-$\scriptstyle 1$}
\put(46,8){-$\scriptstyle 1$}
\put(56,8){-$\scriptstyle 1$}
\put(46,-2){-$\scriptstyle 1$}
\end{picture}
\begin{picture}(50,20)(20,0)
\put(49,18){$\cdot$}
\put(36,8){-$\scriptstyle 1$}
\put(49,8){$\cdot$}
\put(59,8){$\cdot$}
\put(49,-2){$\cdot$}
\end{picture}
\begin{picture}(50,20)(20,0)
\put(49,18){$\cdot$}
\put(39,8){$\cdot$}
\put(46,8){-$\scriptstyle 1$}
\put(59,8){$\cdot$}
\put(49,-2){$\cdot$}
\end{picture}
\begin{picture}(50,20)(20,0)
\put(49,18){$\cdot$}
\put(36,8){-$\scriptstyle 1$}
\put(46,8){-$\scriptstyle 1$}
\put(59,8){$\cdot$}
\put(49,-2){$\cdot$}
\end{picture}
\end{align*}
{\em If we concentrate on the middle row},
then they naturally correspond to the {\em positive roots} of type $A_3$
\begin{align}
\alpha_1,
\alpha_2,
\alpha_3,
\alpha_1+\alpha_2,
\alpha_2+\alpha_3,
\alpha_1+\alpha_2+\alpha_3.
\end{align}
\end{example}

\subsection{Proof of Proposition \ref{prop:lev2}}
For general $r$,
one can directly verify  (i) and (iii)
as in Figure \ref{fig:tropB22}.

Note that (ii) and (iv) can be proved
independently for each variable $y_{\mathbf{i}}$.
(To be precise, we also need to assure that
each monomial is not 1 in total.
However, this can be easily followed up, so that we do not
describe details here.)
As for the powers of  variables
$y_{r1}$ and $y_{r3}$,
it is easy to verify the claim by direct calculations.
Therefore, it is enough to prove   (ii)
 and (iv)
only for the powers of variables
$y_{i1}$ ($i\neq r$) and $y_{r2}$.
To do that, we use the description of  the tropical Y-system
in the region $-h^{\vee}\leq u <0$ by the {\em root system
of type $A_{2r-1}$}, following the spirit of  \cite{FZ2}.

Let $A_{2r-1}$ be the Dynkin diagram of type $A$
with index set $J=\{1,\dots,2r-1\}$.
We assign 
the sign +/$-$ to vertices
(except  for $r$)
 of $A_{2r-1}$ as inherited from
$Q_2(B_r)$.
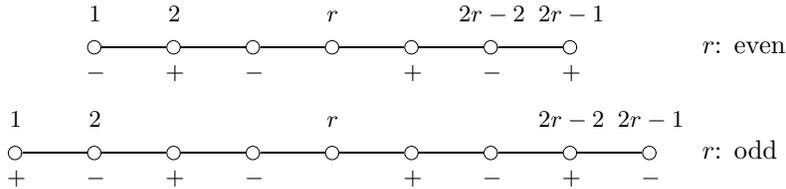
\begin{align*}
\begin{picture}(310,68)(0,20)
\put(45,70)
{
\put(0,0){\circle{5}}
\put(30,0){\circle{5}}
\put(60,0){\circle{5}}
\put(90,0){\circle{5}}
\put(120,0){\circle{5}}
\put(150,0){\circle{5}}
\put(180,0){\circle{5}}
\put(27,0){\line(-1,0){24}}
\put(57,0){\line(-1,0){24}}
\put(87,0){\line(-1,0){24}}
\put(93,0){\line(1,0){24}}
\put(147,0){\line(-1,0){24}}
\put(177,0){\line(-1,0){24}}
\put(-4,8)
{
\put(2,2){\small $1$}
\put(32,2){\small $2$}
\put(62,2){\small }
\put(92,2){\small $r$}
\put(122,2){\small }
\put(142,2){\small $2r-2$}
\put(172,2){\small $2r-1$}
}
\put(-5,-14)
{
\put(2,2){\small $-$}
\put(32,2){\small $+$}
\put(62,2){\small $-$}
\put(122,2){\small $+$}
\put(152,2){\small $-$}
\put(182,2){\small $+$}
}
\put(230,-2){$r$: even}
}
%
%
\put(45,30)
{
\put(-30,0){\circle{5}}
\put(0,0){\circle{5}}
\put(30,0){\circle{5}}
\put(60,0){\circle{5}}
\put(90,0){\circle{5}}
\put(120,0){\circle{5}}
\put(150,0){\circle{5}}
\put(180,0){\circle{5}}
\put(210,0){\circle{5}}
\put(-3,0){\line(-1,0){24}}
\put(27,0){\line(-1,0){24}}
\put(57,0){\line(-1,0){24}}
\put(87,0){\line(-1,0){24}}
\put(93,0){\line(1,0){24}}
\put(147,0){\line(-1,0){24}}
\put(177,0){\line(-1,0){24}}
\put(207,0){\line(-1,0){24}}
\put(-4,8)
{
\put(-28,2){\small $1$}
\put(2,2){\small $2$}
\put(62,2){\small }
\put(92,2){\small $r$}
\put(122,2){\small }
\put(172,2){\small $2r-2$}
\put(202,2){\small $2r-1$}
}
\put(-5,-14)
{
\put(-28,2){\small $+$}
\put(2,2){\small $-$}
\put(32,2){\small $+$}
\put(62,2){\small $-$}
\put(122,2){\small $+$}
\put(152,2){\small $-$}
\put(182,2){\small $+$}
\put(212,2){\small $-$}
}
\put(230,-2){$r$: odd}
}
\end{picture}
\end{align*}

Let $\Pi=\{ \alpha_1,\dots,\alpha_{2r-1}\}$, $-\Pi$,
 $\Phi_+$ be the set of the simple roots,
the negative simple roots, the positive roots, respectively,
of type $A_{2r-1}$.
Following \cite{FZ2}, we introduce the {\em piecewise-linear
analogue} $\sigma_i$ of the simple reflection $s_i$,
acting on the set
of the {\em almost positive roots}
$\Phi_{\geq -1}=\Phi_{+}\sqcup (-\Pi)$,
by
\begin{align}
\begin{split}
\sigma_i(\alpha)&=s_i(\alpha),\quad  \alpha\in \Phi_+,\\
\sigma_i(-\alpha_j)&=
\begin{cases}
\alpha_j& j=i,\\
-\alpha_j&\mbox{otherwise}.\\
\end{cases}
\end{split}
\end{align}
Let
\begin{align}
\sigma_+=\prod_{i\in J_+} \sigma_i,\quad
\sigma_-=\prod_{i\in J_-} \sigma_i,
\end{align}
where $J_{\pm}$ is the set of the
vertices of $A_{2r-1}$
 with property $\pm$. We define $\sigma$ as the composition
\begin{align}
\label{eq:sigma}
\sigma=\sigma_{r}\sigma_-\sigma_{r}\sigma_+.
\end{align}

Let $\omega:i\rightarrow 2r-i$ be an involution on $J$.

\begin{lemma}
\label{lem:orbit}
 The following facts hold.
\par
 (i) For $i\in J_+$,
$\sigma^{k}(-\alpha_i)\in\Phi_+$, $(1\leq k\leq r-1)$,
$\sigma^r(-\alpha_i)=-\alpha_{\omega(i)}$.
\par
 (ii) For $i\in J_-$,
$\sigma^{k}(-\alpha_i)\in\Phi_+$, $(1\leq k\leq r)$,
$\sigma^{r+1}(-\alpha_i)=-\alpha_{\omega(i)}$.
\par
(iii) 
$\sigma^{k}(-\alpha_r)\in\Phi_+$, $(1\leq k\leq r-1)$,
$\sigma^r(-\alpha_r)=-\alpha_{r}$.
\par
(iv) 
$\sigma^{k}(\alpha_r)\in\Phi_+$, $(0\leq k\leq r-1)$,
$\sigma^r(\alpha_r)=\alpha_{r}$.\par
(v) The elements in $\Phi_+$ in (i)--(iv) exhaust the set $\Phi_+$,
thereby providing the orbit decomposition
of $\Phi_+$ by $\sigma$.
\end{lemma}
\begin{proof}
(i)--(iv). They are verified by explicitly calculating
$\sigma^{k}(-\alpha_i)$ ($i\in J$)
 and $\sigma^{k}(\alpha_r)$.
The example for $r=6$
is given in Table \ref{tab:orbit},
where
we use the notation
\begin{align}
[i,j]=\alpha_i+\cdots + \alpha_j
\quad (i < j),
\quad
[i]=\alpha_i.
\end{align}
In fact, it is not difficult to read off the general rule from 
this example.

(v). The total number of 
the elements in $\Phi_+$ in (i)--(iv) is
$r(2r-1)$, which coincides with $|\Phi_+|$.
\end{proof}

\begin{table}
\rotatebox{90}
{
\begin{minipage}{\textheight}
\setlength{\unitlength}{0.95pt}
\begin{picture}(350,230)(-5,-30)
\put(0,220){\small }
\put(20,220){\small }
\put(40,20){\small }
\put(60,220){\small }
\put(80,220){\small }
\put(100,220){ $-1$}
\put(140,220){ $-2$}
\put(180,220){ $-3$}
\put(220,220){ $-4$}
\put(260,220){ $-5$}
\put(300,220){ $-6$}
\put(340,220){ $-7$}
\put(380,220){ $-8$}
\put(420,220){ $-9$}
\put(458,220){ $-10$}
\put(498,220){ $-11$}
\put(-2,200){ 1\ $-$}
\put(20,200){  $-\alpha_1$}
\put(100,200){ [1]}
\put(177,200){ [2,3]}
\put(257,200){ [4,6]}
\put(337,200){ [7,8]}
\put(417,200){ [9,10]}
\put(500,200){ [11]}
\put(572,200){ $-\alpha_{11}$}
\put(-2,180){ 2\ $+$}
\put(53,180){  $-\alpha_2$}
\put(140,180){ [1,3]}
\put(217,180){ [2,6]}
\put(297,180){ [4,8]}
\put(377,180){ [7,10]}
\put(457,180){ [9,11]}
\put(536,180){ $-\alpha_{10}$}
\put(-2,160){ 3\ $-$}
\put(20,160){  $-\alpha_3$}
\put(100,160){ [3]}
\put(177,160){ [1,6]}
\put(257,160){ [2,8]}
\put(337,160){ [4,10]}
\put(417,160){ [7,11]}
\put(500,160){ [9]}
\put(572,160){ $-\alpha_{9}$}
\put(-2,140){ 4\ $+$}
\put(53,140){  $-\alpha_4$}
\put(140,140){ [3,6]}
\put(217,140){ [1,8]}
\put(297,140){ [2,10]}
\put(377,140){ [4,11]}
\put(457,140){ [7,9]}
\put(536,140){ $-\alpha_{8}$}
\put(-2,120){ 5\ $-$}
\put(20,120){  $-\alpha_5$}
\put(97,120){ [5,6]}
\put(177,120){ [3,8]}
\put(257,120){ [1,10]}
\put(337,120){ [2,11]}
\put(417,120){ [4,9]}
\put(500,120){ [7]}
\put(572,120){ $-\alpha_{7}$}
\put(-2,100){ 6}
\put(37,100){  $-\alpha_6$}
\put(80,100){  [6]}
\put(120,100){ [5]}
\put(157,100){ [6,8]}
\put(197,100){ [3,5]}
\put(237,100){ [6,10]}
\put(277,100){ [1,5]}
\put(317,100){ [6,11]}
\put(357,100){ [2,5]}
\put(397,100){ [6,9]}
\put(437,100){ [4,5]}
\put(477,100){ [6,7]}
\put(518,100){ $-\alpha_{6}$}
\put(558,100){ $\alpha_{6}$}
\put(-2,80){ 7\ $+$}
\put(53,80){  $-\alpha_7$}
\put(137,80){ [5,8]}
\put(217,80){ [3,10]}
\put(297,80){ [1,11]}
\put(377,80){ [2,9]}
\put(457,80){ [4,7]}
\put(536,80){ $-\alpha_{5}$}
\put(-2,60){ 8\ $-$}
\put(20,60){  $-\alpha_8$}
\put(100,60){ [8]}
\put(177,60){ [5,10]}
\put(257,60){ [3,11]}
\put(337,60){ [1,9]}
\put(417,60){ [2,7]}
\put(500,60){ [4]}
\put(572,60){ $-\alpha_{4}$}
\put(-2,40){ 9\ $+$}
\put(53,40){  $-\alpha_9$}
\put(137,40){ [8,10]}
\put(217,40){ [5,11]}
\put(297,40){ [3,9]}
\put(377,40){ [1,7]}
\put(457,40){ [2,4]}
\put(536,40){ $-\alpha_{3}$}
\put(-7,20){ 10\ $-$}
\put(20,20){  $-\alpha_{10}$}
\put(100,20){ [10]}
\put(177,20){ [8,11]}
\put(257,20){ [5,9]}
\put(337,20){ [3,7]}
\put(417,20){ [1,4]}
\put(500,20){ [2]}
\put(572,20){ $-\alpha_{2}$}
\put(-7,0){ 11\ $+$}
\put(52,0){  $-\alpha_{11}$}
\put(137,0){ [10,11]}
\put(217,0){ [8,9]}
\put(297,0){ [5,7]}
\put(377,0){ [3,4]}
\put(457,0){ [1,2]}
\put(536,0){ $-\alpha_{1}$}
\put(-5,232){\line(1,0){605}}
\put(-5,213){\line(1,0){605}}
\put(-5,-7){\line(1,0){605}}
\put(-5,-7){\line(0,1){239}}
\put(20,-7){\line(0,1){239}}
\put(80,-7){\line(0,1){239}}
\put(520,-7){\line(0,1){239}}
\put(600,-7){\line(0,1){239}}
%
\end{picture}
\caption{The orbits of $\sigma^k(-\alpha_i)$ and
$\sigma^k(\alpha_r)$ in $\Phi_+$
by $\sigma$ of $\eqref{eq:sigma}$
 for $r=6$.
The orbits of $-\alpha_6$ and $\alpha_6$,
i.e.,  $-\alpha_6\rightarrow [5]
\rightarrow [3,5]\rightarrow \cdots \rightarrow -\alpha_6 $
and
$\alpha_6 \rightarrow [6,8]
\rightarrow [6,10]\rightarrow \cdots \rightarrow \alpha_6 $
 are alternatively
aligned.
The numbers $-1, -2, \cdots$ in the head line
will be identified with the parameter $u$
in \eqref{eq:alpha}.
}
\label{tab:orbit}
\end{minipage}
}
\end{table}

The orbit
of $\sigma(-\alpha_i)$ ($i\neq r$) is further
described by the {\em root system of type $A_{2r-2}$}.

\begin{lemma}
\label{lem:cox}
Let $O_i=\{\sigma^{k}(-\alpha_i)\mid 1\leq k\leq r-1\}$
for $i\in J_+$ and 
$O_i=\{\sigma^{k}(-\alpha_i)\mid 1\leq k\leq r\}$
for $i\in J_-$.
Let $\Phi'_+$ be the set of the positive roots
of type $A_{2r-2}$ with index set $J'=J-\{r\}$,
and
\begin{align}
\rho:\bigsqcup_{i\in J'} O_i \rightarrow \Phi'_+
\end{align}
be the map which removes $\alpha_r$ from $\alpha$ if $\alpha$
 contains $\alpha_r$
and does nothing otherwise.
Then, $\rho$ is a bijection,
and its inverse
$\rho^{-1}(\alpha')$  adds $\alpha_r$ 
if $\alpha'$ contains $\alpha_{r-1}$ and does nothing otherwise.
Furthermore, under the bijection $\rho$,
the action of $\sigma$ is translated into the one
of the Coxeter element $s=s_- s_+$ of type $A_{2r-2}$
acting on $\Phi'_+$,
where $s_{\pm}=\prod_{i\in J_{\pm}} s_i$.
\end{lemma}

For $-h^{\vee}\leq u< 0$, define
\begin{align}
\label{eq:alpha}
&\alpha_{i}(u)=
\begin{cases}
\sigma^{-u/2}(-\alpha_i)
& \mbox{\rm $i\in J_+$, $u\equiv 0$,}\\
\sigma^{-(u-1)/2}(-\alpha_i)
& \mbox{\rm $i\in J_-$, $u\equiv -1$,}\\
\sigma^{-(2u-1)/4}(-\alpha_r)
& \mbox{\rm $i=r$, $u\equiv -\frac{3}{2}$,}\\
\sigma^{-(2u+1)/4}(\alpha_r)
& \mbox{\rm $i=r$, $u\equiv -\frac{1}{2}$,}\\
\end{cases}
\end{align}
where $\equiv$ is  modulo $2\mathbb{Z}$.
Note that they correspond to the positive roots
in Table \ref{tab:orbit}
with $u$ being the parameter in the head line.
By Lemma \ref{lem:orbit} they are all the positive roots
of $A_{2r-1}$.

\begin{lemma}
\label{lem:trec}
The family in \eqref{eq:alpha} satisfies the recurrence relations
\begin{align}
\label{eq:alpha1}
\begin{split}
\alpha_{i}(u-1)+\alpha_{i}(u+1)
&=\alpha_{i-1}(u)+\alpha_{i+1}(u)
\quad (i\neq r-1,r,r+1),\\
\alpha_{r-1}(u-1)+\alpha_{r-1}(u+1)
&=\alpha_{r-2}(u)+\alpha_{r+1}(u),\\
\alpha_{r+1}(u-1)+\alpha_{r+1}(u+1)
&=\alpha_{r-1}(u)+\alpha_{r+2}(u),\\
\textstyle\alpha_{r}(u-\frac{1}{2})+\alpha_{r}(u+\frac{1}{2})
&=\alpha_{r-1}(u) \quad \mbox{\rm ($u$: odd)},\\
\textstyle\alpha_{r}(u-\frac{1}{2})+\alpha_{r}(u+\frac{1}{2})
&=\alpha_{r+1}(u)\quad \mbox{\rm ($u$: even)},
\end{split}
\end{align}
where $\alpha_0(u)=\alpha_{2r}(u)=0$.
\end{lemma}
\begin{proof}
These relations are easily
verified by the explicit expressions
of $\alpha_i(u)$.
See Table \ref{tab:orbit}.
The first three relations are also obtained from
Lemma \ref{lem:cox} and \cite[Eq.~(10.9)]{FZ3}.
\end{proof}

Let us return to
 prove  (ii) of Proposition \ref{prop:lev2}
for the powers of variables
$y_{i1}$ ($i\neq r$) and $y_{r2}$.
For a monomial $m$ in $y=(y_{\mathbf{i}})_{\mathbf{i}\in
\mathbf{I}}$,
let $\pi_A(m)$ denote the specialization
with $y_{r1}=y_{r3}=1$.
For simplicity, we set
$y_{i1}=y_i$ ($i\neq r$), $y_{r2}=y_r$,
and also,
$y_{i1}(u)=y_{i}(u)$
($i\neq r$), $y_{r2}(u)=y_r(u)$.
We define the vectors $\mathbf{t}_{i}(u)
=(t_{i}(u)_k)_{k=1}^{2r-1}$
by
\begin{align}
\pi_A([y_{i}(u)]_{\mathbf{T}})
=
\prod_{k=1}^{2r-1}
 y_{k}^{t_{i}(u)_k}.
\end{align}
We also identify each vector $\mathbf{t}_{i}(u)$
with $\alpha=\sum_{k=1}^{2r-1}
t_{i}(u)_k \alpha_k \in \mathbb{Z}\Pi$.

\begin{proposition}
\label{prop:tvec}
Let $-h^{\vee}\leq u< 0$.
Then, we have
\begin{align}
\label{eq:tvec1}
\mathbf{t}_{i}(u)=-\alpha_i(u)
\end{align}
for $(i,u)$ in \eqref{eq:alpha},
and 
\begin{align}
\label{eq:piA}
&\pi_A([y_{r1}(u)]_{\mathbf{T}})
=\pi_A([y_{r3}(u)]_{\mathbf{T}})=1,
\quad 
\mbox{\rm $u\equiv 0$ mod $2\mathbb{Z}$}.
\end{align}
\end{proposition}
Note that these formulas
determine
$\pi_A([y_{\mathbf{i}}(u)]_{\mathbf{T}})$
 for any $(\mathbf{i},u):\mathbf{p}_+$. 

\begin{proof} 
We can verify the claim
for $-2\leq u \leq -\frac{1}{2}$ by direct
computation.
Then, by induction on $u$ in the backward direction,
one can  establish  the claim,
 together with
the recurrence relations among
$\mathbf{t}_{i}(u)$'s with $(i,u)$ in \eqref{eq:alpha},
\begin{align}
\label{eq:trec}
\begin{split}
\mathbf{t}_{i}(u-1)+\mathbf{t}_{i}(u+1)
&=\mathbf{t}_{i-1}(u)+\mathbf{t}_{i+1}(u),
\quad i\neq r-1,r,r+1,\\
\mathbf{t}_{r-1}(u-1)+\mathbf{t}_{r-1}(u+1)
&=\mathbf{t}_{r-2}(u)+\textstyle\mathbf{t}_{r}(u-\frac{1}{2})
+\mathbf{t}_{r}(u+\frac{1}{2}),\\
\mathbf{t}_{r+1}(u-1)+\mathbf{t}_{r+1}(u+1)
&=\mathbf{t}_{r+2}(u)+\textstyle\mathbf{t}_{r}(u-\frac{1}{2})
+\mathbf{t}_{r}(u+\frac{1}{2}),\\
\textstyle\mathbf{t}_{r}(u-\frac{1}{2})+\mathbf{t}_{r}(u+\frac{1}{2})
&=\mathbf{t}_{r-1}(u), \quad \mbox{\rm $u$: odd},\\
\textstyle\mathbf{t}_{r}(u-\frac{1}{2})+\mathbf{t}_{r}(u+\frac{1}{2})
&=\mathbf{t}_{r+1}(u),\quad \mbox{\rm $u$: even}.
\end{split}
\end{align}
Note that \eqref{eq:trec} coincides with
\eqref{eq:alpha1} under \eqref{eq:tvec1}.
To derive \eqref{eq:trec}, one uses
 the mutations as in Figure \ref{fig:tropB22}
(or  the tropical version of the Y-system
$\mathbb{Y}_2(B_r)$ directly)
 and the positivity/negativity
of $\pi_A([y_{\mathbf{i}}(u)]_{\mathbf{T}})$
resulting from \eqref{eq:tvec1} and \eqref{eq:piA}
by induction hypothesis.
\end{proof}

Now (ii) in Proposition \ref{prop:lev2}
is an immediate consequence of
 Lemma
\ref{lem:orbit} and Proposition \ref{prop:tvec}.
Finally, let us prove (iv), i.e.,
 $\mathbf{t}_{i}(-h^{\vee})=-\alpha_{\omega(i)}$.
This is shown by the following formulas obtained from
Lemma \ref{lem:orbit} and \eqref{eq:ymu}:
\begin{align}
\begin{split}
\mathbf{t}_i(-h^{\vee}+1)&=
\begin{cases}
-\alpha_{\omega(i)-1}-\alpha_{\omega(i)}-\alpha_{\omega(i)+1}
& i\neq r+1,
\quad
\\
-\alpha_{r-2}-\alpha_{r-1}-\alpha_{r}-\alpha_{r+1}
& i= r+1,\\
\end{cases}
i\in J_+,\\
\mathbf{t}_i(-h^{\vee}+\textstyle\frac{1}{2})&=\alpha_{\omega(i)},
\quad
\mathbf{t}_i(-h^{\vee})=-\alpha_{\omega(i)},
\quad
i\in J_-,
\\
\mathbf{t}_r(-h^{\vee}+1)&=\alpha_{r}+\alpha_{r+1},\quad
\mathbf{t}_r(-h^{\vee}+\textstyle\frac{1}{2})=
-\alpha_{r}-\alpha_{r+1},
\end{split}
\end{align}
where $\alpha_0=\alpha_{2r}=0$.
This completes the proof of Proposition \ref{prop:lev2}.

\section{Tropical Y-systems at higher levels}

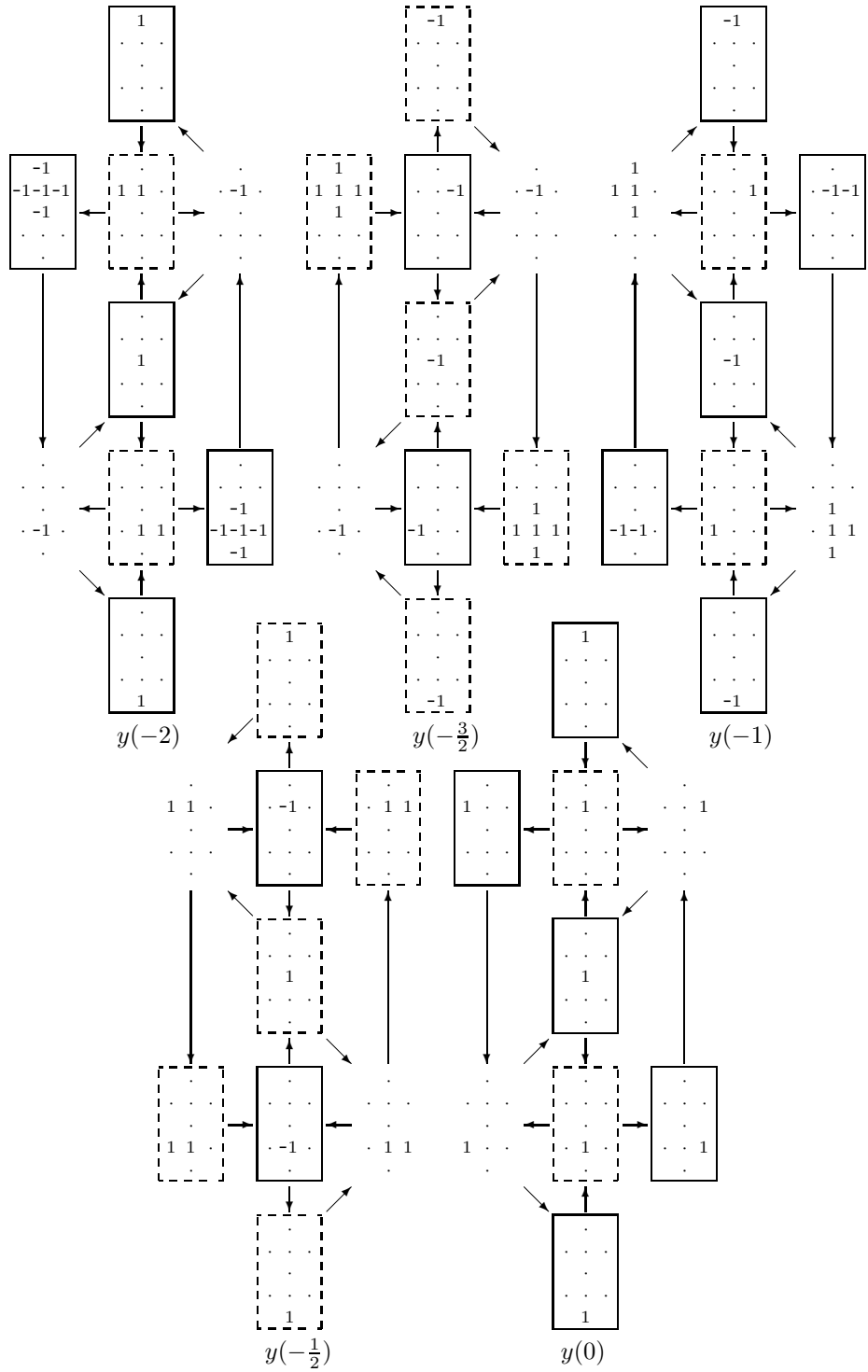
\begin{figure}
\begin{picture}(355,550)(-15,-95)
%
\put(0,170)
{
\put(0,240)
{
\put(37,-2.5){\framebox(26,46)[c]}
\put(48,36){$\scriptstyle  1$}
\put(41,27){$\scriptstyle \cdot$}
\put(49,27){$\scriptstyle \cdot$}
\put(57,27){$\scriptstyle \cdot$}
\put(49,18){$\scriptstyle \cdot$}
\put(41,9){$\scriptstyle \cdot$}
\put(49,9){$\scriptstyle \cdot$}
\put(57,9){$\scriptstyle \cdot$}
\put(49,0){$\scriptstyle \cdot$}
}
\put(-40,180)
{
\put(37,-2.5){\framebox(26,46)[c]}
\put(46,36){-$\scriptstyle  1$}
\put(38,27){-$\scriptstyle 1$}
\put(46,27){-$\scriptstyle 1$}
\put(54,27){-$\scriptstyle 1$}
\put(46,18){-$\scriptstyle 1$}
\put(41,9){$\scriptstyle \cdot$}
\put(49,9){$\scriptstyle \cdot$}
\put(57,9){$\scriptstyle \cdot$}
\put(49,0){$\scriptstyle \cdot$}
}
\put(0,180)
{
\put(37,-2.5){\dashbox{3}(26,46)[c]}
\put(49,36){$\scriptstyle  \cdot$}
\put(40,27){$\scriptstyle 1$}
\put(48,27){$\scriptstyle 1$}
\put(57,27){$\scriptstyle \cdot$}
\put(49,18){$\scriptstyle \cdot$}
\put(41,9){$\scriptstyle \cdot$}
\put(49,9){$\scriptstyle \cdot$}
\put(57,9){$\scriptstyle \cdot$}
\put(49,0){$\scriptstyle \cdot$}
}
\put(40,180)
{
\put(37,-2.5){\makebox(26,46)[c]}
\put(49,36){$\scriptstyle  \cdot$}
\put(41,27){$\scriptstyle \cdot$}
\put(46,27){-$\scriptstyle 1$}
\put(57,27){$\scriptstyle \cdot$}
\put(49,18){$\scriptstyle \cdot$}
\put(41,9){$\scriptstyle \cdot$}
\put(49,9){$\scriptstyle \cdot$}
\put(57,9){$\scriptstyle \cdot$}
\put(49,0){$\scriptstyle \cdot$}
}
\put(0,120)
{
\put(37,-2.5){\framebox(26,46)[c]}
\put(49,36){$\scriptstyle  \cdot$}
\put(41,27){$\scriptstyle \cdot$}
\put(49,27){$\scriptstyle \cdot$}
\put(57,27){$\scriptstyle \cdot$}
\put(48,18){$\scriptstyle 1$}
\put(41,9){$\scriptstyle \cdot$}
\put(49,9){$\scriptstyle \cdot$}
\put(57,9){$\scriptstyle \cdot$}
\put(49,0){$\scriptstyle \cdot$}
}
\put(-40,60)
{
\put(37,-2.5){\makebox(26,46)[c]}
\put(49,36){$\scriptstyle  \cdot$}
\put(41,27){$\scriptstyle \cdot$}
\put(49,27){$\scriptstyle \cdot$}
\put(57,27){$\scriptstyle \cdot$}
\put(49,18){$\scriptstyle \cdot$}
\put(41,9){$\scriptstyle \cdot$}
\put(46,9){-$\scriptstyle 1$}
\put(57,9){$\scriptstyle \cdot$}
\put(49,0){$\scriptstyle \cdot$}
}
\put(0,60)
{
\put(37,-2.5){\dashbox{3}(26,46)[c]}
\put(49,36){$\scriptstyle  \cdot$}
\put(41,27){$\scriptstyle \cdot$}
\put(49,27){$\scriptstyle \cdot$}
\put(57,27){$\scriptstyle \cdot$}
\put(49,18){$\scriptstyle \cdot$}
\put(41,9){$\scriptstyle \cdot$}
\put(48,9){$\scriptstyle 1$}
\put(56,9){$\scriptstyle 1$}
\put(49,0){$\scriptstyle \cdot$}
}
\put(40,60)
{
\put(37,-2.5){\framebox(26,46)[c]}
\put(49,36){$\scriptstyle  \cdot$}
\put(41,27){$\scriptstyle \cdot$}
\put(49,27){$\scriptstyle \cdot$}
\put(57,27){$\scriptstyle \cdot$}
\put(46,18){-$\scriptstyle 1$}
\put(38,9){-$\scriptstyle 1$}
\put(46,9){-$\scriptstyle 1$}
\put(54,9){-$\scriptstyle 1$}
\put(46,0){-$\scriptstyle 1$}
}
\put(0,0)
{
\put(37,-2.5){\framebox(26,46)[c]}
\put(49,36){$\scriptstyle  \cdot$}
\put(41,27){$\scriptstyle \cdot$}
\put(49,27){$\scriptstyle \cdot$}
\put(57,27){$\scriptstyle \cdot$}
\put(49,18){$\scriptstyle \cdot$}
\put(41,9){$\scriptstyle \cdot$}
\put(49,9){$\scriptstyle \cdot$}
\put(57,9){$\scriptstyle \cdot$}
\put(48,0){$\scriptstyle 1$}
}

%
\put(35,200){\vector(-1,0){10}}
\put(65,200){\vector(1,0){10}}
\put(35,80){\vector(-1,0){10}}
\put(65,80){\vector(1,0){10}}
\put(90,105){\vector(0,1){70}}
\put(10,175){\vector(0,-1){70}}
\put(50,235){\vector(0,-1){10}}
\put(50,165){\vector(0,1){10}}
\put(50,115){\vector(0,-1){10}}
\put(50,45){\vector(0,1){10}}
\put(75,225){\vector(-1,1){10}}
\put(75,175){\vector(-1,-1){10}}
\put(25,105){\vector(1,1){10}}
\put(25,55){\vector(1,-1){10}}
\put(40,-15){$y(-2)$}
}
%
%
\put(120,170)
{
\put(0,240)
{
\put(37,-2.5){\dashbox{3}(26,46)[c]}
\put(46,36){-$\scriptstyle  1$}
\put(41,27){$\scriptstyle \cdot$}
\put(49,27){$\scriptstyle \cdot$}
\put(57,27){$\scriptstyle \cdot$}
\put(49,18){$\scriptstyle \cdot$}
\put(41,9){$\scriptstyle \cdot$}
\put(49,9){$\scriptstyle \cdot$}
\put(57,9){$\scriptstyle \cdot$}
\put(49,0){$\scriptstyle \cdot$}
}
\put(-40,180)
{
\put(37,-2.5){\dashbox{3}(26,46)[c]}
\put(48,36){$\scriptstyle  1$}
\put(40,27){$\scriptstyle 1$}
\put(48,27){$\scriptstyle 1$}
\put(56,27){$\scriptstyle 1$}
\put(48,18){$\scriptstyle 1$}
\put(41,9){$\scriptstyle \cdot$}
\put(49,9){$\scriptstyle \cdot$}
\put(57,9){$\scriptstyle \cdot$}
\put(49,0){$\scriptstyle \cdot$}
}
\put(0,180)
{
\put(37,-2.5){\framebox(26,46)[c]}
\put(49,36){$\scriptstyle  \cdot$}
\put(41,27){$\scriptstyle \cdot$}
\put(49,27){$\scriptstyle \cdot$}
\put(54,27){-$\scriptstyle 1$}
\put(49,18){$\scriptstyle \cdot$}
\put(41,9){$\scriptstyle \cdot$}
\put(49,9){$\scriptstyle \cdot$}
\put(57,9){$\scriptstyle \cdot$}
\put(49,0){$\scriptstyle \cdot$}
}
\put(40,180)
{
\put(37,-2.5){\makebox(26,46)[c]}
\put(49,36){$\scriptstyle  \cdot$}
\put(41,27){$\scriptstyle \cdot$}
\put(46,27){-$\scriptstyle 1$}
\put(57,27){$\scriptstyle \cdot$}
\put(49,18){$\scriptstyle \cdot$}
\put(41,9){$\scriptstyle \cdot$}
\put(49,9){$\scriptstyle \cdot$}
\put(57,9){$\scriptstyle \cdot$}
\put(49,0){$\scriptstyle \cdot$}
}
\put(0,120)
{
\put(37,-2.5){\dashbox{3}(26,46)[c]}
\put(49,36){$\scriptstyle  \cdot$}
\put(41,27){$\scriptstyle \cdot$}
\put(49,27){$\scriptstyle \cdot$}
\put(57,27){$\scriptstyle \cdot$}
\put(46,18){-$\scriptstyle 1$}
\put(41,9){$\scriptstyle \cdot$}
\put(49,9){$\scriptstyle \cdot$}
\put(57,9){$\scriptstyle \cdot$}
\put(49,0){$\scriptstyle \cdot$}
}
\put(-40,60)
{
\put(37,-2.5){\makebox(26,46)[c]}
\put(49,36){$\scriptstyle  \cdot$}
\put(41,27){$\scriptstyle \cdot$}
\put(49,27){$\scriptstyle \cdot$}
\put(57,27){$\scriptstyle \cdot$}
\put(49,18){$\scriptstyle \cdot$}
\put(41,9){$\scriptstyle \cdot$}
\put(46,9){-$\scriptstyle 1$}
\put(57,9){$\scriptstyle \cdot$}
\put(49,0){$\scriptstyle \cdot$}
}
\put(0,60)
{
\put(37,-2.5){\framebox(26,46)[c]}
\put(49,36){$\scriptstyle  \cdot$}
\put(41,27){$\scriptstyle \cdot$}
\put(49,27){$\scriptstyle \cdot$}
\put(57,27){$\scriptstyle \cdot$}
\put(49,18){$\scriptstyle \cdot$}
\put(38,9){-$\scriptstyle 1$}
\put(49,9){$\scriptstyle \cdot$}
\put(57,9){$\scriptstyle \cdot$}
\put(49,0){$\scriptstyle \cdot$}
}
\put(40,60)
{
\put(37,-2.5){\dashbox{3}(26,46)[c]}
\put(49,36){$\scriptstyle  \cdot$}
\put(41,27){$\scriptstyle \cdot$}
\put(49,27){$\scriptstyle \cdot$}
\put(57,27){$\scriptstyle \cdot$}
\put(48,18){$\scriptstyle 1$}
\put(40,9){$\scriptstyle 1$}
\put(48,9){$\scriptstyle 1$}
\put(56,9){$\scriptstyle 1$}
\put(48,0){$\scriptstyle 1$}
}
\put(0,0)
{
\put(37,-2.5){\dashbox{3}(26,46)[c]}
\put(49,36){$\scriptstyle  \cdot$}
\put(41,27){$\scriptstyle \cdot$}
\put(49,27){$\scriptstyle \cdot$}
\put(57,27){$\scriptstyle \cdot$}
\put(49,18){$\scriptstyle \cdot$}
\put(41,9){$\scriptstyle \cdot$}
\put(49,9){$\scriptstyle \cdot$}
\put(57,9){$\scriptstyle \cdot$}
\put(46,0){-$\scriptstyle 1$}
}

%
\put(25,200){\vector(1,0){10}}
\put(75,200){\vector(-1,0){10}}
\put(25,80){\vector(1,0){10}}
\put(75,80){\vector(-1,0){10}}
\put(90,175){\vector(0,-1){70}}
\put(10,105){\vector(0,1){70}}
\put(50,225){\vector(0,1){10}}
\put(50,175){\vector(0,-1){10}}
\put(50,105){\vector(0,1){10}}
\put(50,55){\vector(0,-1){10}}
\put(65,235){\vector(1,-1){10}}
\put(65,165){\vector(1,1){10}}
\put(35,115){\vector(-1,-1){10}}
\put(35,45){\vector(-1,1){10}}
\put(40,-15){$y(-\frac{3}{2})$}
}
%
%
\put(240,170)
{
\put(0,240)
{
\put(37,-2.5){\framebox(26,46)[c]}
\put(46,36){-$\scriptstyle  1$}
\put(41,27){$\scriptstyle \cdot$}
\put(49,27){$\scriptstyle \cdot$}
\put(57,27){$\scriptstyle \cdot$}
\put(49,18){$\scriptstyle \cdot$}
\put(41,9){$\scriptstyle \cdot$}
\put(49,9){$\scriptstyle \cdot$}
\put(57,9){$\scriptstyle \cdot$}
\put(49,0){$\scriptstyle \cdot$}
}
\put(-40,180)
{
\put(37,-2.5){\makebox(26,46)[c]}
\put(48,36){$\scriptstyle  1$}
\put(40,27){$\scriptstyle 1$}
\put(48,27){$\scriptstyle 1$}
\put(57,27){$\scriptstyle \cdot$}
\put(48,18){$\scriptstyle 1$}
\put(41,9){$\scriptstyle \cdot$}
\put(49,9){$\scriptstyle \cdot$}
\put(57,9){$\scriptstyle \cdot$}
\put(49,0){$\scriptstyle \cdot$}
}
\put(0,180)
{
\put(37,-2.5){\dashbox{3}(26,46)[c]}
\put(49,36){$\scriptstyle  \cdot$}
\put(41,27){$\scriptstyle \cdot$}
\put(49,27){$\scriptstyle \cdot$}
\put(56,27){$\scriptstyle 1$}
\put(49,18){$\scriptstyle \cdot$}
\put(41,9){$\scriptstyle \cdot$}
\put(49,9){$\scriptstyle \cdot$}
\put(57,9){$\scriptstyle \cdot$}
\put(49,0){$\scriptstyle \cdot$}
}
\put(40,180)
{
\put(37,-2.5){\framebox(26,46)[c]}
\put(49,36){$\scriptstyle  \cdot$}
\put(41,27){$\scriptstyle \cdot$}
\put(46,27){-$\scriptstyle 1$}
\put(54,27){-$\scriptstyle 1$}
\put(49,18){$\scriptstyle \cdot$}
\put(41,9){$\scriptstyle \cdot$}
\put(49,9){$\scriptstyle \cdot$}
\put(57,9){$\scriptstyle \cdot$}
\put(49,0){$\scriptstyle \cdot$}
}
\put(0,120)
{
\put(37,-2.5){\framebox(26,46)[c]}
\put(49,36){$\scriptstyle  \cdot$}
\put(41,27){$\scriptstyle \cdot$}
\put(49,27){$\scriptstyle \cdot$}
\put(57,27){$\scriptstyle \cdot$}
\put(46,18){-$\scriptstyle 1$}
\put(41,9){$\scriptstyle \cdot$}
\put(49,9){$\scriptstyle \cdot$}
\put(57,9){$\scriptstyle \cdot$}
\put(49,0){$\scriptstyle \cdot$}
}
\put(-40,60)
{
\put(37,-2.5){\framebox(26,46)[c]}
\put(49,36){$\scriptstyle  \cdot$}
\put(41,27){$\scriptstyle \cdot$}
\put(49,27){$\scriptstyle \cdot$}
\put(57,27){$\scriptstyle \cdot$}
\put(49,18){$\scriptstyle \cdot$}
\put(40,9){-$\scriptstyle 1$}
\put(48,9){-$\scriptstyle 1$}
\put(57,9){$\scriptstyle \cdot$}
\put(49,0){$\scriptstyle \cdot$}
}
\put(0,60)
{
\put(37,-2.5){\dashbox{3}(26,46)[c]}
\put(49,36){$\scriptstyle  \cdot$}
\put(41,27){$\scriptstyle \cdot$}
\put(49,27){$\scriptstyle \cdot$}
\put(57,27){$\scriptstyle \cdot$}
\put(49,18){$\scriptstyle \cdot$}
\put(40,9){$\scriptstyle 1$}
\put(49,9){$\scriptstyle \cdot$}
\put(57,9){$\scriptstyle \cdot$}
\put(49,0){$\scriptstyle \cdot$}
}
\put(40,60)
{
\put(37,-2.5){\makebox(26,46)[c]}
\put(49,36){$\scriptstyle  \cdot$}
\put(41,27){$\scriptstyle \cdot$}
\put(49,27){$\scriptstyle \cdot$}
\put(57,27){$\scriptstyle \cdot$}
\put(48,18){$\scriptstyle 1$}
\put(41,9){$\scriptstyle \cdot$}
\put(48,9){$\scriptstyle 1$}
\put(56,9){$\scriptstyle 1$}
\put(48,0){$\scriptstyle 1$}
}
\put(0,0)
{
\put(37,-2.5){\framebox(26,46)[c]}
\put(49,36){$\scriptstyle  \cdot$}
\put(41,27){$\scriptstyle \cdot$}
\put(49,27){$\scriptstyle \cdot$}
\put(57,27){$\scriptstyle \cdot$}
\put(49,18){$\scriptstyle \cdot$}
\put(41,9){$\scriptstyle \cdot$}
\put(49,9){$\scriptstyle \cdot$}
\put(57,9){$\scriptstyle \cdot$}
\put(46,0){-$\scriptstyle 1$}
}

%
\put(35,200){\vector(-1,0){10}}
\put(65,200){\vector(1,0){10}}
\put(35,80){\vector(-1,0){10}}
\put(65,80){\vector(1,0){10}}
\put(90,175){\vector(0,-1){70}}
\put(10,105){\vector(0,1){70}}
\put(50,235){\vector(0,-1){10}}
\put(50,165){\vector(0,1){10}}
\put(50,115){\vector(0,-1){10}}
\put(50,45){\vector(0,1){10}}
\put(25,225){\vector(1,1){10}}
\put(25,175){\vector(1,-1){10}}
\put(75,105){\vector(-1,1){10}}
\put(75,55){\vector(-1,-1){10}}
\put(40,-15){$y(-1)$}
}
%
%
\put(60,-80)
{
\put(0,240)
{
\put(37,-2.5){\dashbox{3}(26,46)[c]}
\put(48,36){$\scriptstyle  1$}
\put(41,27){$\scriptstyle \cdot$}
\put(49,27){$\scriptstyle \cdot$}
\put(57,27){$\scriptstyle \cdot$}
\put(49,18){$\scriptstyle \cdot$}
\put(41,9){$\scriptstyle \cdot$}
\put(49,9){$\scriptstyle \cdot$}
\put(57,9){$\scriptstyle \cdot$}
\put(49,0){$\scriptstyle \cdot$}
}
\put(-40,180)
{
\put(37,-2.5){\makebox(26,46)[c]}
\put(49,36){$\scriptstyle  \cdot$}
\put(40,27){$\scriptstyle 1$}
\put(48,27){$\scriptstyle 1$}
\put(57,27){$\scriptstyle \cdot$}
\put(49,18){$\scriptstyle \cdot$}
\put(41,9){$\scriptstyle \cdot$}
\put(49,9){$\scriptstyle \cdot$}
\put(57,9){$\scriptstyle \cdot$}
\put(49,0){$\scriptstyle \cdot$}
}
\put(0,180)
{
\put(37,-2.5){\framebox(26,46)[c]}
\put(49,36){$\scriptstyle  \cdot$}
\put(41,27){$\scriptstyle \cdot$}
\put(46,27){-$\scriptstyle 1$}
\put(57,27){$\scriptstyle \cdot$}
\put(49,18){$\scriptstyle \cdot$}
\put(41,9){$\scriptstyle \cdot$}
\put(49,9){$\scriptstyle \cdot$}
\put(57,9){$\scriptstyle \cdot$}
\put(49,0){$\scriptstyle \cdot$}
}
\put(40,180)
{
\put(37,-2.5){\dashbox{3}(26,46)[c]}
\put(49,36){$\scriptstyle  \cdot$}
\put(41,27){$\scriptstyle \cdot$}
\put(48,27){$\scriptstyle 1$}
\put(56,27){$\scriptstyle 1$}
\put(49,18){$\scriptstyle \cdot$}
\put(41,9){$\scriptstyle \cdot$}
\put(49,9){$\scriptstyle \cdot$}
\put(57,9){$\scriptstyle \cdot$}
\put(49,0){$\scriptstyle \cdot$}
}
\put(0,120)
{
\put(37,-2.5){\dashbox{3}(26,46)[c]}
\put(49,36){$\scriptstyle  \cdot$}
\put(41,27){$\scriptstyle \cdot$}
\put(49,27){$\scriptstyle \cdot$}
\put(57,27){$\scriptstyle \cdot$}
\put(48,18){$\scriptstyle 1$}
\put(41,9){$\scriptstyle \cdot$}
\put(49,9){$\scriptstyle \cdot$}
\put(57,9){$\scriptstyle \cdot$}
\put(49,0){$\scriptstyle \cdot$}
}
\put(-40,60)
{
\put(37,-2.5){\dashbox{3}(26,46)[c]}
\put(49,36){$\scriptstyle  \cdot$}
\put(41,27){$\scriptstyle \cdot$}
\put(49,27){$\scriptstyle \cdot$}
\put(57,27){$\scriptstyle \cdot$}
\put(49,18){$\scriptstyle \cdot$}
\put(40,9){$\scriptstyle 1$}
\put(48,9){$\scriptstyle 1$}
\put(57,9){$\scriptstyle \cdot$}
\put(49,0){$\scriptstyle \cdot$}
}
\put(0,60)
{
\put(37,-2.5){\framebox(26,46)[c]}
\put(49,36){$\scriptstyle  \cdot$}
\put(41,27){$\scriptstyle \cdot$}
\put(49,27){$\scriptstyle \cdot$}
\put(57,27){$\scriptstyle \cdot$}
\put(49,18){$\scriptstyle \cdot$}
\put(41,9){$\scriptstyle \cdot$}
\put(46,9){-$\scriptstyle 1$}
\put(57,9){$\scriptstyle \cdot$}
\put(49,0){$\scriptstyle \cdot$}
}
\put(40,60)
{
\put(37,-2.5){\makebox(26,46)[c]}
\put(49,36){$\scriptstyle  \cdot$}
\put(41,27){$\scriptstyle \cdot$}
\put(49,27){$\scriptstyle \cdot$}
\put(57,27){$\scriptstyle \cdot$}
\put(49,18){$\scriptstyle \cdot$}
\put(41,9){$\scriptstyle \cdot$}
\put(48,9){$\scriptstyle 1$}
\put(56,9){$\scriptstyle 1$}
\put(49,0){$\scriptstyle \cdot$}
}
\put(0,0)
{
\put(37,-2.5){\dashbox{3}(26,46)[c]}
\put(49,36){$\scriptstyle  \cdot$}
\put(41,27){$\scriptstyle \cdot$}
\put(49,27){$\scriptstyle \cdot$}
\put(57,27){$\scriptstyle \cdot$}
\put(49,18){$\scriptstyle \cdot$}
\put(41,9){$\scriptstyle \cdot$}
\put(49,9){$\scriptstyle \cdot$}
\put(57,9){$\scriptstyle \cdot$}
\put(48,0){$\scriptstyle 1$}
}

%
\put(25,200){\vector(1,0){10}}
\put(75,200){\vector(-1,0){10}}
\put(25,80){\vector(1,0){10}}
\put(75,80){\vector(-1,0){10}}
\put(90,105){\vector(0,1){70}}
\put(10,175){\vector(0,-1){70}}
\put(50,225){\vector(0,1){10}}
\put(50,175){\vector(0,-1){10}}
\put(50,105){\vector(0,1){10}}
\put(50,55){\vector(0,-1){10}}
\put(35,245){\vector(-1,-1){10}}
\put(35,165){\vector(-1,1){10}}
\put(65,115){\vector(1,-1){10}}
\put(65,45){\vector(1,1){10}}
\put(40,-15){$y(-\frac{1}{2})$}
}
%
%
\put(180,-80)
{
\put(0,240)
{
\put(37,-2.5){\framebox(26,46)[c]}
\put(48,36){$\scriptstyle  1$}
\put(41,27){$\scriptstyle \cdot$}
\put(49,27){$\scriptstyle \cdot$}
\put(57,27){$\scriptstyle \cdot$}
\put(49,18){$\scriptstyle \cdot$}
\put(41,9){$\scriptstyle \cdot$}
\put(49,9){$\scriptstyle \cdot$}
\put(57,9){$\scriptstyle \cdot$}
\put(49,0){$\scriptstyle \cdot$}
}
\put(-40,180)
{
\put(37,-2.5){\framebox(26,46)[c]}
\put(49,36){$\scriptstyle  \cdot$}
\put(40,27){$\scriptstyle 1$}
\put(49,27){$\scriptstyle \cdot$}
\put(57,27){$\scriptstyle \cdot$}
\put(49,18){$\scriptstyle \cdot$}
\put(41,9){$\scriptstyle \cdot$}
\put(49,9){$\scriptstyle \cdot$}
\put(57,9){$\scriptstyle \cdot$}
\put(49,0){$\scriptstyle \cdot$}
}
\put(0,180)
{
\put(37,-2.5){\dashbox{3}(26,46)[c]}
\put(49,36){$\scriptstyle  \cdot$}
\put(41,27){$\scriptstyle \cdot$}
\put(48,27){$\scriptstyle 1$}
\put(57,27){$\scriptstyle \cdot$}
\put(49,18){$\scriptstyle \cdot$}
\put(41,9){$\scriptstyle \cdot$}
\put(49,9){$\scriptstyle \cdot$}
\put(57,9){$\scriptstyle \cdot$}
\put(49,0){$\scriptstyle \cdot$}
}
\put(40,180)
{
\put(37,-2.5){\makebox(26,46)[c]}
\put(49,36){$\scriptstyle  \cdot$}
\put(41,27){$\scriptstyle \cdot$}
\put(49,27){$\scriptstyle \cdot$}
\put(56,27){$\scriptstyle 1$}
\put(49,18){$\scriptstyle \cdot$}
\put(41,9){$\scriptstyle \cdot$}
\put(49,9){$\scriptstyle \cdot$}
\put(57,9){$\scriptstyle \cdot$}
\put(49,0){$\scriptstyle \cdot$}
}
\put(0,120)
{
\put(37,-2.5){\framebox(26,46)[c]}
\put(49,36){$\scriptstyle  \cdot$}
\put(41,27){$\scriptstyle \cdot$}
\put(49,27){$\scriptstyle \cdot$}
\put(57,27){$\scriptstyle \cdot$}
\put(48,18){$\scriptstyle 1$}
\put(41,9){$\scriptstyle \cdot$}
\put(49,9){$\scriptstyle \cdot$}
\put(57,9){$\scriptstyle \cdot$}
\put(49,0){$\scriptstyle \cdot$}
}
\put(-40,60)
{
\put(37,-2.5){\makebox(26,46)[c]}
\put(49,36){$\scriptstyle  \cdot$}
\put(41,27){$\scriptstyle \cdot$}
\put(49,27){$\scriptstyle \cdot$}
\put(57,27){$\scriptstyle \cdot$}
\put(49,18){$\scriptstyle \cdot$}
\put(40,9){$\scriptstyle 1$}
\put(49,9){$\scriptstyle \cdot$}
\put(57,9){$\scriptstyle \cdot$}
\put(49,0){$\scriptstyle \cdot$}
}
\put(0,60)
{
\put(37,-2.5){\dashbox{3}(26,46)[c]}
\put(49,36){$\scriptstyle  \cdot$}
\put(41,27){$\scriptstyle \cdot$}
\put(49,27){$\scriptstyle \cdot$}
\put(57,27){$\scriptstyle \cdot$}
\put(49,18){$\scriptstyle \cdot$}
\put(41,9){$\scriptstyle \cdot$}
\put(48,9){$\scriptstyle 1$}
\put(57,9){$\scriptstyle \cdot$}
\put(49,0){$\scriptstyle \cdot$}
}
\put(40,60)
{
\put(37,-2.5){\framebox(26,46)[c]}
\put(49,36){$\scriptstyle  \cdot$}
\put(41,27){$\scriptstyle \cdot$}
\put(49,27){$\scriptstyle \cdot$}
\put(57,27){$\scriptstyle \cdot$}
\put(49,18){$\scriptstyle \cdot$}
\put(41,9){$\scriptstyle \cdot$}
\put(49,9){$\scriptstyle \cdot$}
\put(56,9){$\scriptstyle 1$}
\put(49,0){$\scriptstyle \cdot$}
}
\put(0,0)
{
\put(37,-2.5){\framebox(26,46)[c]}
\put(49,36){$\scriptstyle  \cdot$}
\put(41,27){$\scriptstyle \cdot$}
\put(49,27){$\scriptstyle \cdot$}
\put(57,27){$\scriptstyle \cdot$}
\put(49,18){$\scriptstyle \cdot$}
\put(41,9){$\scriptstyle \cdot$}
\put(49,9){$\scriptstyle \cdot$}
\put(57,9){$\scriptstyle \cdot$}
\put(48,0){$\scriptstyle 1$}
}

%
\put(35,200){\vector(-1,0){10}}
\put(65,200){\vector(1,0){10}}
\put(35,80){\vector(-1,0){10}}
\put(65,80){\vector(1,0){10}}
\put(90,105){\vector(0,1){70}}
\put(10,175){\vector(0,-1){70}}
\put(50,235){\vector(0,-1){10}}
\put(50,165){\vector(0,1){10}}
\put(50,115){\vector(0,-1){10}}
\put(50,45){\vector(0,1){10}}
\put(75,225){\vector(-1,1){10}}
\put(75,175){\vector(-1,-1){10}}
\put(25,105){\vector(1,1){10}}
\put(25,55){\vector(1,-1){10}}
\put(40,-15){$y(0)$}
}
\end{picture}
\caption{Tropical Y-system of type $B_2$ at level 3
in the region $-2 \leq u \leq 0$.
}
\label{fig:Blev3}
\end{figure}


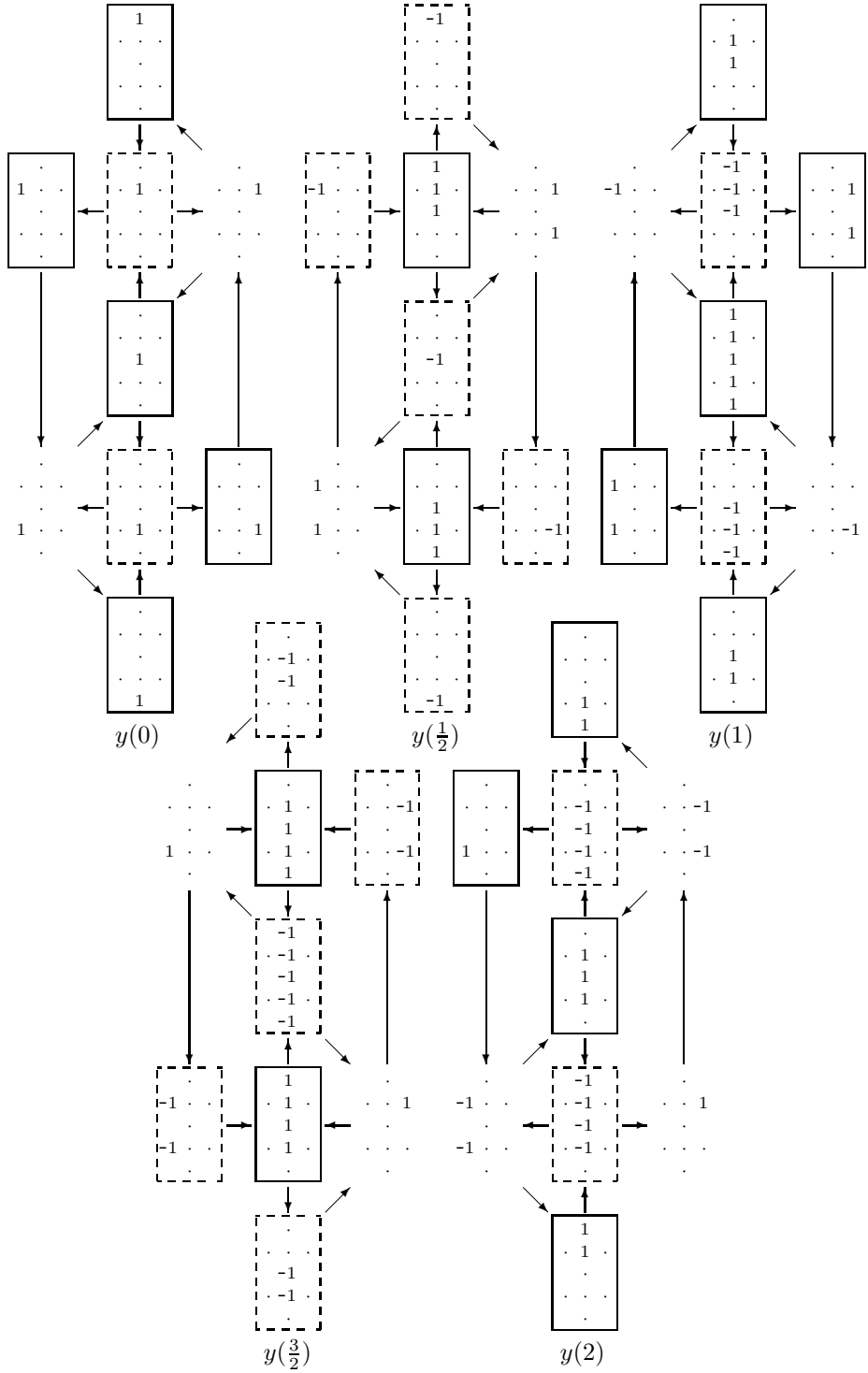
\begin{figure}
\begin{picture}(355,550)(-15,-95)
%
\put(0,170)
{
\put(0,240)
{
\put(37,-2.5){\framebox(26,46)[c]}
\put(48,36){$\scriptstyle  1$}
\put(41,27){$\scriptstyle \cdot$}
\put(49,27){$\scriptstyle \cdot$}
\put(57,27){$\scriptstyle \cdot$}
\put(49,18){$\scriptstyle \cdot$}
\put(41,9){$\scriptstyle \cdot$}
\put(49,9){$\scriptstyle \cdot$}
\put(57,9){$\scriptstyle \cdot$}
\put(49,0){$\scriptstyle \cdot$}
}
\put(-40,180)
{
\put(37,-2.5){\framebox(26,46)[c]}
\put(49,36){$\scriptstyle  \cdot$}
\put(40,27){$\scriptstyle 1$}
\put(49,27){$\scriptstyle \cdot$}
\put(57,27){$\scriptstyle \cdot$}
\put(49,18){$\scriptstyle \cdot$}
\put(41,9){$\scriptstyle \cdot$}
\put(49,9){$\scriptstyle \cdot$}
\put(57,9){$\scriptstyle \cdot$}
\put(49,0){$\scriptstyle \cdot$}
}
\put(0,180)
{
\put(37,-2.5){\dashbox{3}(26,46)[c]}
\put(49,36){$\scriptstyle  \cdot$}
\put(41,27){$\scriptstyle \cdot$}
\put(48,27){$\scriptstyle 1$}
\put(57,27){$\scriptstyle \cdot$}
\put(49,18){$\scriptstyle \cdot$}
\put(41,9){$\scriptstyle \cdot$}
\put(49,9){$\scriptstyle \cdot$}
\put(57,9){$\scriptstyle \cdot$}
\put(49,0){$\scriptstyle \cdot$}
}
\put(40,180)
{
\put(37,-2.5){\makebox(26,46)[c]}
\put(49,36){$\scriptstyle  \cdot$}
\put(41,27){$\scriptstyle \cdot$}
\put(49,27){$\scriptstyle \cdot$}
\put(56,27){$\scriptstyle 1$}
\put(49,18){$\scriptstyle \cdot$}
\put(41,9){$\scriptstyle \cdot$}
\put(49,9){$\scriptstyle \cdot$}
\put(57,9){$\scriptstyle \cdot$}
\put(49,0){$\scriptstyle \cdot$}
}
\put(0,120)
{
\put(37,-2.5){\framebox(26,46)[c]}
\put(49,36){$\scriptstyle  \cdot$}
\put(41,27){$\scriptstyle \cdot$}
\put(49,27){$\scriptstyle \cdot$}
\put(57,27){$\scriptstyle \cdot$}
\put(48,18){$\scriptstyle 1$}
\put(41,9){$\scriptstyle \cdot$}
\put(49,9){$\scriptstyle \cdot$}
\put(57,9){$\scriptstyle \cdot$}
\put(49,0){$\scriptstyle \cdot$}
}
\put(-40,60)
{
\put(37,-2.5){\makebox(26,46)[c]}
\put(49,36){$\scriptstyle  \cdot$}
\put(41,27){$\scriptstyle \cdot$}
\put(49,27){$\scriptstyle \cdot$}
\put(57,27){$\scriptstyle \cdot$}
\put(49,18){$\scriptstyle \cdot$}
\put(40,9){$\scriptstyle 1$}
\put(49,9){$\scriptstyle \cdot$}
\put(57,9){$\scriptstyle \cdot$}
\put(49,0){$\scriptstyle \cdot$}
}
\put(0,60)
{
\put(37,-2.5){\dashbox{3}(26,46)[c]}
\put(49,36){$\scriptstyle  \cdot$}
\put(41,27){$\scriptstyle \cdot$}
\put(49,27){$\scriptstyle \cdot$}
\put(57,27){$\scriptstyle \cdot$}
\put(49,18){$\scriptstyle \cdot$}
\put(41,9){$\scriptstyle \cdot$}
\put(48,9){$\scriptstyle 1$}
\put(57,9){$\scriptstyle \cdot$}
\put(49,0){$\scriptstyle \cdot$}
}
\put(40,60)
{
\put(37,-2.5){\framebox(26,46)[c]}
\put(49,36){$\scriptstyle  \cdot$}
\put(41,27){$\scriptstyle \cdot$}
\put(49,27){$\scriptstyle \cdot$}
\put(57,27){$\scriptstyle \cdot$}
\put(49,18){$\scriptstyle \cdot$}
\put(41,9){$\scriptstyle \cdot$}
\put(49,9){$\scriptstyle \cdot$}
\put(56,9){$\scriptstyle 1$}
\put(49,0){$\scriptstyle \cdot$}
}
\put(0,0)
{
\put(37,-2.5){\framebox(26,46)[c]}
\put(49,36){$\scriptstyle  \cdot$}
\put(41,27){$\scriptstyle \cdot$}
\put(49,27){$\scriptstyle \cdot$}
\put(57,27){$\scriptstyle \cdot$}
\put(49,18){$\scriptstyle \cdot$}
\put(41,9){$\scriptstyle \cdot$}
\put(49,9){$\scriptstyle \cdot$}
\put(57,9){$\scriptstyle \cdot$}
\put(48,0){$\scriptstyle 1$}
}

%
\put(35,200){\vector(-1,0){10}}
\put(65,200){\vector(1,0){10}}
\put(35,80){\vector(-1,0){10}}
\put(65,80){\vector(1,0){10}}
\put(90,105){\vector(0,1){70}}
\put(10,175){\vector(0,-1){70}}
\put(50,235){\vector(0,-1){10}}
\put(50,165){\vector(0,1){10}}
\put(50,115){\vector(0,-1){10}}
\put(50,45){\vector(0,1){10}}
\put(75,225){\vector(-1,1){10}}
\put(75,175){\vector(-1,-1){10}}
\put(25,105){\vector(1,1){10}}
\put(25,55){\vector(1,-1){10}}
\put(40,-15){$y(0)$}
}
%
%
\put(120,170)
{
\put(0,240)
{
\put(37,-2.5){\dashbox{3}(26,46)[c]}
\put(46,36){-$\scriptstyle  1$}
\put(41,27){$\scriptstyle \cdot$}
\put(49,27){$\scriptstyle \cdot$}
\put(57,27){$\scriptstyle \cdot$}
\put(49,18){$\scriptstyle \cdot$}
\put(41,9){$\scriptstyle \cdot$}
\put(49,9){$\scriptstyle \cdot$}
\put(57,9){$\scriptstyle \cdot$}
\put(49,0){$\scriptstyle \cdot$}
}
\put(-40,180)
{
\put(37,-2.5){\dashbox{3}(26,46)[c]}
\put(49,36){$\scriptstyle  \cdot$}
\put(38,27){-$\scriptstyle 1$}
\put(49,27){$\scriptstyle \cdot$}
\put(57,27){$\scriptstyle \cdot$}
\put(49,18){$\scriptstyle \cdot$}
\put(41,9){$\scriptstyle \cdot$}
\put(49,9){$\scriptstyle \cdot$}
\put(57,9){$\scriptstyle \cdot$}
\put(49,0){$\scriptstyle \cdot$}
}
\put(0,180)
{
\put(37,-2.5){\framebox(26,46)[c]}
\put(48,36){$\scriptstyle  1$}
\put(41,27){$\scriptstyle \cdot$}
\put(48,27){$\scriptstyle 1$}
\put(57,27){$\scriptstyle \cdot$}
\put(48,18){$\scriptstyle 1$}
\put(41,9){$\scriptstyle \cdot$}
\put(49,9){$\scriptstyle \cdot$}
\put(57,9){$\scriptstyle \cdot$}
\put(49,0){$\scriptstyle \cdot$}
}
\put(40,180)
{
\put(37,-2.5){\makebox(26,46)[c]}
\put(49,36){$\scriptstyle  \cdot$}
\put(41,27){$\scriptstyle \cdot$}
\put(49,27){$\scriptstyle \cdot$}
\put(56,27){$\scriptstyle 1$}
\put(49,18){$\scriptstyle \cdot$}
\put(41,9){$\scriptstyle \cdot$}
\put(49,9){$\scriptstyle \cdot$}
\put(56,9){$\scriptstyle 1$}
\put(49,0){$\scriptstyle \cdot$}
}
\put(0,120)
{
\put(37,-2.5){\dashbox{3}(26,46)[c]}
\put(49,36){$\scriptstyle  \cdot$}
\put(41,27){$\scriptstyle \cdot$}
\put(49,27){$\scriptstyle \cdot$}
\put(57,27){$\scriptstyle \cdot$}
\put(46,18){-$\scriptstyle 1$}
\put(41,9){$\scriptstyle \cdot$}
\put(49,9){$\scriptstyle \cdot$}
\put(57,9){$\scriptstyle \cdot$}
\put(49,0){$\scriptstyle \cdot$}
}
\put(-40,60)
{
\put(37,-2.5){\makebox(26,46)[c]}
\put(49,36){$\scriptstyle  \cdot$}
\put(40,27){$\scriptstyle 1$}
\put(49,27){$\scriptstyle \cdot$}
\put(57,27){$\scriptstyle \cdot$}
\put(49,18){$\scriptstyle \cdot$}
\put(40,9){$\scriptstyle 1$}
\put(49,9){$\scriptstyle \cdot$}
\put(57,9){$\scriptstyle \cdot$}
\put(49,0){$\scriptstyle \cdot$}
}
\put(0,60)
{
\put(37,-2.5){\framebox(26,46)[c]}
\put(49,36){$\scriptstyle  \cdot$}
\put(41,27){$\scriptstyle \cdot$}
\put(49,27){$\scriptstyle \cdot$}
\put(57,27){$\scriptstyle \cdot$}
\put(48,18){$\scriptstyle 1$}
\put(41,9){$\scriptstyle \cdot$}
\put(48,9){$\scriptstyle 1$}
\put(57,9){$\scriptstyle \cdot$}
\put(48,0){$\scriptstyle 1$}
}
\put(40,60)
{
\put(37,-2.5){\dashbox{3}(26,46)[c]}
\put(49,36){$\scriptstyle  \cdot$}
\put(41,27){$\scriptstyle \cdot$}
\put(49,27){$\scriptstyle \cdot$}
\put(57,27){$\scriptstyle \cdot$}
\put(49,18){$\scriptstyle \cdot$}
\put(41,9){$\scriptstyle \cdot$}
\put(49,9){$\scriptstyle \cdot$}
\put(54,9){-$\scriptstyle 1$}
\put(49,0){$\scriptstyle \cdot$}
}
\put(0,0)
{
\put(37,-2.5){\dashbox{3}(26,46)[c]}
\put(49,36){$\scriptstyle  \cdot$}
\put(41,27){$\scriptstyle \cdot$}
\put(49,27){$\scriptstyle \cdot$}
\put(57,27){$\scriptstyle \cdot$}
\put(49,18){$\scriptstyle \cdot$}
\put(41,9){$\scriptstyle \cdot$}
\put(49,9){$\scriptstyle \cdot$}
\put(57,9){$\scriptstyle \cdot$}
\put(46,0){-$\scriptstyle 1$}
}

%
\put(25,200){\vector(1,0){10}}
\put(75,200){\vector(-1,0){10}}
\put(25,80){\vector(1,0){10}}
\put(75,80){\vector(-1,0){10}}
\put(90,175){\vector(0,-1){70}}
\put(10,105){\vector(0,1){70}}
\put(50,225){\vector(0,1){10}}
\put(50,175){\vector(0,-1){10}}
\put(50,105){\vector(0,1){10}}
\put(50,55){\vector(0,-1){10}}
\put(65,235){\vector(1,-1){10}}
\put(65,165){\vector(1,1){10}}
\put(35,115){\vector(-1,-1){10}}
\put(35,45){\vector(-1,1){10}}
\put(40,-15){$y(\frac{1}{2})$}
}
%
%
\put(240,170)
{
\put(0,240)
{
\put(37,-2.5){\framebox(26,46)[c]}
\put(49,36){$\scriptstyle \cdot$}
\put(41,27){$\scriptstyle \cdot$}
\put(48,27){$\scriptstyle 1$}
\put(57,27){$\scriptstyle \cdot$}
\put(48,18){$\scriptstyle 1$}
\put(41,9){$\scriptstyle \cdot$}
\put(49,9){$\scriptstyle \cdot$}
\put(57,9){$\scriptstyle \cdot$}
\put(49,0){$\scriptstyle \cdot$}
}
\put(-40,180)
{
\put(37,-2.5){\makebox(26,46)[c]}
\put(49,36){$\scriptstyle  \cdot$}
\put(38,27){-$\scriptstyle 1$}
\put(49,27){$\scriptstyle \cdot$}
\put(57,27){$\scriptstyle \cdot$}
\put(49,18){$\scriptstyle \cdot$}
\put(41,9){$\scriptstyle \cdot$}
\put(49,9){$\scriptstyle \cdot$}
\put(57,9){$\scriptstyle \cdot$}
\put(49,0){$\scriptstyle \cdot$}
}
\put(0,180)
{
\put(37,-2.5){\dashbox{3}(26,46)[c]}
\put(46,36){-$\scriptstyle  1$}
\put(41,27){$\scriptstyle \cdot$}
\put(46,27){-$\scriptstyle 1$}
\put(57,27){$\scriptstyle \cdot$}
\put(46,18){-$\scriptstyle 1$}
\put(41,9){$\scriptstyle \cdot$}
\put(49,9){$\scriptstyle \cdot$}
\put(57,9){$\scriptstyle \cdot$}
\put(49,0){$\scriptstyle \cdot$}
}
\put(40,180)
{
\put(37,-2.5){\framebox(26,46)[c]}
\put(49,36){$\scriptstyle  \cdot$}
\put(41,27){$\scriptstyle \cdot$}
\put(49,27){$\scriptstyle \cdot$}
\put(56,27){$\scriptstyle 1$}
\put(49,18){$\scriptstyle \cdot$}
\put(41,9){$\scriptstyle \cdot$}
\put(49,9){$\scriptstyle \cdot$}
\put(56,9){$\scriptstyle 1$}
\put(49,0){$\scriptstyle \cdot$}
}
\put(0,120)
{
\put(37,-2.5){\framebox(26,46)[c]}
\put(48,36){$\scriptstyle  1$}
\put(41,27){$\scriptstyle \cdot$}
\put(48,27){$\scriptstyle 1$}
\put(57,27){$\scriptstyle \cdot$}
\put(48,18){$\scriptstyle 1$}
\put(41,9){$\scriptstyle \cdot$}
\put(48,9){$\scriptstyle 1$}
\put(57,9){$\scriptstyle \cdot$}
\put(48,0){$\scriptstyle 1$}
}
\put(-40,60)
{
\put(37,-2.5){\framebox(26,46)[c]}
\put(49,36){$\scriptstyle  \cdot$}
\put(40,27){$\scriptstyle 1$}
\put(49,27){$\scriptstyle \cdot$}
\put(57,27){$\scriptstyle \cdot$}
\put(49,18){$\scriptstyle \cdot$}
\put(40,9){$\scriptstyle 1$}
\put(49,9){$\scriptstyle \cdot$}
\put(57,9){$\scriptstyle \cdot$}
\put(49,0){$\scriptstyle \cdot$}
}
\put(0,60)
{
\put(37,-2.5){\dashbox{3}(26,46)[c]}
\put(49,36){$\scriptstyle  \cdot$}
\put(41,27){$\scriptstyle \cdot$}
\put(49,27){$\scriptstyle \cdot$}
\put(57,27){$\scriptstyle \cdot$}
\put(46,18){-$\scriptstyle 1$}
\put(41,9){$\scriptstyle \cdot$}
\put(46,9){-$\scriptstyle 1$}
\put(57,9){$\scriptstyle \cdot$}
\put(46,0){-$\scriptstyle 1$}
}
\put(40,60)
{
\put(37,-2.5){\makebox(26,46)[c]}
\put(49,36){$\scriptstyle  \cdot$}
\put(41,27){$\scriptstyle \cdot$}
\put(49,27){$\scriptstyle \cdot$}
\put(57,27){$\scriptstyle \cdot$}
\put(49,18){$\scriptstyle \cdot$}
\put(41,9){$\scriptstyle \cdot$}
\put(49,9){$\scriptstyle \cdot$}
\put(54,9){-$\scriptstyle 1$}
\put(49,0){$\scriptstyle \cdot$}
}
\put(0,0)
{
\put(37,-2.5){\framebox(26,46)[c]}
\put(49,36){$\scriptstyle  \cdot$}
\put(41,27){$\scriptstyle \cdot$}
\put(49,27){$\scriptstyle \cdot$}
\put(57,27){$\scriptstyle \cdot$}
\put(48,18){$\scriptstyle 1$}
\put(41,9){$\scriptstyle \cdot$}
\put(48,9){$\scriptstyle 1$}
\put(57,9){$\scriptstyle \cdot$}
\put(49,0){$\scriptstyle \cdot$}
}

%
\put(35,200){\vector(-1,0){10}}
\put(65,200){\vector(1,0){10}}
\put(35,80){\vector(-1,0){10}}
\put(65,80){\vector(1,0){10}}
\put(90,175){\vector(0,-1){70}}
\put(10,105){\vector(0,1){70}}
\put(50,235){\vector(0,-1){10}}
\put(50,165){\vector(0,1){10}}
\put(50,115){\vector(0,-1){10}}
\put(50,45){\vector(0,1){10}}
\put(25,225){\vector(1,1){10}}
\put(25,175){\vector(1,-1){10}}
\put(75,105){\vector(-1,1){10}}
\put(75,55){\vector(-1,-1){10}}
\put(40,-15){$y(1)$}
}
%
%
\put(60,-80)
{
\put(0,240)
{
\put(37,-2.5){\dashbox{3}(26,46)[c]}
\put(49,36){$\scriptstyle  \cdot$}
\put(41,27){$\scriptstyle \cdot$}
\put(46,27){-$\scriptstyle 1$}
\put(57,27){$\scriptstyle \cdot$}
\put(46,18){-$\scriptstyle 1$}
\put(41,9){$\scriptstyle \cdot$}
\put(49,9){$\scriptstyle \cdot$}
\put(57,9){$\scriptstyle \cdot$}
\put(49,0){$\scriptstyle \cdot$}
}
\put(-40,180)
{
\put(37,-2.5){\makebox(26,46)[c]}
\put(49,36){$\scriptstyle  \cdot$}
\put(41,27){$\scriptstyle \cdot$}
\put(49,27){$\scriptstyle \cdot$}
\put(57,27){$\scriptstyle \cdot$}
\put(49,18){$\scriptstyle \cdot$}
\put(40,9){$\scriptstyle 1$}
\put(49,9){$\scriptstyle \cdot$}
\put(57,9){$\scriptstyle \cdot$}
\put(49,0){$\scriptstyle \cdot$}
}
\put(0,180)
{
\put(37,-2.5){\framebox(26,46)[c]}
\put(49,36){$\scriptstyle  \cdot$}
\put(41,27){$\scriptstyle \cdot$}
\put(48,27){$\scriptstyle 1$}
\put(57,27){$\scriptstyle \cdot$}
\put(48,18){$\scriptstyle 1$}
\put(41,9){$\scriptstyle \cdot$}
\put(48,9){$\scriptstyle 1$}
\put(57,9){$\scriptstyle \cdot$}
\put(48,0){$\scriptstyle 1$}
}
\put(40,180)
{
\put(37,-2.5){\dashbox{3}(26,46)[c]}
\put(49,36){$\scriptstyle  \cdot$}
\put(41,27){$\scriptstyle \cdot$}
\put(49,27){$\scriptstyle \cdot$}
\put(54,27){-$\scriptstyle 1$}
\put(49,18){$\scriptstyle \cdot$}
\put(41,9){$\scriptstyle \cdot$}
\put(49,9){$\scriptstyle \cdot$}
\put(54,9){-$\scriptstyle 1$}
\put(49,0){$\scriptstyle \cdot$}
}
\put(0,120)
{
\put(37,-2.5){\dashbox{3}(26,46)[c]}
\put(46,36){-$\scriptstyle  1$}
\put(41,27){$\scriptstyle \cdot$}
\put(46,27){-$\scriptstyle 1$}
\put(57,27){$\scriptstyle \cdot$}
\put(46,18){-$\scriptstyle 1$}
\put(41,9){$\scriptstyle \cdot$}
\put(46,9){-$\scriptstyle 1$}
\put(57,9){$\scriptstyle \cdot$}
\put(46,0){-$\scriptstyle 1$}
}
\put(-40,60)
{
\put(37,-2.5){\dashbox{3}(26,46)[c]}
\put(49,36){$\scriptstyle  \cdot$}
\put(38,27){-$\scriptstyle 1$}
\put(49,27){$\scriptstyle \cdot$}
\put(57,27){$\scriptstyle \cdot$}
\put(49,18){$\scriptstyle \cdot$}
\put(38,9){-$\scriptstyle 1$}
\put(49,9){$\scriptstyle \cdot$}
\put(57,9){$\scriptstyle \cdot$}
\put(49,0){$\scriptstyle \cdot$}
}
\put(0,60)
{
\put(37,-2.5){\framebox(26,46)[c]}
\put(48,36){$\scriptstyle  1$}
\put(41,27){$\scriptstyle \cdot$}
\put(48,27){$\scriptstyle 1$}
\put(57,27){$\scriptstyle \cdot$}
\put(48,18){$\scriptstyle 1$}
\put(41,9){$\scriptstyle \cdot$}
\put(48,9){$\scriptstyle 1$}
\put(57,9){$\scriptstyle \cdot$}
\put(49,0){$\scriptstyle \cdot$}
}
\put(40,60)
{
\put(37,-2.5){\makebox(26,46)[c]}
\put(49,36){$\scriptstyle  \cdot$}
\put(41,27){$\scriptstyle \cdot$}
\put(49,27){$\scriptstyle \cdot$}
\put(56,27){$\scriptstyle 1$}
\put(49,18){$\scriptstyle \cdot$}
\put(41,9){$\scriptstyle \cdot$}
\put(49,9){$\scriptstyle \cdot$}
\put(57,9){$\scriptstyle \cdot$}
\put(49,0){$\scriptstyle \cdot$}
}
\put(0,0)
{
\put(37,-2.5){\dashbox{3}(26,46)[c]}
\put(49,36){$\scriptstyle  \cdot$}
\put(41,27){$\scriptstyle \cdot$}
\put(49,27){$\scriptstyle \cdot$}
\put(57,27){$\scriptstyle \cdot$}
\put(46,18){-$\scriptstyle 1$}
\put(41,9){$\scriptstyle \cdot$}
\put(46,9){-$\scriptstyle 1$}
\put(57,9){$\scriptstyle \cdot$}
\put(49,0){$\scriptstyle \cdot$}
}

%
\put(25,200){\vector(1,0){10}}
\put(75,200){\vector(-1,0){10}}
\put(25,80){\vector(1,0){10}}
\put(75,80){\vector(-1,0){10}}
\put(90,105){\vector(0,1){70}}
\put(10,175){\vector(0,-1){70}}
\put(50,225){\vector(0,1){10}}
\put(50,175){\vector(0,-1){10}}
\put(50,105){\vector(0,1){10}}
\put(50,55){\vector(0,-1){10}}
\put(35,245){\vector(-1,-1){10}}
\put(35,165){\vector(-1,1){10}}
\put(65,115){\vector(1,-1){10}}
\put(65,45){\vector(1,1){10}}
\put(40,-15){$y(\frac{3}{2})$}
}
%
%
\put(180,-80)
{
\put(0,240)
{
\put(37,-2.5){\framebox(26,46)[c]}
\put(49,36){$\scriptstyle  \cdot$}
\put(41,27){$\scriptstyle \cdot$}
\put(49,27){$\scriptstyle \cdot$}
\put(57,27){$\scriptstyle \cdot$}
\put(49,18){$\scriptstyle \cdot$}
\put(41,9){$\scriptstyle \cdot$}
\put(48,9){$\scriptstyle 1$}
\put(57,9){$\scriptstyle \cdot$}
\put(48,0){$\scriptstyle 1$}
}
\put(-40,180)
{
\put(37,-2.5){\framebox(26,46)[c]}
\put(49,36){$\scriptstyle  \cdot$}
\put(41,27){$\scriptstyle \cdot$}
\put(49,27){$\scriptstyle \cdot$}
\put(57,27){$\scriptstyle \cdot$}
\put(49,18){$\scriptstyle \cdot$}
\put(40,9){$\scriptstyle 1$}
\put(49,9){$\scriptstyle \cdot$}
\put(57,9){$\scriptstyle \cdot$}
\put(49,0){$\scriptstyle \cdot$}
}
\put(0,180)
{
\put(37,-2.5){\dashbox{3}(26,46)[c]}
\put(49,36){$\scriptstyle  \cdot$}
\put(41,27){$\scriptstyle \cdot$}
\put(46,27){-$\scriptstyle 1$}
\put(57,27){$\scriptstyle \cdot$}
\put(46,18){-$\scriptstyle 1$}
\put(41,9){$\scriptstyle \cdot$}
\put(46,9){-$\scriptstyle 1$}
\put(57,9){$\scriptstyle \cdot$}
\put(46,0){-$\scriptstyle 1$}
}
\put(40,180)
{
\put(37,-2.5){\makebox(26,46)[c]}
\put(49,36){$\scriptstyle  \cdot$}
\put(41,27){$\scriptstyle \cdot$}
\put(49,27){$\scriptstyle \cdot$}
\put(54,27){-$\scriptstyle 1$}
\put(49,18){$\scriptstyle \cdot$}
\put(41,9){$\scriptstyle \cdot$}
\put(49,9){$\scriptstyle \cdot$}
\put(54,9){-$\scriptstyle 1$}
\put(49,0){$\scriptstyle \cdot$}
}
\put(0,120)
{
\put(37,-2.5){\framebox(26,46)[c]}
\put(49,36){$\scriptstyle  \cdot$}
\put(41,27){$\scriptstyle \cdot$}
\put(48,27){$\scriptstyle 1$}
\put(57,27){$\scriptstyle \cdot$}
\put(48,18){$\scriptstyle 1$}
\put(41,9){$\scriptstyle \cdot$}
\put(48,9){$\scriptstyle 1$}
\put(57,9){$\scriptstyle \cdot$}
\put(49,0){$\scriptstyle \cdot$}
}
\put(-40,60)
{
\put(37,-2.5){\makebox(26,46)[c]}
\put(49,36){$\scriptstyle  \cdot$}
\put(38,27){-$\scriptstyle 1$}
\put(49,27){$\scriptstyle \cdot$}
\put(57,27){$\scriptstyle \cdot$}
\put(49,18){$\scriptstyle \cdot$}
\put(38,9){-$\scriptstyle 1$}
\put(49,9){$\scriptstyle \cdot$}
\put(57,9){$\scriptstyle \cdot$}
\put(49,0){$\scriptstyle \cdot$}
}
\put(0,60)
{
\put(37,-2.5){\dashbox{3}(26,46)[c]}
\put(46,36){-$\scriptstyle  1$}
\put(41,27){$\scriptstyle \cdot$}
\put(46,27){-$\scriptstyle 1$}
\put(57,27){$\scriptstyle \cdot$}
\put(46,18){-$\scriptstyle 1$}
\put(41,9){$\scriptstyle \cdot$}
\put(46,9){-$\scriptstyle 1$}
\put(57,9){$\scriptstyle \cdot$}
\put(49,0){$\scriptstyle \cdot$}
}
\put(40,60)
{
\put(37,-2.5){\makebox(26,46)[c]}
\put(49,36){$\scriptstyle  \cdot$}
\put(41,27){$\scriptstyle \cdot$}
\put(49,27){$\scriptstyle \cdot$}
\put(56,27){$\scriptstyle 1$}
\put(49,18){$\scriptstyle \cdot$}
\put(41,9){$\scriptstyle \cdot$}
\put(49,9){$\scriptstyle \cdot$}
\put(57,9){$\scriptstyle \cdot$}
\put(49,0){$\scriptstyle \cdot$}
}
\put(0,0)
{
\put(37,-2.5){\framebox(26,46)[c]}
\put(48,36){$\scriptstyle  1$}
\put(41,27){$\scriptstyle \cdot$}
\put(48,27){$\scriptstyle 1$}
\put(57,27){$\scriptstyle \cdot$}
\put(49,18){$\scriptstyle \cdot$}
\put(41,9){$\scriptstyle \cdot$}
\put(49,9){$\scriptstyle \cdot$}
\put(57,9){$\scriptstyle \cdot$}
\put(49,0){$\scriptstyle \cdot$}
}

%
\put(35,200){\vector(-1,0){10}}
\put(65,200){\vector(1,0){10}}
\put(35,80){\vector(-1,0){10}}
\put(65,80){\vector(1,0){10}}
\put(90,105){\vector(0,1){70}}
\put(10,175){\vector(0,-1){70}}
\put(50,235){\vector(0,-1){10}}
\put(50,165){\vector(0,1){10}}
\put(50,115){\vector(0,-1){10}}
\put(50,45){\vector(0,1){10}}
\put(75,225){\vector(-1,1){10}}
\put(75,175){\vector(-1,-1){10}}
\put(25,105){\vector(1,1){10}}
\put(25,55){\vector(1,-1){10}}
\put(40,-15){$y(2)$}
}
\end{picture}
\caption{Tropical Y-system of type $B_2$ at level 3
in the region $0 \leq u \leq 2$.
}
\label{fig:Blev3-2}
\end{figure}

We have a generalization of Proposition
\ref{prop:lev2} for the tropical
Y-systems at higher levels.

\begin{proposition}
\label{prop:levh}
Let $\ell> 2$ be an integer.
 For 
$[\mathcal{G}_Y(B,y)]_{\mathbf{T}}$
with $B=B_{\ell}(B_r)$, the following facts hold.
\par
(i) Let $u$ be in the region $0\le u < \ell$.
For any $(\mathbf{i},u):\mathbf{p}_+$,
the  monomial $[y_{\mathbf{i}}(u)]_{\mathbf{T}}$
is positive.

\par
(ii) Let $u$ be in the region $-h^{\vee}\le u < 0$.
\begin{itemize}
\item[\em (a)] Let $\mathbf{i}\in \mathbf{I}^{\circ}
\sqcup \mathbf{I}^{\bullet}_-
$.
For any $(\mathbf{i},u):\mathbf{p}_+$,
the  monomial $[y_{\mathbf{i}}(u)]_{\mathbf{T}}$
is negative.
\item[\em (b)]
 Let $\mathbf{i}\in \mathbf{I}^{\bullet}_+$.
For any $(\mathbf{i},u):\mathbf{p}_+$,
the  monomial $[y_{\mathbf{i}}(u)]_{\mathbf{T}}$
is negative for $u=-1,-3,\dots$
and positive for $u=-2,-4,\dots$.
\end{itemize}
\par
\par
(iii)
$y_{ii'}(\ell)=y_{i,\ell-i'}^{-1}$ if $i\neq r$ and
$y_{r,2\ell-i'}^{-1}$ if $i=r$.
\par
(iv) $y_{ii'}(-h^{\vee})=
y_{2r-i,i'}^{-1}$.
\end{proposition}

\begin{proof}
This proposition is a consequence of the {\em factorization
property} of the tropical Y-system found in \cite{Nkn}
for simply laced case.
Roughly speaking,
 in the region $-h^{\vee}\leq u \leq 0$,
the system is factorized into the `level 2 pieces',
while 
in the region $0\leq u \leq \ell$,
the system is factorized into the `type $A$ pieces'.

First, we consider the region $-h^{\vee}\leq u < 0$.
Let us concentrate on the simplest nontrivial example
$B_2$ with level 3.
It turns out that this example is almost general enough.
In Figure \ref{fig:Blev3}, all the
variables $[y_{\mathbf{i}}(u)]_{\mathbf{T}}$ in the region
$-2\leq u \leq 0$
are given explicitly.
One had better view it in the backward direction from
$u=0$ to $-2$.
If we look at only the first three rows from the bottom,
we observe that the mutations occur in exactly the same
pattern as the level 2 case in Figure
\ref{fig:tropB22}.
So is for the last three rows
(with $180^{\circ}$ rotation).
This is the factorization property.
It occurs  due to the coordination of the
mutation sequence \eqref{eq:mutseq}
and the positivity/negativity
in Proposition \ref{prop:lev2}.
For example, in Figure \ref{fig:Blev3},
let us look at the vertical arrows between the second and the fourth
rows.
At any mutation point $(\mathbf{i},u):\mathbf{p}_-$
these arrows are {\em incoming}
while $[y_{\mathbf{i}}(u)]_{\mathbf{T}}$ are {\em positive}
by \eqref{eq:ymu} and Proposition \ref{prop:lev2} (ii).
Therefore, by Lemma \ref{lem:mutation},
these arrows can be forgotten during mutations.
By a similar reason, the  variables in the third rows
are not affected by mutations in the second and the fourth rows.
Therefore, as long as the positivity in the
second and the fourth rows continue,
so does the factorization;
hence (ii) holds.
Moreover, (iv) holds because it does for level 2.
This argument is also applicable to any rank $r$ and level $\ell$
because of the definitions of the quiver $Q_{\ell}(B_r)$
and the mutation sequence \eqref{eq:mutseq}.

Next, we consider the region $0 \leq u < \ell$.
Again, let us concentrate on the case
$B_2$ with level 3.
In Figure \ref{fig:Blev3-2}, all the
variables $[y_{\mathbf{i}}(u)]_{\mathbf{T}}$ in the region
$0\leq u \leq 2$
are given explicitly.
One had better view it in the forward direction from
$u=0$ to $2$.
We observe that each column does not affect each other
in mutations.
More precisely, the first and third columns (from the left)
are in the same mutation pattern as
the tropical Y-system of type $A_2$ studied in \cite{FZ2,FZ3}
(see also \cite{Nkn} for more detail).
The second column is in the same mutation pattern 
 as the tropical Y-system of type $A_5$
(at the {\em twice} faster pace).
Due to \cite[Proposition 10.7]{FZ3},
 $[y_{\mathbf{i}}(u)]_{\mathbf{T}}$ for $(\mathbf{i},u):
\mathbf{p}_+$ are positive in the region
$0\leq u < 3$
(3 is the Coxeter number of $A_2$
{\em and} the {\em half} of the Coxeter number of $A_5$).
The factorization occurs by the same reason as before;
at any mutation point $(\mathbf{i},u):\mathbf{p}_+$
the arrows between adjacent columns are {\em incoming}
while $[y_{\mathbf{i}}(u)]_{\mathbf{T}}$ are {\em positive}.
Therefore, these arrows can be forgotten,
and we have (ii).
Finally, (iii) is also a consequence 
of \cite[Proposition 10.7]{FZ3}.
Again,
this argument is also applicable to any rank $r$ and level $\ell$
by replacing the number 3
 in the above with $\ell$, which is the Coxeter number
of $A_{\ell-1}$ and the half of the Coxeter number
of $A_{2\ell-1}$.
\end{proof}

We obtain two important corollaries of
Propositions \ref{prop:lev2} and  \ref{prop:levh}.

\begin{theorem}
\label{thm:tYperiod}
 For $[\mathcal{G}_Y(B,y)]_{\mathbf{T}}$
with $B=B_{\ell}(B_r)$,
the following relations hold.
\par
(i) Half periodicity: 
$[y_{\mathbf{i}}(u+h^{\vee}+\ell)]_{\mathbf{T}}
=[y_{\boldsymbol\omega(\mathbf{i})}(u)]_{\mathbf{T}}$.
\par
(ii) 
 Full periodicity: 
$[y_{\mathbf{i}}(u+2(h^{\vee}+\ell))]_{\mathbf{T}}
=[y_{\mathbf{i}}(u)]_{\mathbf{T}}$.
\end{theorem}
\begin{proof}
(i) follows from (iii) and (iv) of
 Propositions
\ref{prop:lev2} and \ref{prop:levh}.
 (ii) follows from (i).
\end{proof}

We remark that the half periodicity above
is compatible with the one for mutation matrices;
namely, set
$B(u)=B:=B_{\ell}(B_r),-B,\boldsymbol{r}(B),
-\boldsymbol{r}(B)$
for $u\equiv 0,\frac{1}{2},1,\frac{3}{2}$ {\rm mod} $2\mathbb{Z}$,
respectively.
Then,
\begin{align}
\label{eq:Bperiod}
B(u+h^{\vee}+\ell)=\boldsymbol{\omega}(B(u))
\end{align}
holds due to Lemma \ref{lem:Qmut} (ii).

\begin{theorem}
\label{thm:levhd}
For $[\mathcal{G}_Y(B,y)]_{\mathbf{T}}$
with $B=B_{\ell}(B_r)$,
let $N_+$ and $N_-$ denote the
total numbers of the positive and negative monomials,
respectively,
among $[y_{\mathbf{i}}(u)]_{\mathbf{T}}$
for $(\mathbf{i},u):\mathbf{p}_+$
in the region $0\leq u < 2(h^{\vee}+\ell)$.
Then, we have
\begin{align}
N_+=2\ell(\ell r+\ell-1),
\quad
N_-=2r(2\ell r -2r+1).
\end{align}
\end{theorem}
\begin{proof}
This follows from (i) and (ii) of  Propositions
\ref{prop:lev2} and \ref{prop:levh},
and (i)
in Theorem \ref{thm:tYperiod}.
\end{proof}

\section{2-Calabi-Yau realization
and periodicity}

In this section we prove   Theorems \ref{thm:Tperiod}
and \ref{thm:Yperiod}.
Our method is based on the 2-Calabi-Yau realizations of
cluster algebras. They are triangulated categories satisfying
2-Calabi-Yau property, and their categorical structure realizes
the combinatorial structure of cluster algebras.

\subsection{Periodicity theorem for cluster algebras.}

In this subsection we assume
that $B=(B_{ij})_{i,j\in I}$
is an arbitrary {\em skew-symmetric} (integer) matrix.
Let ${\mathcal{A}}(B,x,y)$ be the cluster algebra
with coefficients in the {\em universal semifield}
$\mathbb{Q}_{\mathrm{sf}}(y)$, where $(B,x,y)$ is the initial
seed.
Alternatively, we may consider the {\em cluster pattern}
assigned to the $n$-regular tree $\mathbb{T}_n$, $n=|I|$,
with the following data (see Section 2.1).
Let $Q$ be the quiver (without loops or 2-cycles)
corresponding to $B$ with the set of vertices $I$.
We fix a vertex $t_0\in {\mathbb{T}}_n$, and assign the initial seed
to $t_0$, $(Q(t_0),x(t_0),y(t_0))=(Q,x,y)$.
Then we have  a seed $(Q(t),x(t),y(t))$ for each $t\in{\mathbb{T}}_n$,
where we identify the quiver $Q(t)$ and the corresponding
skew-symmetric matrix $B(t)$ in each seed.

Let  ${\rm Trop}(y)$ be the tropical semifield.
As in Section 3.1, we define the
tropical evaluation of $y_i(t)$
\begin{align}
\label{eq:tropy}
[y_i(t)]_{{\mathbf{T}}}\in {\rm Trop}(y)
\end{align}
as the image of $y_i(t)\in
\mathbb{Q}_{\mathrm{sf}}(y)$ under the natural map
$\pi_{\mathbf{T}}:
\mathbb{Q}_{\mathrm{sf}}(y)\to {\rm Trop}(y)$.
Here we call $[y_i(t)]_{{\mathbf{T}}}$'s  the {\em principal coefficients}
in accordance with the nomenclature of \cite{FZ3}.
(They may be also called the {\em tropical $Y$-variables}.)

For an automorphism $\omega:I\to I$ of $I$, we define a new quiver $\omega(Q)$ with the same set $I$ of vertices by
drawing an arrow $\omega(a):\omega(i)\to\omega(j)$ in $\omega(Q)$
for each arrow $a:i\to j$ in $Q$.

The following theorem is Observation 1 explained in Section 1.5.
While it is crucial in our proof of the periodicities
of the T and Y-systems in this paper and also in \cite{IIKKN},
it is also expected to be  useful to other applications.

\begin{theorem}[Periodicity theorem]
\label{thm:XXX}
Let $B$ be an arbitrary skew-symmetric matrix,
and let $Q$ be the quiver corresponding to $B$.
Let $(Q(t),x(t),y(t))$ be the seed at
 $t\in \mathbb{T}_n$ for ${\mathcal{A}}(B,x,y)$ as above.
Suppose that there exists some $t\in{\mathbb{T}}_n$
and an automorphism  $\omega$ of $I$
such that $[y_i(t)]_{{\mathbf{T}}}=[y_{\omega(i)}]_{{\mathbf{T}}}$
 holds for any $i\in I$.
Then we have
\begin{align}
\begin{split}
Q(t)&=\omega^{-1}(Q)
\quad (\mbox{equivalently, $B_{ij}(t)=B_{\omega(i)\omega(j)}$}),
\\
x_i(t)&=x_{\omega(i)}\quad (i\in I) ,\\
y_i(t)&=y_{\omega(i)}\quad (i\in I).
\end{split}
\end{align}
In particular, 
 the periodicity of seeds of $\mathcal{A}(B,x,y)$
  coincides with the periodicity of
principal coefficient tuples.
\end{theorem}

Let ${\mathcal{A}}_{\bullet}(B,x,y)$ be the cluster algebra with coefficients
in the {\em tropical semifield} ${\rm Trop}(y)$,
where $(B,x,y)$ is the initial seed.
It is also called the cluster algebra with \emph{principal coefficients}
  \cite{FZ3}.
Note that a coefficient $y_i(t)$ in ${\mathcal{A}}_{\bullet}(B,x,y)$
coincides with $[y_i(t)]_{\mathbf{T}}$ in \eqref{eq:tropy}.

According to \cite[Theorem 4.6]{FZ3} and its proof,
Theorem \ref{thm:XXX} reduces to
the following result for ${\mathcal{A}}_{\bullet}(B,x,y)$.

\begin{theorem}\label{thm:reduced}
Let $B$ and $Q$ be the same as in Theorem \ref{thm:XXX}.
Let $(Q(t),x(t),y(t))$ be the seed at
 $t\in \mathbb{T}_n$ for ${\mathcal{A}}_{\bullet}(B,x,y)$.
Suppose that there exists some $t\in{\mathbb{T}}_n$
and an automorphism  $\omega$ of $I$
such that $y_i(t)=y_{\omega(i)}$ holds for any $i\in I$.
Then, we have
\begin{align}
\begin{split}
Q(t)&=\omega^{-1}(Q),\\
x_i(t)&=x_{\omega(i)}\quad (i\in I).
\end{split}
\end{align}
\end{theorem}

We give a proof of Theorem \ref{thm:reduced} in 
Section 5.3.

\begin{remark}
 Theorem \ref{thm:reduced} was essentially conjectured by Fomin-Zelevinsky
\cite[Conjecture 4.7]{FZ3} for an {\em arbitrary skew-symmetrizable}
matrix $B$.
Therefore, we partly prove the conjecture
for a {\em skew-symmetric} matrix $B$.
Note that our claim  is a little stronger than \cite[Conjecture 4.7]{FZ3}
because the periodicity
of the {\em principal parts} of exchange matrices also follows from
the periodicity of principal coefficients.
(Meanwhile, the periodicity
of the {\em complementary parts} of exchange matrices
coincides with the periodicity 
of principal coefficients by definition \cite{FZ3}.)

\begin{remark}
Since Theorems \ref{thm:XXX} and \ref{thm:reduced} involve quivers corresponding to arbitrary
skew-symmetric matrices, the categories appearing in their proof 
will in general have infinite-dimensional morphism spaces. It is
an interesting question to ask whether the quivers which appear
in our applications of the two theorems actually admit categorifications
by Hom-finite $2$-CY categories. We conjecture that this is indeed
the case for all levels $l\geq 2$. For $l=2$, it is clear since
the quiver $Q_2(B_r)$ is of finite cluster type $D_{2r+1}$.
However, for $l>2$, the conjecture appears to be nontrivial.
\end{remark}

\end{remark}

\subsection{2-Calabi-Yau realization of ${\mathcal{A}}_{\bullet}(B,x,y)$} 

Our proof of Theorem \ref{thm:reduced} uses the 
categorification of the cluster algebra
${\mathcal{A}}_{\bullet}(B,x,y)$ by a 
certain 2-Calabi-Yau category \cite{Pa,FK,Ke1,A,KY,Pl1,Pl2}.
Here we review the recent result by Plamondon \cite{Pl1,Pl2}.

Let $Q$ be the quiver corresponding to an arbitrary
skew-symmetric matrix $B$.
Define the \emph{principal extension} $\widetilde{Q}$ of $Q$
as the quiver obtained from $Q$ by adding
a new vertex $i'$ and an arrow $i'\to i$ for each $i\in I$.
Thus the set of vertices in $\widetilde{Q}$
is given by $\widetilde{I}:=I\sqcup I'$ with
 $I':=\{i'\ |\ i\in I\}$.
By mutations one can associate 
a quiver $\widetilde{Q}(t)$ with  each $t\in{\mathbb{T}}_n$,
where $\widetilde{Q}(t)$ contains $Q(t)$ as a full subquiver.
Note that we do not make mutations for `frozen indices' $i'\in I'$.

We fix a base field $K$ to be an infinite one.
Since $\widetilde{Q}$ does not have loops and 2-cycles,
we have the following result by \cite[Corollary 7.4]{DWZ1}.

\begin{proposition}
There exists a non-degenerate potential $W$ on $\widetilde{Q}$.
\end{proposition}

{}From now on we assume $W$ is a non-degenerate potential on $\widetilde{Q}$.
We denote by
\begin{align}
{\mathcal{C}}:={\mathcal{C}}_{(\widetilde{Q},W)}
\end{align}
the \emph{cluster category} associated to the quiver with potential $(\widetilde{Q},W)$, which is not necessarily Hom-finite \cite{A,KY,Pl1}.
The category ${\mathcal{C}}$ canonically contains a rigid object
\begin{align}
T=\bigoplus_{i\in \widetilde{I}}T_i\in{\mathcal{C}}
\end{align}
such that ${\operatorname{End}\nolimits}_{{\mathcal{C}}}(T)$
 is isomorphic to the Jacobian algebra of 
$(\widetilde{Q},W)$.
For each $t\in{\mathbb{T}}_n$, we have a rigid object
\begin{align}
T(t)=\bigoplus_{i\in \widetilde{I}}T_i(t)\in{\mathcal{C}}
\end{align}
by applying successive mutations (see \cite[Section 2.6]{Pl1}).
We have $T_{i'}(t)=T_{i'}$ for any $i'\in I'$.

{}From the definition of $T(t)$ and the non-degeneracy of $(\widetilde{Q},W)$
we have the following description of $\widetilde{Q}(t)$.

\begin{proposition}\label{quiver and object}
For each $t\in{\mathbb{T}}_n$, the quiver
 of ${\operatorname{End}\nolimits}_{{\mathcal{C}}}(T(t))$
 is $\widetilde{Q}(t)$,
where each vertex $i\in\widetilde{I}$ corresponds
 to the direct summand $T_i(t)$
of $T(t)$.
\end{proposition}

As usual we denote by ${\operatorname{add}\nolimits}\, T(t)$
 the full subcategory of ${\mathcal{C}}$ consisting of
all direct summands of finite direct sums of copies of $T(t)$.
We denote by
\begin{align}
{\operatorname{pr}\nolimits}\, T(t)
\end{align}
the full subcategory of ${\mathcal{C}}$ consisting of objects
$M\in{\mathcal{C}}$ such that there exists a triangle
\begin{eqnarray}\label{resolution}
T''\to T'\to M\to T''[1]
\end{eqnarray}
in ${\mathcal{C}}$ with $T',T''\in{\operatorname{add}\nolimits}\, T(t)$.

\begin{proposition}[{\cite[Proposition 2.7, Corollary 2.12]{Pl1}}]
\label{KS}
(1) We have ${\operatorname{pr}\nolimits}\,
 T(t)={\operatorname{pr}\nolimits}\, T$ for any $t\in{\mathbb{T}}_n$.
\par
(2) The category ${\operatorname{pr}\nolimits}\, T$
 is Krull-Schmidt in the sense that any object can be
written as a finite direct sum of objects whose endomorphism rings are local.
\end{proposition}

Now let us introduce the following notion.

\begin{definition}
Let $t\in{\mathbb{T}}_n$.
For an object $T'=\bigoplus_{i\in\widetilde{I}}T_i(t)^{\ell_i}$
 in ${\operatorname{add}\nolimits}\, T(t)$,
we put
\begin{align}
[T']_{T(t)}:=(\ell_i)_{i\in\widetilde{I}}\in{\mathbb{Z}}^{\widetilde{I}}.
\end{align}
For an object $M\in{\operatorname{pr}\nolimits}\, T(t)$,
 we take a triangle \eqref{resolution}
and define the \emph{index} of $M$ by
\begin{align}
{\operatorname{ind}\nolimits}_{T(t)}(M)
:=[T']_{T(t)}-[T'']_{T(t)}\in{\mathbb{Z}}^{\widetilde{I}}.
\end{align}
This is independent of the choice of the triangle \eqref{resolution}
by Proposition \ref{KS}(2).
\end{definition}

We have the following relationship between indices and principal
coefficients. 

\begin{proposition}[{\cite[Corollary 3.10]{Pl2}, \cite[Theorem 7.13(b)]{Ke1}}]
\label{y and index}
Let $y_i(t)$ be a coefficient in 
${\mathcal{A}}_{\bullet}(B,x,y)$.
For $t\in{\mathbb{T}}_n$, we put $y_j(t)=\prod_{i\in I}y_i^{c_{ij}(t)}$
for any $j\in I$. Then we have
\begin{align}
-{\operatorname{ind}\nolimits}_{T(t)}(T_i[1])=(c_{ij}(t))_{j\in I}
\in{\mathbb{Z}}^I
\end{align}
for any $i\in I$, where we embed ${\mathbb{Z}}^I$ into
 ${\mathbb{Z}}^{\widetilde{I}}$ naturally.
\end{proposition}

\begin{remark} Proposition \ref{y and index} shows that
the {\em principal coefficients} of a quiver $Q$ are determined by the
{\em $g$-vectors\/} of the opposite quiver $Q^{\mathrm{op}}$. This can
also be deduced from Conjecture~1.6 (proved in Theorem~1.7)
of \cite{DWZ2}
 using Remark~7.15 of \cite{FZ3} on `Langlands duality'.
\end{remark}

The following analogue of \cite[Theorem 2.3]{DK} is
an important ingredient in our proof.

\begin{proposition}[{\cite[Proposition 3.1]{Pl2}}]\label{index preparation}
Let $X,Y\in{\operatorname{pr}\nolimits}\, T$ be rigid objects
 and $t\in{\mathbb{T}}_n$.
Then $X\simeq Y$ if and only if ${\operatorname{ind}\nolimits}_{T(t)}(X)
={\operatorname{ind}\nolimits}_{T(t)}(Y)$.
\end{proposition}

Let us introduce Caldero-Chapoton-type map.

\begin{definition}
Define a full subcategory
\begin{align}
{\mathcal{D}}:=\{M\in{\operatorname{pr}\nolimits}\,
 T\cap{\operatorname{pr}\nolimits}\,
 T[-1]\mid \dim_K{\operatorname{Hom}\nolimits}_{{\mathcal{C}}}(T,M[1])<\infty\}.
\end{align}
For any object $M\in{\mathcal{D}}$, we define an element
${\mathbb{X}}_M$ in
$\mathbb{Z}[x^{\pm1},y]$ by
\begin{align}
\widehat{y}_j&:=y_j\prod_{i\in I}x_i^{B_{ij}}\quad (j\in I),\\
{\mathbb{X}}_M&:=
\left(
\prod_{i\in I}x_i^{{\operatorname{ind}\nolimits}_T(M)_i}
y_i^{{\operatorname{ind}\nolimits}_T(M)_{i'}}
\right)
\sum_{e\in{\mathbb{Z}}^{\widetilde{I}}}
\chi({\operatorname{Gr}\nolimits}_e(
{\operatorname{Hom}\nolimits}_{{\mathcal{C}}}(T,M[1])))
\prod_{j\in I}\widehat{y}_j^{e_j},
\end{align}
where ${\operatorname{Gr}\nolimits}_e$ is the quiver Grassmannian
 and $\chi$ is the Euler characteristic.
\end{definition}

For each $t\in{\mathbb{T}}_n$ we have $T(t)\in{\mathcal{D}}$
 by Proposition \ref{KS}(1).
The following description of cluster variables
in $\mathcal{A}_{\bullet}(B,x,y)$  is crucial in our proof.

\begin{proposition}\label{CT and F}
Let $x_i(t)$ be a cluster variable in $\mathcal{A}_{\bullet}(B,x,y)$.
Then, we have
\begin{align}
x_i(t)={\mathbb{X}}_{T_i(t)}
\end{align}
for any $t\in{\mathbb{T}}_n$ and $i\in I$.
\end{proposition}

\begin{proof}
Specializing $(B,n)$ in \cite[Theorem 3.12]{Pl1} to
\begin{align}
(\left(\begin{array}{cc}
B&-E_n\\
E_n&O
\end{array}\right),2n),
\end{align}
we have that ${\mathbb{X}}$ is a cluster character
in the sense of \cite[Definition 3.10]{Pl1}.
Now the assertion is an immediate consequence.
\end{proof}

\subsection{Proof of Theorem \ref{thm:reduced}}

Now we are ready to prove Theorem \ref{thm:reduced}.

Since $y_j(t)=y_{\omega(j)}$ for any $j\in I$, we have
 in Proposition \ref{y and index}
\begin{align}
c_{ij}(t)=\begin{cases}
1&i=\omega(j)\\
0&\mbox{otherwise}
\end{cases}
\end{align}
for any $i,j\in I$. Thus
 ${\operatorname{ind}\nolimits}_{T(t)}(T_{\omega(j)}[1])
={\operatorname{ind}\nolimits}_{T(t)}(T_j(t)[1])$
for any $j\in I$.
By Proposition \ref{index preparation} we have 
\begin{eqnarray}\label{T omega}
T_{\omega(j)}\simeq T_j(t)
\end{eqnarray}
for any $j\in I$.
By Proposition \ref{CT and F} and \eqref{T omega}, we have
\begin{align}
x_j(t)={\mathbb{X}}_{T_j(t)}={\mathbb{X}}_{T_{\omega(j)}}=x_{\omega(j)}
\end{align}
for any $j\in I$.

Finally by \eqref{T omega} and Proposition \ref{quiver and object}
we have
\begin{align}
\omega^{-1}(\widetilde{Q})=\widetilde{Q}(t)
\ \mbox{ and so }\ \omega^{-1}(Q)=Q(t).
\end{align}
\qed

\subsection{Proof of periodicities of T and Y-systems}

Now the proof of   Theorems \ref{thm:Tperiod}
and \ref{thm:Yperiod} is at hand.

As corollaries of
Theorems \ref{thm:tYperiod} and \ref{thm:XXX}
we immediately obtain the periodicities of
cluster variables
and coefficients in $\mathcal{A}(B,x,y)$ with $B=B_{\ell}(B_r)$.

\begin{corollary}
\label{cor:Xperiod1}
For $\mathcal{A}(B,x,y)$ with $B=B_{\ell}(B_r)$,
the following relations hold.
\par
(i) Half periodicity: 
$x_{\mathbf{i}}(u+h^{\vee}+\ell)
=x_{\boldsymbol\omega(\mathbf{i})}(u)$.
\par
(ii) 
 Full periodicity: 
$x_{\mathbf{i}}(u+2(h^{\vee}+\ell))
=x_{\mathbf{i}}(u)$.
\end{corollary}

\begin{corollary}
\label{cor:Yperiod1}
For $\mathcal{G}(B,y)$ with $B=B_{\ell}(B_r)$,
the following relations hold.
\par
(i) Half periodicity: 
$y_{\mathbf{i}}(u+h^{\vee}+\ell)
=y_{\boldsymbol\omega(\mathbf{i})}(u)$.
\par
(ii) 
 Full periodicity: 
$y_{\mathbf{i}}(u+2(h^{\vee}+\ell))
=y_{\mathbf{i}}(u)$.
\end{corollary}

As further corollaries of Corollaries
\ref{cor:Xperiod1} and \ref{cor:Yperiod1} 
and Theorems \ref{thm:Tiso} and \ref{thm:Yiso}
we obtain Theorems \ref{thm:Tperiod} and \ref{thm:Yperiod}.

As a corollary of
Theorems \ref{thm:tYperiod} and \ref{thm:reduced}
we also obtain the periodicity of
{\em $F$-polynomials} \cite{FZ3}  (see Section 2.1 for the
definition), which will be used in the next section.

\begin{corollary}
\label{cor:Fperiod1}
For $\mathcal{A}(B,x,y)$ with $B=B_{\ell}(B_r)$,
let $F_{\mathbf{i}}(u)$ be the $F$-polynomial
at $(\mathbf{i},u)$.
Then, the following relations hold.
\par
(i) Half periodicity: 
$F_{\mathbf{i}}(u+h^{\vee}+\ell)
=F_{\boldsymbol\omega(\mathbf{i})}(u)$.
\par
(ii) 
 Full periodicity: 
$F_{\mathbf{i}}(u+2(h^{\vee}+\ell))
=F_{\mathbf{i}}(u)$.
\end{corollary}

\section{Dilogarithm identities}
\label{sec:dilog}

In this section we prove Theorem \ref{thm:DI2}.

In the cluster algebraic formulation here,
Theorem \ref{thm:DI2} is expressed as follows.

\begin{theorem}
\label{thm:DI3}
For  $\mathcal{G}_Y(B,y)$
with $B=B_{\ell}(B_r)$,
let $y^{(a)}_m(u)$ be the coefficient
tuple 
in Theorem \ref{thm:Yiso}.
Then, for  any
semifield homomorphism
   $\varphi: \mathbb{Q}_{\mathrm{sf}}(y)
\rightarrow \mathbb{R}_+$,
we have the identity
\begin{align}\label{eq:DI4}
\frac{6}{\pi^2}
\sum_{
(a,m,u)\in S'_+}
L\left(
\frac{\varphi(y^{(a)}_m(u))}{1+\varphi(y^{(a)}_m(u))}
\right)
&=
2r(2r\ell - 2r + 1),
\end{align}
where $S'_{+}=\{(a,m,u)\in \mathcal{I}'_{\ell+} \mid 
0\leq u < 2(h^{\vee}+\ell)\}$.
\end{theorem}

Let $F^{(a)}_m(u)$ denote the $F$-polynomial
$F_{\mathbf{i}}(v)$
at $(\mathbf{i},v)=g((a,m,u))$,
i.e., with the same parametrization by $
\mathcal{I}_{\ell+}$ as $x^{(a)}_m(u)$.

\begin{lemma}
\label{lem:F}
 (i) For $(a,m,u)\in \mathcal{I}'_{\ell+}$,
 the following relations hold.
\begin{align}
\label{eq:F1}
\begin{split}
\textstyle
F^{(a)}_{m}\left(u-\frac{1}{t_a}\right)
F^{(a)}_{m}\left(u+\frac{1}{t_a}\right)
&=
\left[
\frac{y^{(a)}_m(u)}{1+y^{(a)}_m(u)}
\right]_{\mathbf{T}}
\prod_{(b,k,v)\in \mathcal{I}_{\ell+}}
F^{(b)}_{k}(v)^{G(b,k,v;\,a,m,u)}\\
&\quad +
\left[
\frac{1}{1+y^{(a)}_m(u)}
\right]_{\mathbf{T}}
F^{(a)}_{m-1}(u)F^{(a)}_{m+1}(u),\\
\end{split}
\end{align}
\begin{align}
\label{eq:F2}
y^{(a)}_m(u)
&=
[y^{(a)}_m(u)]_{\mathbf{T}}
\frac{
\displaystyle
\prod_{(b,k,v)\in \mathcal{I}_{\ell+}}
F^{(b)}_{k}(v)^{G
(b,k,v;\, a,m,u)}
}
{
F^{(a)}_{m-1}(u)F^{(a)}_{m+1}(u)
},\\
\label{eq:F3}
1+y^{(a)}_m(u)
&=
[1+y^{(a)}_m(u)]_{\mathbf{T}}
\frac{
F^{(a)}_{m}\left(u-\frac{1}{t_a}\right)
F^{(a)}_{m}\left(u+\frac{1}{t_a}\right)
}
{
F^{(a)}_{m-1}(u)F^{(a)}_{m+1}(u)
}.
\end{align}
\par
(ii) Periodicity:
$
 F^{(a)}_m(u+2(h^{\vee}+\ell))= F^{(a)}_m(u).
$
\par
(iii)
 Each polynomial $F^{(a)}_m(u)$
has constant term 1.
\end{lemma}
\begin{proof}
(i).
\eqref{eq:F1}  is a specialization of \eqref{eq:clust}.
\eqref{eq:F2} is due to \cite[Proposition 3.13]{FZ3}.
\eqref{eq:F3} follows from \eqref{eq:F1} and \eqref{eq:F2}.
(ii). This is a special case of Corollary \ref{cor:Fperiod1}.
(iii). The claim is shown by induction
on $u$, by using  $F_{\mathbf{i}}(0)=1$,
\eqref{eq:F1}, and Proposition \ref{prop:levh}
(cf. \cite[Proposition 5.6]{FZ3}).
\end{proof}

According to \cite{FS,C,Nkn},
the proof of Theorem \ref{thm:DI2}
reduces to the next claim.

\begin{proposition}
\label{prop:wedge}
(i) In $\bigwedge^2 \mathbb{Q}_{\mathrm{sf}}(y)$,
we have
\begin{align}
\label{eq:wedge}
\sum_{(a,m,u)\in S'_+}
y^{(a)}_m(u)
\wedge
(1+y^{(a)}_m(u))
=0.
\end{align}
(ii) The total number of the negative monomials
among $[y^{(a)}_m(u)]_{\mathbf{T}}$
$((a,m,u)\in S'_+)$ is $2r(2r\ell-2r+1)$.
\end{proposition}

(ii) is already proved in Theorem \ref{thm:levhd}.
Let us prove (i).
It is parallel to the simply laced case \cite[Proposition 4.1]{Nkn},
but little more complicated.
Therefore, we present the calculations.

We put \eqref{eq:F2} and \eqref{eq:F3} into
\eqref{eq:wedge}, and expand it.

Firstly,
\begin{align}
\sum_{(a,m,u)\in S'_+}
[y^{(a)}_m(u)]_{\mathbf{T}}
\wedge
[1+y^{(a)}_m(u)]_{\mathbf{T}}
=0,
\end{align}
since each monomial
$[y^{(a)}_m(u)]_{\mathbf{T}}$ is either positive or negative
by Proposition \ref{prop:levh}.

Secondly, the contributions from the terms involving
only $F_{\mathbf{i}}(u)$'s vanish.
To see it, we separate them into two parts.
The first part
\begin{align}
\sum_{(a,m,u)\in S'_+}
\textstyle
F^{(a)}_{m-1}(u)
F^{(a)}_{m+1}(u)
\wedge
F^{(a)}_{m}(u-\frac{1}{t_a})
F^{(a)}_{m}(u+\frac{1}{t_a})
\end{align}
vanishes due to the symmetry argument of
\cite[Section 3]{CGT},
where we use
the periodicity of $F$-polynomials (Lemma \ref{lem:F} (ii)).
The second part
\begin{align}
\begin{split}
\sum_{(a,m,u)\in S'_+}
\displaystyle
\prod_{(b,k,v)\in \mathcal{I}_{\ell+}}
&F^{(b)}_{k}(v)^{G(b,k,v;\,
a,m,u)}
\wedge
\displaystyle
\frac{
F^{(a)}_{m}\left(u\textstyle-\frac{1}{t_a}\right)
F^{(a)}_{m}\left(u\textstyle+\frac{1}{t_a}\right)
}
{
F^{(a)}_{m-1}(u)F^{(a)}_{m+1}(u)
}
\end{split}
\end{align}
reduces, by the symmetry argument again, to
the sum consisting of the terms with $(a,b)=(r-1,r), (r,r-1)$;
namely,
\begin{align}
\begin{split}
&\sum_{m=1}^{\ell-1}
\sum_{
\genfrac{}{}{0pt}{1}
{
u\equiv 0\ \mathrm{mod}\ \mathbb{Z}
}
{
0\leq u < 2(h^{\vee}+\ell)
}
}
\displaystyle
F^{(r)}_{2m}(u)
\wedge
\displaystyle
\frac{
F^{(r-1)}_{m}\left(u-1\right)
F^{(r-1)}_{m}\left(u+1\right)
}
{
F^{(r-1)}_{m-1}(u)F^{(r-1)}_{m+1}(u)
}
\\
&
\quad
+
\sum_{m=1}^{\ell-1}
\sum_{
\genfrac{}{}{0pt}{1}
{
u\equiv \frac{1}{2}\ \mathrm{mod}\ \mathbb{Z}
}
{
0\leq u < 2(h^{\vee}+\ell)
}
}
\displaystyle
F^{(r-1)}_{m}(u-\textstyle\frac{1}{2})
F^{(r-1)}_{m}(u+\frac{1}{2})
\wedge
\displaystyle
\frac{
F^{(r)}_{2m}\left(u-\textstyle\frac{1}{2}\right)
F^{(r)}_{2m}\left(u+\frac{1}{2}\right)
}
{
F^{(r)}_{2m-1}(u)F^{(r)}_{2m+1}(u)
}
\\
&
\quad
+
\sum_{m=0}^{\ell-1}
\sum_{
\genfrac{}{}{0pt}{1}
{
u\equiv 0\ \mathrm{mod}\ \mathbb{Z}
}
{
0\leq u < 2(h^{\vee}+\ell)
}
}
\displaystyle
F^{(r-1)}_{m}(u)
F^{(r-1)}_{m+1}(u)
\wedge
\displaystyle
\frac{
F^{(r)}_{2m+1}\left(u-\textstyle\frac{1}{2}\right)
F^{(r)}_{2m+1}\left(u+\frac{1}{2}\right)
}
{
F^{(r)}_{2m}(u)F^{(r)}_{2m+2}(u)
},
\end{split}
\end{align}
where 
$F^{(r-1)}_{0}(u)=
F^{(r-1)}_{\ell}(u)=F^{(r)}_{0}(u)=F^{(r)}_{2\ell}(u)=1$.
It is easy to check that all the terms cancel  each other.

Thirdly, the contribution from
the remaining terms are as follows,
where $S_{+}=\{(a,m,u)\in \mathcal{I}_{\ell+} \mid 
0\leq u < 2(h^{\vee}+\ell)\}$:
\begin{align}
\begin{split}
&\sum_{(a,m,u)\in S'_+}
[y^{(a)}_m(u)]_{\mathbf{T}}
\wedge
F^{(a)}_m(u\textstyle-\frac{1}{t_a})
F^{(a)}_m(u\textstyle+\frac{1}{t_a})\\
&\qquad\qquad
=
\sum_{(a,m,u)\in S_+}
[y^{(a)}_m(u\textstyle-\frac{1}{t_a})]_{\mathbf{T}}
[y^{(a)}_m(u\textstyle+\frac{1}{t_a})]_{\mathbf{T}}
\wedge
F^{(a)}_m(u),
\end{split}
\\
\begin{split}
&-\sum_{(a,m,u)\in S'_+}
[y^{(a)}_m(u)]_{\mathbf{T}}
\wedge
F^{(a)}_{m-1}(u)
F^{(a)}_{m+1}(u)\\
&\qquad\qquad
=
-\sum_{(a,m,u)\in S_+}
[y^{(a)}_{m-1}(u)]_{\mathbf{T}}
[y^{(a)}_{m+1}(u)]_{\mathbf{T}}
\wedge
F^{(a)}_m(u),
\end{split}
\\
\begin{split}
&\sum_{(a,m,u)\in S'_+}
[1+y^{(a)}_m(u)]_{\mathbf{T}}
\wedge
F^{(a)}_{m-1}(u)
F^{(a)}_{m+1}(u)\\
&\qquad\qquad
=
\sum_{(a,m,u)\in S_+}
[1+y^{(a)}_{m-1}(u)]_{\mathbf{T}}
[1+y^{(a)}_{m+1}(u)]_{\mathbf{T}}
\wedge
F^{(a)}_m(u),
\end{split}
\\
\begin{split}
&-\sum_{(a,m,u)\in S'_+}
[1+y^{(a)}_m(u)]_{\mathbf{T}}
\wedge
\displaystyle
\prod_{(b,k,v)\in \mathcal{I}_{\ell+}}
F^{(b)}_{k}(v)^{G(b,k,v;\,
a,m,u)}
\\
&\qquad\qquad
=
-\sum_{(a,m,u)\in S_+}
\displaystyle
\prod_{(b,k,v)\in \mathcal{I}'_{\ell+}}
[1+y^{(b)}_{k}(v)]_{\mathbf{T}}^{G(a,m,u;\,
b,k,v)}
\wedge
F^{(a)}_m(u).
\end{split}
\end{align}
These terms cancel if we have the relation
\begin{align}
\begin{split}
&[y^{(a)}_m(u\textstyle-\frac{1}{t_a})]_{\mathbf{T}}
[y^{(a)}_m(u\textstyle +\frac{1}{t_a})]_{\mathbf{T}}
=
\frac{
\displaystyle
\prod_{(b,k,v)\in \mathcal{I}'_{\ell+}}
[1+y^{(b)}_{k}(v)]_{
\mathbf{T}}^{{}^t\!G(b,k,v;a,m,u)}
}
{
\displaystyle
[1+y^{(a)}_{m-1}(u)^{-1}]_{\mathbf{T}}
[1+y^{(a)}_{m+1}(u)^{-1}]_{\mathbf{T}}
}.
\end{split}
\end{align}
This is nothing but the Y-system \eqref{eq:Yu},
therefore, satisfied by Lemma \ref{lem:y2}.

This completes the proof of Proposition \ref{prop:wedge}.

\section{Alternative proof of periodicities
of T and Y-systems of simply laced type}

Let ($X_r$,$X'_{r'}$)  be a pair of simply laced
Dynkin diagrams of finite type
with index sets $I$ and $I'$.

As an application of Theorem \ref{thm:XXX},
we give an alternative and simplified
proof of the periodicities
of the T and Y-systems associated with ($X_r$,$X'_{r'}$).
They were formerly proved by \cite{FZ1,FZ2} for
$X'_{r'}=A_1$  (`level 2 case')
and \cite{Ke1,IIKNS,Ke2} for general case.

For a family of variables
 $\{T_{ii'} (u)\mid i\in I, i'\in I',
u\in \mathbb{Z}\}$,
the {\em T-system $\mathbb{T}(X_r,X'_{r'})$ associated with
a pair $(X_r,X'_{r'})$}
is a system of the  relations
\begin{align}
\label{eq:T2}
T_{ii'}(u-1)T_{ii'}(u+1)= 
\prod_{j:j\sim i} T_{ji'}(u)
+
\prod_{j':j'\sim i'} T_{ij'}(u),
\end{align}
where $j\sim i$ means $j$ is adjacent to $i$ in $X_r$,
while $j'\sim i'$ means $j'$ is adjacent to $i'$ in 
$X'_{r'}$.

For a family of variables
 $\{Y_{ii'} (u)\mid i\in I, i'\in I',
u\in \mathbb{Z}\}$,
the {\em Y-system $\mathbb{Y}(X_r,X'_{r'})$ associated with
a pair $(X_r,X'_{r'})$}
is a system of the  relations
\begin{align}
\label{eq:Y2}
Y_{ii'}(u-1)Y_{ii'}(u+1)= 
\frac{\displaystyle
\prod_{j:j\sim i} (1+Y_{ji'}(u))}
{\displaystyle
\prod_{j':j'\sim i'} (1+Y_{ij'}(u)^{-1})}.
\end{align}

Let
$C=(C_{ij})_{i,j\in I}$ and $C'=(C_{i'j'})_{i',j'\in I}$
be a pair of Cartan matrices of types
$X_r$ and $X'_{r'}$  with fixed
bipartite decompositions
$I=I_+\sqcup I_-$ and $I'=I'_+\sqcup I'_-$.
Set $\mathbf{I}=I\times I'$.
For $\mathbf{i}=(i,i')\in \mathbf{I}$,
let us write $\mathbf{i}:(++)$ if $(i,i')\in I_+\times I'_+$, {\em etc}.
Define the matrix $B=B(X_r,X'_{r'})=
(B_{\mathbf{i}\mathbf{j}})_{\mathbf{i},\mathbf{j}
\in \mathbf{I}}$ by
\begin{align}
\label{eq:Bsq}
B_{\mathbf{i}\mathbf{j}}=
\begin{cases}
-C_{ij}\delta_{i'j'}
 &
 \mathbf{i}:(-+), \mathbf{j}:(++)
\ \mbox{or}\
 \mathbf{i}:(+-), \mathbf{j}:(--),
\\
C_{ij}\delta_{i'j'}
 &
 \mathbf{i}:(++), \mathbf{j}:(-+)
\ \mbox{or}\
 \mathbf{i}:(--), \mathbf{j}:(+-),
\\
-\delta_{ij}C'_{i'j'}
 &
 \mathbf{i}:(++), \mathbf{j}:(+-)
\ \mbox{or}\
 \mathbf{i}:(--), \mathbf{j}:(-+),
\\
\delta_{ij}C'_{i'j'}
 &
 \mathbf{i}:(+-), \mathbf{j}:(++)
\ \mbox{or}\
 \mathbf{i}:(-+), \mathbf{j}:(--),
\\
0 & \mbox{otherwise}.
\end{cases}
\end{align}
Then, as in Section 2, one can formulate the T and Y-systems
in terms of the cluster algebra
$\mathcal{A}(B,x,y)$ and its coefficient group
$\mathcal{G}(B,y)$
with $B=B(X_r,X'_{r'})$ (cf. \cite[Proposition 2.6]{Nkn}).
 
\begin{theorem}
\label{thm:tYperiod2}
The following relations hold
for the tropical Y-system
of $\mathcal{G}(B,y)$ with $B=B(X_r,X'_{r'})$.
\par
(i) Half periodicity: 
$[y_{\mathbf{i}}(u+h+h')]_{\mathbf{T}}
=[y_{\boldsymbol\omega(\mathbf{i})}(u)]_{\mathbf{T}}$.
\par
(ii) 
 Full periodicity: 
$[y_{\mathbf{i}}(u+2(h+h')]_{\mathbf{T}}
=[y_{\mathbf{i}}(u)]_{\mathbf{T}}$.
\par
Here, $h$ and $h'$ are the Coxeter numbers of $X_r$ and
$X'_{r'}$,
and $\boldsymbol{\omega}=\omega\times \omega'$,
where $\omega$ (resp.\ $\omega'$) is the Dynkin automorphism
of $X_r$ (resp. $X'_{r'}$) for types $A_r$, $D_{r}$ ($r:odd$), or $E_6$,
and the identity otherwise.
\end{theorem}
\begin{proof}
This is an immediate consequence of the factorization property
of the tropical Y-system studied in \cite[Proposition 3.2]{Nkn}.
\end{proof}

As a corollary of Theorem \ref{thm:XXX} and Theorem \ref{thm:tYperiod2},
we obtain the periodicities of the T and Y-systems.

\begin{corollary}
\label{cor:tYperiod2}
 The following relations hold.
\par
(i) Half periodicity: 
$T_{\mathbf{i}}(u+h+h')=
T_{\boldsymbol\omega(\mathbf{i})}(u)$,
$Y_{\mathbf{i}}(u+h+h')=
Y_{\boldsymbol\omega(\mathbf{i})}(u)$.
\par
(ii) 
 Full periodicity: 
$T_{\mathbf{i}}(u+2(h+h'))=
T_{\mathbf{i}}(u)$,
$Y_{\mathbf{i}}(u+2(h+h'))=
Y_{\mathbf{i}}(u)$.
\end{corollary}

\end{document}